\documentclass[3p]{elsarticle}
\usepackage[utf8]{inputenc}
\usepackage{amsmath,amssymb,amsfonts,amsthm}

\usepackage{pgfplots}
\pgfplotsset{compat=newest}
\pgfplotsset{plot coordinates/math parser=false}

\usepackage{hyperref}
\usepackage[capitalize,nameinlink]{cleveref}
\usepackage{dsfont}
\usepackage{todonotes}
\usepackage{multicol}
\usepackage{enumitem}
\usepackage{graphicx} 
\usepackage{mathtools}
\graphicspath{ {images/} }
\allowdisplaybreaks

\newcommand{\e}{\mathrm{e}}

\newcommand{\R}{\mathbb R}
\newcommand{\N}{\mathbb N}

\newcommand{\eps}{\epsilon}

\renewcommand{\phi}{\varphi}

\newcommand{\mEF}{\mathcal{EF}}

\newcommand{\OT}{{\Omega_{T}}}

\newcommand{\loc}{\textnormal{loc}}

\newcommand{\dd}{\ensuremath{\,\mathrm{d}}}

\DeclareMathOperator{\sgn}{sgn}
\DeclareMathOperator*{\esssup}{ess-\sup}
\DeclareMathOperator{\supp}{supp}
\DeclareMathOperator*{\essinf}{ess\ inf}

\renewcommand{\epsilon}{\varepsilon}

\newtheorem{theo}{Theorem}[section]
\newtheorem{defi}[theo]{Definition}

\newtheorem{cor}[theo]{Corollary}
\newtheorem{lem}[theo]{Lemma}

\newtheorem{ass}[theo]{Assumption}
\newtheorem{rem}[theo]{Remark}
\usepackage{verbatim}
\crefname{theo}{Thm.}{Theorems}
\crefname{ass}{Asm.}{Assumptions}
\crefname{defi}{Def.}{Definitions}
\crefname{rem}{Rmk.}{Remarks}
\crefname{prop}{Prop.}{Propositions}
\crefname{lem}{Lem.}{Lemmata}

\usepackage{xcolor}
\usepackage{pifont}

\newenvironment{talign*}
 {\let\displaystyle\textstyle\csname align*\endcsname}
 {\endalign}

 \newenvironment{tgather*}
 {\let\displaystyle\textstyle\csname gather*\endcsname}
 {\endgather}

\title{Nonlocal conservation laws with \texorpdfstring{\(p\)}{p}-norm, the singular limit problem and applications to traffic flow}
\usepackage{lineno}

\author[<1>]{Felisia Angela Chiarello}

\ead{<felisiaangela.chiarello@univaq.it>}

\affiliation[<1>]{organization={University of L'Aquila, Department of Engineering and Information Science and Mathematics,
(DISIM)},
            addressline={Via Vetoio, Ed. Coppito 1}, 
            city={L'Aquila},
            postcode={67100}, 
            country={Italy}}

\author[<2>]{Alexander Keimer}

\affiliation[<2>]{organization={University of Rostock, Institute of Mathematics},
            addressline={Ulmenstraße 69}, 
            city={Rostock},postcode={18057}, 
            country={Germany}}

 \ead{<alexander.keimer@fau.de>}   

 \author[<3>]{Lukas Pflug}

\affiliation[<3>]{organization={Friedrich-Alexander-Universität Erlangen-Nürnberg (FAU), Department of Mathematics},
            addressline={Cauerstr. 11}, 
            city={Erlangen},postcode={91058}, 
            country={Germany}}
 \ead{<lukas.pflug@fau.de>}

\begin{document}
\begin{abstract}
In this contribution, we study scalar nonlocal conservation laws with the \(p\)-norm. Here, ``nonlocal'' means that the velocity of the conservation law depends on an integral term in space. Typically, the nonlocal term consists of integrating the solution in \(L^{1}\), whereas here we will study the case when the solution is integrated in the \(L^{p}\)-norm. We consider even the case of the \(L^{p}\) metric when \(p\in (0,1)\) and establish, for an initial datum which is uniformly bounded away from zero, the existence and uniqueness of weak solutions. We then demonstrate that there are also solutions to the initial datum being zero under more restrictive assumptions. Furthermore, we investigate the singular limit, i.e., what happens when the nonlocal kernel converges to a Dirac distribution. Indeed, for the one-sided exponential kernel, we recover the (entropy) solution of the corresponding local conservation law for all \(p\in(0,\infty)\) with further restrictions for \(p\in(0,1)\). This generalizes the celebrated singular limit result for nonlocal conservation laws for \(p=1\) significantly and showcases the robustness of the approximation of local conservation laws by nonlocal ones.

We investigate also the monotonicity of the solution when assuming that the initial datum is monotone. Finally, we prove the convergence of solutions for \(p\rightarrow 0\) on a small time horizon, resulting in a different kind of nonlocal conservation law.
Numerical studies showcasing the effect of \(p\) on the singular limit convergence and more conclude the contribution.
\end{abstract}
\begin{keyword}
    nonlocal conservation laws\sep \(p\)-norm\sep singular limit\sep traffic flow \sep entropy solution \sep phantom shocks
    \MSC[2020]{35L03, 35L65, 76A30}
\end{keyword}
\maketitle

\section{Introduction}
 Nonlocal conservation laws and nonlocal systems of conservation laws have been widely analyzed because they are suitable for describing various phenomena arising in many fields like sedimentation \cite{betancourt, Burger2023}, conveyor belts \cite{rossi2020well}, granular flows \cite{amadori}, crowd dynamics \cite{ColomboMercier2012}, supply chains \cite{colombo}, gradient constraint \cite{amorim}, viscoelasticity \cite{ChenChristoforou2007}, laser technology \cite{ColomboMarcellini2015}, biology \cite{ColomboRossi2015}, and more. 

Regarding the application to the traffic flow setting, we can enumerate nonlocal versions of the classical LWR model (after Lighthill, Whitham and Richards) \cite{ lighthill1955kinematic,lwr_2} and ARZ model \cite{aw2000resurrection}, in \cite{blandin2016well, chiarello, chiarello2020micro, Ciaramaglia2025}. 
Moreover, there are also some nonlocal models on junctions and networks, like in \cite{chiarello2019non,Friedrich2022,Gugat2015}, and nonlocal multi-class and multi-lane systems that describe heterogeneous traffic flow; see \cite{ciaramagliamulti-class2025, Friedrich2021,KeimerMultilane2022}.
In particular, the speed is assumed to depend on a weighted mean of the downstream traffic density, becoming a Lipschitz continuous function with respect to space and time variables, ensuring bounded acceleration, i.e.,  not allowing for speed discontinuities unlike classical first-order macroscopic traffic models.
In \cite{kloeden}, the authors use $L^p$-valued Lebesgue measurable density functions and signed Radon measures to model the flow characteristics of different types of vehicles, such as cars and trucks.
General existence and uniqueness results for nonlocal equations for scalar equations in one spatial dimension can be found in \cite{AmorimColomboTexeira,scialanga,Keimer2023}, in \cite{colombo, ColomboGaravelloMercier2012, spinola} for multi-dimensional scalar equations, and in \cite{aggarwal, ColomboRossiIBVP, Crippa2012} for the multi-dimensional system case.
To prove the well-posedness of these nonlocal models, there are various approaches: one approach provides suitable compactness estimates on a sequence of approximate solutions constructed through finite volume schemes, as in \cite{aggarwal, AmorimColomboTexeira, chiarello,  friedrich2018godunov}; another relies on characteristics and the use of fixed points as proposed in \cite{wang,pflug}. It is important to underline that, when using the fixed-point approach, the uniqueness of weak solutions is obtained without prescribing any kind of entropy condition, thus rendering entropy conditions obsolete.

A very challenging analytical problem regarding nonlocal conservation laws is that of the singular local limit: In the one-dimensional scalar case, a parameter $\eta>0$ related to the support of the kernel is fixed and the re-scaled kernel function is considered
\begin{equation*}
    \gamma_\eta\equiv\tfrac{1}{\eta}\gamma\left(\tfrac{\cdot}{\eta}\right).
\end{equation*}
Owing to the assumption $\int_{\R} \gamma_\eta(x)\dd x=1,$ when $\eta\to 0^+$ the family $\gamma_\eta$ converges weakly$^*$ in the sense of measures to the Dirac distribution, and we formally obtain the corresponding local conservation law:
\begin{equation*}
 \textbf{nonlocal: }  \begin{cases} \partial_t q_\eta+ \partial_x\big(q_\eta V(q_\eta\ast\gamma_\eta)\big)=0,\\ q_\eta(0,x)= q_{0}(x), \end{cases} \overset{\eta\rightarrow 0}{\longrightarrow}\  \textbf{ Local: }\begin{cases} \partial_t q+ \partial_x\big(q V(q)\big)=0,\\ q(0,x)=q_{0}(x),\end{cases}
\end{equation*}
with $V:\R \to \R$ being a Lipschitz continuous function.
The above derivation is merely formal and has to be rigorously justified. This means that one needs to show that the solution $q_\eta$ of the nonlocal Cauchy problem converges to the entropy-admissible solution of the corresponding local Cauchy problem in some suitable topology. The rigorous derivation of this limit was
initially posed in \cite{AmorimColomboTexeira} for the one-dimensional scalar case, motivated by numerical evidence, and later numerically corroborated in \cite{blandin2016well}.
In general, the solution of the nonlocal Cauchy problem does not converge to the solution of the local one and \cite{ColomboCrippaMarconiSpinolo2021} shows three counterexamples in this sense, while, in \cite{coclite2021singular, Graff}, the role of (numerical) viscosity of such nonlocal-to-local limit is highlighted, enabling passing to the limit on a ``diagonal sequence'' in both the viscosity approximation and the nonlocal kernel.

The first positive analytical answer was formulated in \cite{zumbrun}, under the assumptions of an even kernel function and a smooth compactly supported initial datum on a significantly small time horizon. 
Moreover, in \cite{pflug4} the authors proved that the solution of one-dimensional scalar nonlocal conservation laws converges to the entropy solution of the corresponding local conservation laws assuming monotone initial data. In \cite{ColomboCrippaMarconiSpinolo2021}, the authors show then that, for an initial datum which is one-sided Lipschitz continuous (OSL), such an OSL condition holds uniformly in the nonlocal kernel, so that passing to the limit and recovering the local entropy solution follow.

These results were then complemented in the one-dimensional scalar case and one-sided kernels in \cite{bressan2019traffic, bressan2021entropy, coclite2022general} for exponential kernel functions; \cite{Marconi2023,Colombo2024} generalized to convex kernels and other not necessarily fully monotonic kernels in \cite{keimerpflug23fixed}. Regarding nonlocal systems with a specific structure, applied to lane-changing traffic dynamics, the singular limit has been studied and proved to converge in \cite{chiarello2023system}.

A recent preprint \cite{Coclite2025singularlimit} approaches the singular limit problem via compensated compactness \cite{tartar1979compensated,Tartar1983} for a broad class of nonlocal kernels (particularly the constant kernel), avoiding otherwise required \(TV\) bounds on the initial datum. While it might be possible to use that also in this manuscript, our focus (as detailed below) is more on the \(p-\)norm, so we leave this potential generalization to the future.
Also fitting to the topic considered are \cite{Huang2024,DeNitti2025asymptotically} where the authors give first, for an initial datum with one-sided Lipschitz condition, and later, in the general case, numerical approximations which behave reasonably in the singular limit.

The aim of this paper is to investigate a new class of nonlocal scalar conservation laws with a flux function depending on the \(L^p-\)norm of a downstream convolution term between a general monotone kernel function and the solution/density as stated in \cref{defi:model_class}. One can apply this kind of equation in traffic flow modelling where the nonlocal operator reproduces, for larger \(p\), a more sensitive reaction on the traffic density ahead. The model class thus represents a generalization of the already existing nonlocal dynamics by replacing the \(L^{1}\) integral by a more general \(p-\)norm integral (including the special case \(p<1\)).
We investigate the well-posedness of this type of conservation law, point out the difference to the \(p=1\) case, and study the singular limit problem as well as the limit as $p\to0$. For all the named cases, we prove convergence and generalize the results in literature to `general' \(p\in (0,\infty)\).

The paper is organized as follows. In \cref{sec:existence_uniqueness}, we introduce the problem and prove the well-posedness starting from an initial datum bounded away from zero.  The singular limit problem, i.e., the convergence of the solutions when the kernel function in the nonlocal operator converges to a Dirac distribution is proven in \cref{sec:singular_limit}. We do this for the exponential kernel, although it might be generalizable to a much broader class of kernels following the lines of \cite{Marconi2023,Colombo2024,keimerpflug23fixed,Coclite2025singularlimit}. We also look into Oleinik estimates \cite{oleinik,oleinik_english}, and manage to obtain those under more restrictive assumptions in \cref{subsec:Oleinik}, similar to approaches taken for \(p=1\) in \cite{coclite2023oleinik}. In \cref{sec:existence_zero_initial_datum}, it is then shown that solutions can also be constructed for an initial datum which can be zero; some comments on a potential singular limit argument are made.
\Cref{sec:monotonicity} analyzes the monotonicity-preserving behavior of the $p-$norm nonlocal conservation law; in \cref{sec:p=0}, we prove the convergence of solutions for $p\to0$ to another non-canonical nonlocal conservation law. \Cref{sec:numerics} numerically illustrates the findings of existence and uniqueness of solutions as well as the singular limit problem. The paper is concluded in \cref{sec:conclusions} with some open problems worth pursuing in the future.

\section{Model class and existence and uniqueness of solutions to the nonlocal conservation law}\label{sec:existence_uniqueness}

We start by stating the model class under consideration as well as the required assumptions.
\begin{defi}[The model class considered]\label{defi:model_class}
For \(T\in\R_{>0}\) and \(p\in(0,\infty)\) we consider for the (density) function \(q:(0,T)\times\R\rightarrow\R_{\geq0}\) the following Cauchy problem
\begin{align}\label{eq:pnorm_problem}
q_{t}(t,x)+\partial_{x}\Big(V(W_{p, \eta}[q,\gamma](t,x))q(t,x)\Big)&=0,&& (t,x)\in\OT,\\
q(0,x)      &=q_{0}(x),      && x\in\R,
\intertext{with the nonlocal term in \(p\)}
W_{p,\eta}[q,\gamma](t,x)&=\left(\tfrac{1}{\eta}\int_{x}^{\infty}\Big(\gamma\big(\tfrac{y-x}{\eta}\big)q(t,y)\Big)^{p}\dd y\right)^{\frac{1}{p}},&&(t,x)\in\OT,\label{eq:nonlocal_operator_2}
\end{align}
\(q_{0}:\R\rightarrow\R_{\geq0}\) the \textbf{initial density}, \(W_{p,\eta}\) the \textbf{nonlocal} term, \(V:\R\rightarrow\R\) the \textbf{velocity}, \(\gamma:\R\rightarrow\R_{\geq0}\) the \textbf{nonlocal weight or kernel} and \(\eta\in\R_{>0}\) the \textbf{look-ahead parameter or nonlocal reach}.
\end{defi}
\begin{rem}[Kernel with \(p-\)th power]
One could consider, as the nonlocal operator, also
\begin{equation}
\left(\tfrac{1}{\eta}\int_{x}^{\infty}\gamma\big(\tfrac{y-x}{\eta}\big)q(t,y)^{p}\dd y\right)^{\frac{1}{p}},\ (t,x)\in\OT,\label{eq:weight_outside}
\end{equation}
but this would not really take the \(p-\)norm of some weighted down-stream density. This comes particularly clear when letting \(p\rightarrow\infty\) as the nonlocal weight would then collapse to a constant function. However, when not considering any limiting behavior like \(\eta\rightarrow 0\) as in \cref{sec:singular_limit} or \(p\rightarrow \{0,\infty\}\) in \cref{sec:p=0}, one can rescale the kernel so that all results concerning existence and uniqueness and maximum principles also apply in the case of \cref{eq:weight_outside}.
\end{rem}
To prove well-posedness, that is, existence and uniqueness for the Cauchy problem introduced in \cref{defi:model_class} on any finite time horizon, we require the following assumptions.
\begin{ass}[Assumption on the datum and the velocity function \(V\)]\label{ass:input_datum_and_velocity}
We assume that the following holds:

\begin{description}
    \item[Initial datum:] \(q_{0}\in L^{\infty}(\R;\R_{>q_{\min}})\cap TV(\R),\ q_{\min}\in\R_{>0}\)
    \item[Velocity function:] \(V\in W^{2,\infty}_{\loc}(\R;\R), \ V'\leqq0\text{ on }\R\)
    \item[Nonlocal kernel:] \(\gamma\in W^{2,1}(\R_{>0};\R_{\geq 0})\), \(\gamma\) monotonically decreasing, \(\|\gamma\|_{L^{p}(\R_{>0})}=1\)
    \item[Nonlocal reach, integrability:] \(\eta\in\R_{>0}\), \(p\in(0,\infty)\)
\end{description}
\end{ass}
\begin{rem}[The meaning of \(q_{\min}\), \(V\), \(\gamma\)]
Assuming that the initial datum is bounded away from zero by \(q_{\min}\) in \cref{ass:input_datum_and_velocity} is crucial for the proof of existence and uniqueness of solutions. Thanks to a later-derived maximum principle in \cref{theo:existence_uniqueness_maximum_principle} the solution can never become smaller than \(q_{\min}\) so that the nonlocal operator, i.e., the \(p-\)norm integral, remains Lipschitz in space and is thus Lipschitz continuous, a key property when using the characteristics and applying Banach's fixed-point theorem. 

The requirements on the velocity \(V\) and the kernel \(\gamma\) are imposed for obtaining a maximum and minimum principle. From a traffic perspective, a decreasing velocity is quite natural, and the kernel being monotonically decreasing means that traffic density closer to the current position is taken into account more than traffic further away.
\end{rem}
Next, we will state the definition of a weak solution, which as given here, is quite classical:
\begin{defi}[Weak solution for \cref{defi:model_class}]\label{defi:weak_solution}
  For \(T\in\R_{>0}\), we call the function \(q\in C\big([0,T];L^{1}_{\text{loc}}(\R)\big)\) a weak solution to the Cauchy problem in \cref{defi:model_class} given \cref{ass:input_datum_and_velocity} iff \(\forall\phi\in C^{1}_{\text{c}}((-42,T)\times\R)\) it holds that
  \begin{gather*}
      \iint_{\OT}\Big(\phi_{t}(t,x)+V(W_{p,\eta}[q,\gamma](t,x))\phi_{x}(t,x)\Big)q(t,x)\dd x\dd t+\int_{\R} q_{0}(x)\phi(0,x)\dd x=0
  \end{gather*}
  with \(W_{p,\eta}[q,\gamma]\) being the nonlocal operator as in \cref{eq:nonlocal_operator_2}.
\end{defi}
Having stated the required assumptions, we aim now to prove the existence and uniqueness of solutions. This type of proof has already been demonstrated for different but similar nonlocal conservation laws (see \cite{pflug,Friedrich2024,Keimer2023}) in the literature, so we only sketch it here. The approach goes back to \cite{wang,pflug}, replacing the nonlocal term with a time- and space-dependent velocity field $w$, constructing a solution of the corresponding linear conservation law through the method of characteristics and then deducing a fixed-point equation for the nonlocal term.
\begin{theo}[Existence, uniqueness, maximum principle and continuity w.r.t.\ initial datum]\label{theo:existence_uniqueness_maximum_principle}
Let \cref{ass:input_datum_and_velocity} and \(T\in\R_{>0}\) hold. Then, the considered Cauchy problem in \cref{defi:model_class} admits a unique weak solution in the sense of \cref{defi:weak_solution} \[q\in C\big([0,T];L^{1}_{\loc}(\R)\big)\cap L^{\infty}((0,T);L^{\infty}(\R)\cap TV(\R))\] and the maximum principle is satisfied, i.e., it holds that
\begin{equation}
    (t,x)\in(0,T)\times\R \text{ a.e: } q_{\min}\leq \essinf_{\tilde{x}\in\R} q_{0}(\tilde{x})\leq q(t,x)\leq \|q_{0}\|_{L^{\infty}(\R)}.\label{eq:min_max}
\end{equation}
Moreover, choosing a mollifier of the initial datum and calling it \(\tilde{q}_{0}\in TV(\R)\cap L^{\infty}(\R)\cap C^{1}(\R)\) and the corresponding solution \(\tilde{q}\in C\big([0,T];L^{1}_{\text{loc}}(\R)\big)\), we have the following stability result:
\[
\forall \eps\in\R_{>0} \ \exists \delta\in\R_{>0}:\ \|q_{0}-\tilde{q}_{0}\|_{L^{1}(\R)}\leq \delta\implies \|q-\tilde{q}\|_{C([0,T];L^{1}(\R))}\leq \eps,
\]
and \(\tilde{q}\in C^{1}(\OT),\) i.e.,\ \(\tilde{q}\) is a classical solution.
\end{theo}
\begin{proof}
As in \cite{pflug,Keimer2023} the idea is to consider a fixed-point mapping in the nonlocal term \(w\coloneqq W_{p}[q,\gamma_\eta]\). Assume that this term is given and  belongs to \(L^{\infty}\big((0,T^{*});W^{1,\infty}_{\text{loc}}(\R)\big)\) such that \(\partial_{2}W_{p}[q,\gamma_\eta]\in L^{\infty}((0,T);L^{\infty}(\R))\). Then, we have a Lipschitz continuous velocity field and can write the solution to the conservation law in terms of characteristics (see \cite{pflug})  as
\begin{equation}
q(t,x)=q_{0}\big(\xi_{w}(t,x;0)\big)\,\partial_{2}\xi_{w}(t,x;0),\ (t,x)\in\OT,\label{eq:solution_formula}
\end{equation}
where \(\xi:\OT\times(0,T)\) satisfies the integral equality (characteristics in integral form)
\begin{equation}
\xi_{w}(t,x;\tau)=x+\int_{t}^{\tau}V\left(w\big(s,\xi_{w}(t,x;s)\big)\right)\dd s,\  (t,x,\tau)\in\OT\times (0,T).\label{eq:characteristics}
\end{equation}
Plugging the solution formula in \cref{eq:solution_formula} into the nonlocal operator for \((t,x)\in\OT\) leads to
\begin{align}
w(t,x)&=\left(\tfrac{1}{\eta}\int_{x}^{\infty}\Big(\gamma\big(\tfrac{y-x}{\eta}\big)q(t,y)\Big)^{p}\dd y\right)^{\frac{1}{p}}=\left(\tfrac{1}{\eta}\int_{x}^{\infty}\Big(\gamma\big(\tfrac{y-x}{\eta}\big)q_{0}(\xi_{w}(t,y;0))\,\partial_{2}\xi_{w}(t,y;0)\Big)^{p}\,\dd y\right)^{\frac{1}{p}},\label{eq:fixed_point_1}
\intertext{and substituting \(z=\xi_{w}(t,y;0)\)}
&=\left(\tfrac{1}{\eta}\int_{\xi_{w}(t,x;0)}^{\infty}\Big(\gamma\big(\tfrac{\xi_{w}(0,z;t)-x}{\eta}\big)q_{0}(z)\,\partial_{2}\xi_{w}(t,\xi_{w}(0,z;t);0)\Big)^{p}\partial_{2}\xi_{w}(0,z;t)\dd z\right)^{\frac{1}{p}},
\intertext{which simplifies further because \(\partial_{2}\xi_{w}(t,\xi_{w}(0,z;t);0)=\tfrac{1}{\xi_{w}(0,z;t)} \forall (t,z)\in\OT\):}
&=\left(\tfrac{1}{\eta}\int_{\xi_{w}(t,x;0)}^{\infty}\Big(\gamma\big(\tfrac{\xi_{w}(0,z;t)-x}{\eta}\big)q_{0}(z)\Big)^{p}\partial_{2}\xi_{w}(0,z;t)^{1-p}\dd z\right)^{\frac{1}{p}}.\label{eq:fixed_point_2}
\end{align}
This is a fixed-point equation in the nonlocal term, as the right-hand side is a function of \(w\). In the next steps, we will analyze the dependency of the characteristics on the nonlocal term \(w\) to finally prove the existence and uniqueness of a nonlocal term \(w\) in the previously mentioned topology with additional \(TV\) regularity on the spatial derivative of \(w\). 
This is done through Banach's fixed-point theorem.  
We can show the self-mapping property and the contraction property in the topology \(L^{\infty}(L^{\infty})\) achieved on a significantly small time horizon \(T^{*}\in(0,T]\). Then, we will be able to prove the stability result, which is just a revisiting of our fixed-point mapping. In this way, we can approximate the given solution smoothly on a significantly small time horizon \([0,T^{*}]\). At this point, we can prove the maximum principle on the smoothed solution \(\tilde{q}\) on \([0,T^{*}]\)  as follows (for details, refer to \cite{Friedrich2024,friedrich2022conservation}): Let us consider, for a given time \(t\in[0,T^{*}]\) an \(\tilde{x}\in\R:\ \tilde{q}(t,\tilde{x})=\|\tilde{q}(t,\cdot)\|_{L^{\infty}(\R)}\), i.e., a point where the solution is maximal, meaning that \(\partial_{2}q(t,\tilde{x})=0\). Then, one can compute
\begin{align*}
    \partial_{t}\tilde{q}(t,x)&=-\partial_{x}\big(V(W_{p,\eta}[\tilde{q},\gamma,](t,x))\tilde{q}(t,x)\big)\\
    &=-V'(W_{p,\eta}[\gamma,\tilde{q}](t,x))\,\partial_{2}W_{p,\eta}[\tilde{q},\gamma](t,x)\tilde{q}(t,x)+ V\big(W_{p, \eta}[\tilde{q},\gamma](t,x)\big)\,\partial_{2}\tilde{q}(t,x)
    \intertext{which, evaluated at \(x=\tilde{x}\),}
    &=-V'\big(W_{p, \eta}[\tilde{q},\gamma](t,\tilde{x})\big)\ \tilde{q}(t,\tilde{x})\ \partial_{2}W_{p, \eta}[\tilde{q},\gamma](t,\tilde{x}).
\end{align*}
As the first two terms together with the minus sign are positive, we look at the third term and compute its derivative as follows:
\begin{align}
\partial_{2}W_{p,\eta}[\tilde{q},\gamma](t,\tilde{x})=\tfrac{1}{p}W_{p,\eta}[\tilde{q},\gamma](t,\tilde{x})^{1-p}\bigg(-\tfrac{p}{\eta^{2}}\int_{\tilde{x}}^{\infty}\gamma(\tfrac{y-\tilde{x}}{\eta})^{p-1}\gamma'(\tfrac{y-\tilde{x}}{\eta})\tilde{q}(t,y)^{p}\dd y-\tfrac{1}{\eta}\gamma(0)^{p}\tilde{q}(t,\tilde{x})^{p} \bigg)\label{eq:partial_2_W_reformulations}
\intertext{in which the first factor is positive, and \(\gamma'\leqq 0\), so we can estimate at the spatial point \(\tilde{x}\) where the solution is maximal as follows:}
\leq \tfrac{1}{p}W_{p,\eta}[\tilde{q},\gamma](t,\tilde{x})^{1-p}\bigg(-\tilde{q}(t,\tilde{x})^{p}\tfrac{p}{\eta^{2}}\int_{\tilde{x}}^{\infty}\gamma(\tfrac{y-\tilde{x}}{\eta})^{p-1}\gamma'(\tfrac{y-\tilde{x}}{\eta})\dd y-\tfrac{1}{\eta}\gamma(0)^{p}\tilde{q}(t,\tilde{x})^{p} \bigg)\notag
&=0.
\end{align}
This means that the solution cannot increase when it is at its maximum point (for the case where the supremum is attended at \(\pm\infty\) we refer the reader to \cite{friedrich2022conservation}).
Analogously, one can perform estimates for the minimum and show that the minimum can only increase. Altogether, this argument yields uniform constants on the smoothed solution for \(t\in[0,T^{*}]\) and thus on the original solution as well. By a time-clustering argument and recalling that the solution can be evaluated at \(t=T^{*}\) thanks to the specified regularity, we can cover any finite time horizon with our fixed-point approach and show the existence and the stability of solutions on such time horizon.

The uniqueness of the weak solution is analogous to what has been proven in \cite{pflug} and comes from the fact that the fixed-point mapping gives us a unique solution for our nonlocal term together with a Lipschitz-continuous velocity field,  and we can show that any weak solution can be posed in terms of characteristics.
\end{proof}

 \section{The singular limit problem} \label{sec:singular_limit}
In this section, we show the limit convergence of the nonlocal solution to the unique entropy solution of the corresponding local problem when the integral kernel converges to a Dirac distribution. To simplify the analysis, we focus merely on the exponential kernel, though it might be possible to generalize to convex kernels as was done in \cite{Marconi2023} or to kernels with finite support as in \cite{keimerpflug23fixed}.
We require \cref{lem:approx}, as follows, on approximating functions in \(L^{1}(\R)\) when assuming \(TV\) bounds:
\begin{lem}[Approximation of bounded functions with bounded total variation on  \(\R\)] \label{lem:approx}
Let \(q\in L^{\infty}\cap TV(\R)\). Then, \(q\) can be approximated by a smooth function in \(L^1(\R)\), i.e.,
\begin{equation*}
    \forall \varepsilon \in \R_{>0}\, \exists \psi_{\varepsilon}\in C^{\infty}(\R): \|\psi_\varepsilon-q\|_{L^1(\R)}\leq \varepsilon.
\end{equation*}
\end{lem}
\begin{proof}
The proof is along the lines of \cite[Lemma 4.15]{pflug4} and consists of using, as an approxmation, a standard mollifier as defined in~\cite{leoni}.
\end{proof}
We will show convergence by means of a \(TV\) estimate. 
As is well known for nonlocal conservation laws, the \(TV\) estimates on the nonlocal term seem to give generally stronger results than \(TV\) estimates on the solution (which might not even exist \cite{ColomboCrippaMarconiSpinolo2021} uniformly with respect to the integral kernel) or some Oleinik estimates, as the latter require primarily points-wise estimations on the solution which is more involved for nonlocal dynamics (compare \cite{coclite2023oleinik}).
As stated before, we will focus on the exponential kernel.
\begin{ass}[Exponential kernel]\label{ass:exponential}
We assume that the nonlocal operator in \cref{eq:nonlocal_operator_2} is of exponential type\footnote{The factor $p$ in front of the following integral is due to the kernel's normalization in the given \(p\) norm as postulated in \cref{ass:input_datum_and_velocity}}, i.e., for \(\eta\in\R_{>0}\ni p\) 
\begin{align}
W_{p,\eta}[q,\gamma]\coloneqq W_{p,\eta}[q,\exp(-\cdot)]&\equiv\left(\tfrac{p}{\eta}\int_{\cdot}^{\infty}\Big(\exp\big(\tfrac{\cdot-y}{\eta}\big)q(t,y)\Big)^{p}\dd y\right)^{\frac{1}{p}}\text{ on } \OT.\label{eq:nonlocal_operator_exponential}
\end{align}
\end{ass}

\begin{lem}[A PDE entirely in the nonlocal term]\label{lem:PDE_nonlocal}
    Let \cref{ass:exponential} hold and abbreviate, for simplicity, \(W_{\eta}\coloneqq W_{p,\eta}[q,\exp(-\cdot)]\). Then, the following identity holds
\begin{equation}
q_{\eta}^{p}(t,x)=W_{\eta}^{p}(t,x)-\eta W_{\eta}^{p-1}(t,x)\,\partial_{x}W_{\eta}(t,x),\ (t,x)\in\OT.\label{eq:nonlocal_identity}
\end{equation}
Moreover, the nonlocal operator satisfies its own Cauchy problem:
\begin{equation}
\begin{aligned}
    \partial_{t}W_{\eta}(t,x)&=-V\big(W_{\eta}(t,x)\big)\partial_{x}W_{\eta}(t,x)-W_{\eta}^{1-p}(t,x)\tfrac{p}{\eta}\int_{x}^{\infty}\exp(\tfrac{px-py}{\eta})V'\big(W_{\eta}(t,y)\big)\,\partial_{y}W_{\eta}(t,y)W_{\eta}(t,y)^{p}\dd y\\
    &\quad +W_{\eta}(t,x)^{1-p}(p-1)\int_{x}^{\infty}\!\!\!\!\!\!\exp(\tfrac{px-py}{\eta})V'(W_{\eta}(t,y))W_{\eta}^{p-1}(t,y)\big(\partial_{y}W_{\eta}(t,y)\big)^{2}\dd y,\ (t,x)\in\OT\\
    W_{\eta}(0,x)&=\left(\tfrac{p}{\eta}\int_{x}^{\infty}\Big(\exp\big(\tfrac{x-y}{\eta}\big)q_{0}(y)\Big)^{p}\dd y\right)^{\frac{1}{p}}\!\!,\ x\in\R.
\end{aligned}\label{eq:nonlocal_PDE}
\end{equation}
\end{lem}
\begin{proof}
Computing the spatial derivative, we obtain, for \((t,x)\in\OT\)
\begin{align*}
    \partial_x W_{\eta}(t,x)&= \tfrac{1}{p} \Big(\tfrac{p}{\eta}\int_x^\infty \left(\exp(\tfrac{x-y}{\eta})q(t,y)\right)^p \dd y\Big)^{\tfrac{1}{p}-1} \left(\tfrac{p^{2}}{\eta^{2}} \int_x^{\infty}  (\exp\big(\tfrac{x-y}{\eta})\big)^p q^p(t,y) \dd y- \tfrac{p}{\eta} q^p(t,x)\right)\\
    &=\tfrac{1}{\eta} W_{\eta}(t,x)- \tfrac{1}{\eta}\Big( W_{\eta}(t,x)\Big)^{1-p} q_{\eta}^p(t,x)
\end{align*}
which can be rewritten
\begin{equation}
    q_{\eta}^{p}(t,x)=W_{\eta}^{p}(t,x)-\eta W_{\eta}^{p-1}(t,x)\,\partial_{x}W_{\eta}(t,x).
\label{eq:relation_q_W}
\end{equation}
This is exactly the identity claimed in \cref{eq:nonlocal_identity}.

For the second part, let us consider strong or even classical solutions of \cref{defi:model_class}, which is possible thanks to the approximation result in \cref{theo:existence_uniqueness_maximum_principle} and \cref{lem:approx}. Then, 
\begin{align}
    \partial_{t}q_{\eta}^{p}(t,x)&=pq_{\eta}(t,x)^{p-1}\partial_{t} q_{\eta}(t,x)\notag
    \intertext{using that \(q_{\eta}\) is a strong solution}
    &=-p q_{\eta}(t,x)^{p-1}\partial_{x}(V(W_{\eta}(t,x))q_{\eta}(t,x))\notag\\
    &=-p q_{\eta}(t,x)^p V'\big(W_{\eta}(t,x)\big) \, \partial_x W_{\eta}(t,x)-p q_{\eta}(t,x)^{p-1} V\big(W_{\eta}(t,x)\big) \partial_{x}q_{\eta}(t,x)\notag
    \intertext{taking advantage of the identity for \(q_{\eta}\) in \cref{eq:relation_q_W} to eliminate \(q_{\eta}\)}
    &=-pV'\big(W_{\eta}(t,x)\big) \partial_x W_{\eta}(t,x)\big(W_{\eta}(t,x)^{p}-\eta W_{\eta}(t,x)^{p-1}\partial_{x}W_{\eta}(t,x)\big)\notag\\
    & \quad-V(W_{\eta}(t,x))\Big(pW_{\eta}(t,x)^{p-1}\partial_{x}W_{\eta}(t,x)\label{eq:time_derivative_q}
\\
    &\qquad-\eta W_{\eta}(t,x)^{p-1}\partial_{x}^{2}W_{\eta}(t,x)-\eta (p-1)W_{\eta}(t,x)^{p-2}\big(\partial_{x}W_{\eta}(t,x)\big)^{2}\Big).\notag
\end{align}
Now, we can compute the time derivative of the nonlocal term \(W_{\eta}\) recalling its definition in \cref{eq:nonlocal_operator_exponential} as follows
\begin{align*}
    \partial_{t}W_{\eta}(t,x)&=W_{\eta}(t,x)^{1-p}\tfrac{1}{\eta}\int_{x}^{\infty}\exp(\tfrac{px-py}{\eta})\partial_{t}\big(q_{\eta}(t,y)^{p}\big)\dd y\\
    \intertext{plugging in the terms around \cref{eq:time_derivative_q}}
    &=-pW_{\eta}(t,x)^{1-p}\tfrac{1}{\eta}\int_{x}^{\infty}\exp(\tfrac{px-py}{\eta})W_{\eta}(t,y)^{p}V'(W_{\eta}(t,y))\partial_{y}W_{\eta}(t,y)\dd y\\
    &\quad +pW_{\eta}(t,x)^{1-p}\int_{x}^{\infty}\exp(\tfrac{px-py}{\eta}) W_{\eta}^{p-1}(t,y)\big(\partial_{y}W_{\eta}(t,y)\big)^{2}V'(W_{\eta}(t,y))\dd y\\
    &\quad -pW_{\eta}(t,x)^{1-p}\tfrac{1}{\eta}\int_{x}^{\infty}\exp(\tfrac{px-py}{\eta})W_{\eta}(t,y)^{p-1}\partial_{y}W_{\eta}(t,y)V(W_{\eta}(t,y))\dd y\\
    &\quad +W_{\eta}(t,x)^{1-p}(p-1)\int_{x}^{\infty}\exp(\tfrac{px-py}{\eta})W_{\eta}(t,y)^{p-2}\big(\partial_{y}W_{\eta}(t,y)\big)^{2}V(W_{\eta}(t,y))\dd y\\
    &\quad +W_{\eta}(t,x)^{1-p}\int_{x}^{\infty}\exp(\tfrac{px-py}{\eta})W_{\eta}(t,y)^{p-1}\partial_{y}^{2}W_{\eta}(t,y)V(W_{\eta}(t,y))\dd y.\\
    \intertext{Now an integration by parts in the last term gets rid of the second spatial derivative of \(W_{\eta}\) to yield}
 &=-pW_{\eta}(t,x)^{1-p}\tfrac{1}{\eta}\int_{x}^{\infty}\exp(\tfrac{px-py}{\eta})W_{\eta}(t,y)^{p}V'\big(W_{\eta}(t,y)\big)\partial_{y}W_{\eta}(t,y)\dd y\\
    &\quad +(p-1)W_{\eta}(t,x)^{1-p}\int_{x}^{\infty}\exp(\tfrac{px-py}{\eta}) W_{\eta}^{p-1}(t,y)\big(\partial_{y}W_{\eta}(t,y)\big)^{2}V'\big(W_{\eta}(t,y)\big)\dd y\\
    &\quad -V\big(W_{\eta}(t,x)\big)\partial_{x}W_{\eta}(t,x),
\end{align*}
which is the identity in \cref{eq:nonlocal_PDE}. The initial datum follows from the definition of the nonlocal operator in \cref{eq:nonlocal_operator_exponential} when setting \(t=0\) and recalling that \(q(0,\cdot)\equiv q_{0}\). Finally, one can use the stability result in \cref{theo:existence_uniqueness_maximum_principle} to conclude that the result also holds for an initial datum \(q_{0}\) satisfying only \cref{ass:input_datum_and_velocity}, i.e., being total variation bounded and essentially bounded but not smooth.
\end{proof}
It is also possible to identify a PDE in the \(p\)-th power of the nonlocal operator. This is what the following remark sketches.
\begin{rem}[A PDE on \(W_{\eta}^{p},\) the special case \(p=1\) and on minimal/maximal points of \(W_{\eta}\)]\label{rem:PDE_nonlocal_p}
Multiplying the PDE in \cref{eq:nonlocal_PDE} with \(W_{\eta}^{p-1}\), we obtain for \((t,x)\in\OT\)
\begin{align*}
        W^{p-1}_{\eta}(t,x)\partial_{t}W_{\eta}(t,x)&=-V\big(W_{\eta}(t,x)\big)W^{p-1}_{\eta}(t,x)\,\partial_{x}W_{\eta}(t,x)\\
        &\quad -\tfrac{p}{\eta}\int_{x}^{\infty}\!\!\exp(\tfrac{px-py}{\eta})V'\big(W_{\eta}(t,y)\big)\,\partial_{y}W_{\eta}(t,y)W_{\eta}(t,y)^{p}\dd y\\
    &\quad +(p-1)\int_{x}^{\infty}\!\!\exp(\tfrac{px-py}{\eta})V'\big(W_{\eta}(t,y)\big)W_{\eta}^{p-1}(t,y)\big(\partial_{y}W_{\eta}(t,y)\big)^{2}\dd y,
\end{align*}
which can be written as
\begin{equation}
\begin{aligned}
\partial_{t}W^{p}_{\eta}(t,x)&=-V\big(W_{\eta}(t,x)\big)\,\partial_{x}W^{p}_{\eta}(t,x)\\
        &\quad -\tfrac{p^{2}}{\eta}\int_{x}^{\infty}\!\!\exp(\tfrac{px-py}{\eta})V'\big(W_{\eta}(t,y)\big)\,\partial_{y}W_{\eta}(t,y)W_{\eta}(t,y)^{p}\dd y\\
    &\quad +p(p-1)\int_{x}^{\infty}\!\!\exp(\tfrac{px-py}{\eta})V'\big(W_{\eta}(t,y)\big)W_{\eta}^{p-1}(t,y)\big(\partial_{y}W_{\eta}(t,y)\big)^{2}\dd y.
\end{aligned}
\label{eq:nonlocal_W_p}
\end{equation}
Supplementing this with the proper initial datum,
\[
W^{p}_{\eta}(0,x)=\tfrac{p}{\eta}\int_{x}^{\infty}\!\!\Big(\exp\big(\tfrac{x-y}{\eta}\big)q_{0}(y)\Big)^{p}\dd y,\ x\in\R,
\]
we have obtained a PDE in \(W^{p}_{\eta}\) (involving, however, still \(\partial_{y}W_{\eta}\)) which will become handy in the following \cref{theo:TV_bounds_Wp} when obtaining uniform \(TV\) bounds.

Furthermore, on setting \(p=1\) in \cref{eq:nonlocal_PDE}, the previously obtained nonlocal dynamics in \(W_{\eta}^{p}\) result in the well-known identity \cite[Lemma 3.1]{coclite2022general} which plays a crucial role in the singular limit case for \(p=1\)
\begin{align*}
\partial_{t}W_{\eta}(t,x)&=-V\big(W_{\eta}(t,x)\big)\,\partial_{x}W_{\eta}(t,x) -\tfrac{1}{\eta}\int_{x}^{\infty}\!\!\exp(\tfrac{x-y}{\eta})V'\big(W_{\eta}(t,y)\big)\,\partial_{y}W_{\eta}(t,y)\,W_{\eta}(t,y)\dd y,\quad (t,x)\in\OT,
\end{align*}
as it guarantees uniform (in \(\eta\)) \(TV\) bounds of the nonlocal operator (for which see \cite[Theorem 3.2]{coclite2022general}). A similar strategy will be used (compare in particular \cref{theo:TV_bounds_Wp}) in this contribution for the general \(p\in(0,\infty)\) case.

Another point worth mentioning is the behavior of \(W_{\eta}\) at a minimum/maximum for \(p\neq 1\). To this end, define \(F(\cdot)\coloneqq \int_{0}^{\cdot} V'(s)s^{p}\dd s\) and rewrite the nonlocal PDE in \cref{eq:nonlocal_PDE} as
\begin{align*}
    \partial_{t}W_{\eta}(t,x)&=-V\big(W_{\eta}(t,x)\big)\,\partial_{x}W_{\eta}(t,x)-W_{\eta}^{1-p}(t,x)\tfrac{p}{\eta}\int_{x}^{\infty}\!\!\exp(\tfrac{px-py}{\eta})\tfrac{\dd}{\dd y}F\big(W_{\eta}(t,y)\big)\dd y\\
    &\quad +W_{\eta}(t,x)^{1-p}(p-1)\int_{x}^{\infty}\!\!\exp(\tfrac{px-py}{\eta})V'\big(W_{\eta}(t,y)\big)W_{\eta}^{p-1}(t,y)\big(\partial_{y}W_{\eta}(t,y)\big)^{2}\dd y.
    \intertext{Integrating by parts in the second term,}
    &=-V\big(W_{\eta}(t,x)\big)\,\partial_{x}W_{\eta}(t,x)-W_{\eta}^{1-p}(t,x)\tfrac{p}{\eta}\Big(\tfrac{p}{\eta}\int_{x}^{\infty}\!\!\exp(\tfrac{px-py}{\eta})F\big(W_{\eta}(t,y)\big)\dd y-F\big(W_{\eta}(t,x)\big)\Big)\\
    &\quad +W_{\eta}(t,x)^{1-p}(p-1)\int_{x}^{\infty}\!\!\exp(\tfrac{px-py}{\eta})V'\big(W_{\eta}(t,y)\big)\,W_{\eta}^{p-1}(t,y)\big(\partial_{y}W_{\eta}(t,y)\big)^{2}\dd y.
    \intertext{At a maximal point, viz., \(\tilde{x}\in\R: W_{\eta}(t,\tilde{x})=\|W_{\eta}(t,\cdot)\|_{L^{\infty}(\R)},\ \partial_{2}W_{\eta}(t,\tilde{x})=0\) it holds that}
    \partial_{t}W_{\eta}(t,\tilde{x})&=-W_{\eta}(t,\tilde{x})^{1-p}\tfrac{p}{\eta}\Big(\tfrac{p}{\eta}\int_{\tilde{x}}^{\infty}\!\!\exp(\tfrac{p\tilde{x}-py}{\eta})F\big(W_{\eta}(t,y)\big)\dd y-F(W_{\eta}(t,\tilde{x}))\Big)\\
    &\quad +W_{\eta}(t,\tilde{x})^{1-p}(p-1)\int_{\tilde{x}}^{\infty}\!\!\exp(\tfrac{p\tilde{x}-py}{\eta})V'\big(W_{\eta}(t,y)\big)W_{\eta}^{p-1}(t,y)\big(\partial_{y}W_{\eta}(t,y)\big)^{2}\dd y
    \intertext{and, as \(W(t,\tilde{x})\) is maximal and \(F\) is monotonically decreasing,}
    & \leq -W_{\eta}(t,\tilde{x})^{1-p}\tfrac{p}{\eta}\Big(F(W_{\eta}(t,x))\tfrac{p}{\eta}\int_{\tilde{x}}^{\infty}\!\!\exp(\tfrac{p\tilde{x}-py}{\eta})\dd y-F\big(W_{\eta}(t,\tilde{x})\big)\Big)\\
    &\quad +W_{\eta}(t,\tilde{x})^{1-p}(p-1)\int_{\tilde{x}}^{\infty}\!\!\exp(\tfrac{p\tilde{x}-py}{\eta})V'\big(W_{\eta}(t,y)\big)W_{\eta}^{p-1}(t,y)\big(\partial_{y}W_{\eta}(t,y)\big)^{2}\dd y\\
    &=W_{\eta}(t,\tilde{x})^{1-p}(p-1)\int_{\tilde{x}}^{\infty}\!\!\exp(\tfrac{p\tilde{x}-py}{\eta})V'\big(W_{\eta}(t,y)\big)W_{\eta}^{p-1}(t,y)\big(\partial_{y}W_{\eta}(t,y)\big)^{2}\dd y.
\end{align*}
This last term is, in general, strictly negative (as long as the solution is not constant for \(x\geq\tilde{x}\)) for \(p>1\) so that the maximal value of \(W_{\eta}(0,\cdot)\) is, for \(t\in\R_{>0}\), never attained and decreasing. Similarly, if \(\hat{x}\) is chosen so that we are at a point where the solution is minimal, we have the inequality
\[
\partial_{t}W_{\eta}(t,\hat{x})\geq W_{\eta}(t,\hat{x})^{1-p}(p-1)\int_{\hat{x}}^{\infty}\!\!\exp(\tfrac{p\hat{x}-py}{\eta})V'\big(W_{\eta}(t,y)\big)W_{\eta}^{p-1}(t,y)\big(\partial_{y}W_{\eta}(t,y)\big)^{2}\dd y.
\]
Thus, even at a minimal point, if \(p\in (1,\infty)\) then the solution will decay further. This might sound contradictory, however, following the maximum/minimum principle in \cref{theo:existence_uniqueness_maximum_principle}, the solution can never become smaller than \(q_{\min}\in\R_{>0}\). For \(p\in (0,1),\) quite the opposite is true: the minimum increases (but not beyond \(\|q_{0}\|_{L^{\infty}(\R)}\)) as well as the maximum decreasing.

Interestingly, the case \(p=1\) is ``equilibrated'' in the sense that this last signed term vanishes so that we end up with
\[
\partial_{t}W_{\eta}(t,\tilde{x})\leq 0 \text{ and } \partial_{t}W_{\eta}(t,\hat{x})\geq 0,
\]
i.e., for \(p=1\) (and only in this case (as long as the solution is not constant)) it indeed holds that
\[
\essinf_{\tilde{x}\in\R}W_{\eta}(t,\tilde{x})\leq W_{\eta}(t,x)\leq \|W_{\eta}(0,\cdot)\|_{L^{\infty}(\R)},\ (t,x)\in\OT \text{ a.e.}
\]
\end{rem}

\subsection{\texorpdfstring{\(TV\)}{TV} estimates on the nonlocal operator}
With the previous identity in \cref{lem:PDE_nonlocal,{rem:PDE_nonlocal_p}} for the \textit{exponential nonlocal term} in mind, we are well equipped to show total variation bounds for the \(p-\)th power of \(W_{\eta}\), \(W^{p}_{\eta}, p\in(0,\infty)\) uniformly in \(\eta\in\R_{>0}\). 
\begin{theo}[\(TV\) bounds on {\(W^{p}_{\eta}\equiv W_{p,\eta}^{p}[q,\exp(-\cdot)]\)}]\label{theo:TV_bounds_Wp}
Let \cref{ass:input_datum_and_velocity} and \cref{ass:exponential} hold. If either \begin{itemize}
    \item \(p\in\R_{\geq 1}\)
    \item or \(p<1\) and the initial datum satisfies \(\tfrac{\essinf_{x\in\R}q_{0}(x)}{\|q_{0}\|_{L^{\infty}(\R)}}\geq \sqrt[p]{1-p}\)
\end{itemize}
it holds that the total variation of \(W_{\eta}^{p}\coloneqq W_{p,\eta}[q,\exp(-\cdot)]^{p}\) as in \cref{lem:PDE_nonlocal} diminishes over time, uniformly in \(\eta\in\R_{>0}\), i.e. that
\[
\forall (\eta,t)\in\R_{>0}\times[0,T]: \quad \big|W_{\eta}^{p}(t,\cdot)\big|_{TV(\R)}\leq \big|W_{\eta}^{p}(0,\cdot)\big|_{TV(\R)}\leq|q_{0}^{p}|_{TV(\R)}.
\]
\end{theo}

\begin{proof}
    Before we start computing the \(TV\) estimate, we require the time and space derivatives of \(W_{\eta}^{p}\) with the notation introduced in \cref{lem:PDE_nonlocal}, following \cref{rem:PDE_nonlocal_p} and assuming again that the solutions are smooth which is possible thanks to \cref{lem:approx}. For \((t,x)\in\OT\), \(\eta\in\R_{>0}\ni p\) we take the spatial derivative in \cref{eq:nonlocal_W_p}
    \begin{equation}
    \begin{aligned}
\partial_{t}\partial_{x}W_{\eta}^{p}(t,x)&=-\tfrac{\dd}{\dd x}\big(V(W_{\eta}(t,x))\partial_{x}W_{\eta}^{p}(t,x)\big)+\tfrac{p^{2}}{\eta}V'\big(W_{\eta}(t,x)\big)\partial_{x}W_{\eta}(t,x)W_{\eta}(t,x)^{p}\\
&\quad-\tfrac{p^{3}}{\eta^{2}}\int_{x}^{\infty}\!\!\!\!\!\exp(\tfrac{px-py}{\eta})V'\big(W_{\eta}(t,y)\big)\,\partial_{y}W_{\eta}(t,y)\,W_{\eta}(t,y)^{p}\dd y\\
    &\quad-p(p-1)V'\big(W_{\eta}(t,x)\big)W_{\eta}(t,x)^{p-1}\big(\partial_{x}W_{\eta}(t,x)\big)^{2}\\
    &\quad +\tfrac{p^{2}(p-1)}{\eta}\int_{x}^{\infty}\!\!\exp(\tfrac{px-py}{\eta})V'\big(W_{\eta}(t,y)\big)\,W_{\eta}^{p-1}(t,y)\big(\partial_{y}W_{\eta}(t,y)\big)^{2}\dd y.
\end{aligned}
\label{eq:partial_t_x_W_p}
\end{equation}
With this identity, we can now estimate the temporal change in the total variation of \(W_{\eta}^{p}\) as follows
\begin{align*}
    &\tfrac{\dd}{\dd t}\int_{\R}\big|\tfrac{\dd}{\dd x} W^{p}(t,x)\big|\dd x=\int_{\R}\sgn\big(\tfrac{\dd}{\dd x}W_{\eta}(t,x)^{p}\big)\partial_{t}\partial_{x} W_{\eta}(t,x)^{p}\dd x.
    \intertext{Plugging in the previous identity \cref{eq:partial_t_x_W_p},}
    &= -\tfrac{p^{3}}{\eta^{2}}\int_{\R}\sgn\big(\tfrac{\dd}{\dd x}W_{\eta}(t,x)^{p}\big)\int_{x}^{\infty}\exp(\tfrac{px-py}{\eta})V'\big(W_{\eta}(t,y)\big)\partial_{y}W_{\eta}(t,y)W_{\eta}(t,y)^{p}\dd y \dd x\\
    &\quad+\tfrac{p^{2}}{\eta}\int_{\R}\sgn\big(\tfrac{\dd}{\dd x}W_{\eta}(t,x)^{p}\big)V'\big(W_{\eta}(t,x)\big)\partial_{x}W_{\eta}(t,x)W_{\eta}(t,x)^{p}\dd x\\
    &\quad +\tfrac{p^{2}(p-1)}{\eta}\int_{\R}\sgn\big(\tfrac{\dd}{\dd x}W_{\eta}(t,x)^{p}\big)\int_{x}^{\infty}\exp(\tfrac{px-py}{\eta})V'\big(W_{\eta}(t,y)\big)W_{\eta}^{p-1}(t,y)\big(\partial_{y}W_{\eta}(t,y)\big)^{2}\dd y\dd x\\
    &\quad -p(p-1)\int_{\R}\sgn\big(\tfrac{\dd}{\dd x}W_{\eta}(t,x)^{p}\big)V'\big(W_{\eta}(t,x)\big)W_{\eta}(t,x)^{p-1}\big(\partial_{x}W_{\eta}(t,x)\big)^{2} \dd x\\
    &\quad -p\int_{\R}\sgn\Big(\tfrac{\dd}{\dd x}W_{\eta}(t,x)^{p}\Big)\tfrac{\dd}{\dd x} \Big(V\big(W_{\eta}(t,x)\big)W_{\eta}(t,x)^{p-1}\partial_{x}W_{\eta}(t,x)\Big)\dd x.
    \intertext{Integrating by parts in the last term,}
    &= -\tfrac{p^{3}}{\eta^{2}}\int_{\R}\sgn\big(\tfrac{\dd}{\dd x}W_{\eta}(t,x)^{p}\big)\int_{x}^{\infty}\exp(\tfrac{px-py}{\eta})V'\big(W_{\eta}(t,y)\big)\partial_{y}W_{\eta}(t,y)W_{\eta}(t,y)^{p}\dd y \dd x\\
    &\quad+\tfrac{p}{\eta}\int_{\R}\big|\tfrac{\dd}{\dd x}W_{\eta}^{p}(t,x)\big|W_{\eta}(t,x)V'\big(W_{\eta}(t,x)\big)\dd x\\
    &\quad +\tfrac{p^{2}(p-1)}{\eta}\int_{\R}\sgn\big(\tfrac{\dd}{\dd x}W_{\eta}(t,x)^{p}\big)\int_{x}^{\infty}\exp(\tfrac{px-py}{\eta})V'\big(W_{\eta}(t,y)\big)W_{\eta}(t,y)^{p-1}\big(\partial_{y}W_{\eta}(t,y)\big)^{2}\dd y\dd x\\
    &\quad -(p-1)\int_{\R}\big|\tfrac{\dd}{\dd x}W_{\eta}^{p}(t,x)\big|V'\big(W_{\eta}(t,x)\big)\partial_{x}W_{\eta}(t,x)\dd x\\
    &\quad+\int_\R \delta\big(\tfrac{\dd}{\dd x}W_{\eta}(t,x)^{p}\big)
    V\big(W_\eta(t,x)\big) \tfrac{\dd}{\dd x}W_{\eta}(t,x)^{p}\partial_{x}^{2}W_{\eta}(t,x)^{p} \dd x
    \intertext{and, as the last term is zero, we obtain by combining the first and third and second and fourth terms}
    &= -\tfrac{p^{2}}{\eta^{2}}\int_{\R}\sgn\big(\tfrac{\dd}{\dd x}W_{\eta}(t,x)^{p}\big)\int_{x}^{\infty}\exp(\tfrac{px-py}{\eta})V'\big(W_{\eta}(t,y)\big)\partial_{y}W_{\eta}(t,y)\\
    &\quad\quad \cdot\Big(pW_{\eta}(t,y)^{p}-(p-1)\eta W_{\eta}(t,y)^{p-1}\partial_{y}W_{\eta}(t,y)\Big)\dd y \dd x\\
    &\quad+\tfrac{1}{\eta}\int_{\R}\big|\tfrac{\dd}{\dd x}W_{\eta}^{p}(t,x)\big|V'\big(W_{\eta}(t,x)\big)\big(pW_{\eta}(t,x)-(p-1)\eta \partial_{x}W_{\eta}(t,x)\big)\dd x.\\
    \intertext{Changing the order of integration, recalling \(V'\leqq\), and estimating the first term yields}
    &\leq -\tfrac{p^{2}}{\eta^{2}}\int_{\R}V'\big(W_{\eta}(t,y)\big)|\partial_{y}W_{\eta}(t,y)|\big|pW_{\eta}(t,y)^{p}-(p-1)\eta W_{\eta}(t,y)^{p-1}\,\partial_{y}W_{\eta}(t,y)\big|\\
    &\qquad\qquad \cdot\int_{-\infty}^{y}\big|\sgn\big(\tfrac{\dd}{\dd x}W_{\eta}(t,x)^{p}\big)\big|\exp(\tfrac{px-py}{\eta})\dd x\dd y\\
    &\quad+\tfrac{1}{\eta}\int_{\R}\big|\tfrac{\dd}{\dd x}W_{\eta}^{p}(t,x)\big|V'\big(W_{\eta}(t,x)\big)\big(pW_{\eta}(t,x)-(p-1)\eta \partial_{x}W_{\eta}(t,x)\big)\dd x.
    \intertext{Carrying out the integral in \(x\) as \(|\sgn(\cdot)|\equiv 1\) in the first term,}
    &\leq -\tfrac{p}{\eta}\int_{\R}V'(W_{\eta}(t,y))|\partial_{y}W_{\eta}(t,y)|\big|pW_{\eta}(t,y)^{p}-(p-1)\eta W_{\eta}(t,y)^{p-1}\partial_{y}W_{\eta}(t,y)\big|\dd y\\
    &\quad+\tfrac{1}{\eta}\int_{\R}\big|\tfrac{\dd}{\dd x}W_{\eta}^{p}(t,x)\big|V'\big(W_{\eta}(t,x)\big)\big(pW_{\eta}(t,x)-(p-1)\eta \partial_{x}W_{\eta}(t,x)\big)\dd x
    \intertext{and, as \(W_{\eta}\geqq 0\), we can use \(\big|\tfrac{\dd}{\dd x}W_{\eta}^{p}(t,x)\big|=pW^{p-1}_{\eta}(t,x)|\partial_{x} W_{\eta}(t,x)|,\ (t,x)\in\OT\) in the second term}
    &= -\tfrac{p}{\eta}\int_{\R}|\partial_{y}W_{\eta}(t,y)|V'\big(W_{\eta}(t,y)\big)\big|pW_{\eta}(t,y)^{p}-(p-1)\eta W_{\eta}(t,y)^{p-1}\partial_{y}W_{\eta}(t,y)\big|\dd y\\
    &\quad+\tfrac{p}{\eta}\int_{\R}\big|\partial_{x}W_{\eta}(t,x)\big|V'\big(W_{\eta}(t,x)\big)\big(pW_{\eta}(t,x)^{p}-(p-1)\eta W_{\eta}^{p-1}(t,x)\partial_{x}W_{\eta}(t,x)\big)\dd x.
    \intertext{Assume for now that \(pW_{\eta}^{p}-(p-1)\eta W_{\eta}^{p-1}\partial_{2}W_{\eta}\geqq 0\), which will be justified below we continue.}
    &= -\tfrac{p}{\eta}\int_{\R}|\partial_{y}W_{\eta}(t,y)|V'\big(W_{\eta}(t,y)\big)\big(pW_{\eta}(t,y)^{p}-(p-1)\eta W_{\eta}(t,y)^{p-1}\partial_{y}W_{\eta}(t,y)\big)\dd y\\
    &\quad+\tfrac{p}{\eta}\int_{\R}\big|\partial_{x}W_{\eta}(t,x)\big|V'\big(W_{\eta}(t,x)\big)\big(pW_{\eta}(t,x)^{p}-(p-1)\eta W_{\eta}^{p-1}(t,x)\partial_{x}W_{\eta}(t,x)\big)\dd x\\
    &=0.
\end{align*}   
This proves the \(TV\) diminishing property in the case that 
\[
pW_{\eta}^{p}-(p-1)\eta W_{\eta}^{p-1}\partial_{2}W_{\eta}\geqq 0 \text{ on } \OT.
\]
Using the exponential identity in \cref{eq:nonlocal_identity}, this can be reformulated as
\[
 W_{\eta}^{p}+(p-1)W_{\eta}^{p}-(p-1)\eta W_{\eta}^{p-1}\partial_{2}W_{\eta}\geqq 0 \Longleftrightarrow W_{\eta}^{p}+(p-1)q_{\eta}^{p}\geqq 0 \text{ on } \OT.
\]
Thanks to \cref{theo:existence_uniqueness_maximum_principle}, this inequality is always satisfied for \(p\geq 1\), which is guaranteed by the postulated assumption. For \(p<1\), we can perform a ``worst-case'' estimate when taking advantage of the previously derived min/max principle in \cref{theo:existence_uniqueness_maximum_principle}, namely in \cref{eq:min_max}  to arrive at
\[
W_{\eta}^{p}\geqq (1-p)q_{\eta}^{p}\Longleftarrow \essinf_{x\in\R}q_{0}(x)^{p}\geqq (1-p)\|q_{0}\|_{L^{\infty}(\R)}^{p}\Longleftrightarrow \tfrac{\essinf_{x\in\R}q_{0}(x)}{\|q_{0}\|_{L^{\infty}(\R)}}\geq \sqrt[p]{1-p}.
\]
This is exactly the requirement that we postulated in the assumptions of the proof for \(p\in (0,1)\).
As, in both cases \(p\geq 1,\ 0<p<1,\) the time-derivative of the total variation is less than or equal to zero, we obtain 
\[
\big|W_{\eta}^{p}(t,\cdot)\big|_{TV(\R)}\leq \big|W_{\eta}^{p}(0,\cdot)\big|_{TV(\R)}.
\]
Recalling the exponential operator, we have, for \(t=0\)
\begin{equation}
\begin{aligned}
    \int_{\R}\big|\partial_{x}W_{\eta}^{p}(0,x)\big|\dd x&=\int_{\R}\Big|\partial_{x}\tfrac{p}{\eta}\int_{x}^{\infty}\exp\big(\tfrac{px-py}{\eta}\big)q_{0}(y)^{p}\dd y\Big|\dd x=\int_{\R}\Big|\tfrac{p}{\eta}\int_{x}^{\infty}\exp\big(\tfrac{px-py}{\eta}\big)\tfrac{\dd}{\dd y}q_{0}(y)^{p}\dd y\Big|\dd x\\
    &\leq \int_{\R}\tfrac{p}{\eta}\int_{x}^{\infty}\exp\big(\tfrac{px-py}{\eta}\big)\big|\tfrac{\dd}{\dd y}q_{0}(y)^{p}\big|\dd y\dd x
    =\int_{\R}\big|\tfrac{\dd}{\dd y}q_{0}(y)^{p}\big|\tfrac{p}{\eta}\int_{-\infty}^{y}\exp\big(\tfrac{px-py}{\eta}\big)\dd x\dd y\\
    &=|q_{0}^{p}|_{TV(\R)}\leq \begin{cases} p\|q_{0}\|^{p-1}_{L^{\infty}(\R)}|q_{0}|_{TV(\R)}& \text{ for } p\geq 1\\
    p\|\tfrac{1}{q_{0}}\|_{L^{\infty}(\R)}^{1-p} |q_{0}|_{TV(\R)} &\text{ for } p\leq 1
    \end{cases}
\end{aligned}
\label{eq:uniform_TV_bound_initial_uniform_eta}
\end{equation}
which is uniform in \(\eta\) as claimed and depends only on the \(L^{\infty}\)-norm and \(TV\)-semi-norm of \(q_{0}\), which is bounded as guaranteed in \cref{ass:input_datum_and_velocity}.
\end{proof}
\begin{rem}[\(TV\) bounds on the nonlocal term, zero initial data and the \(p<1\) case]\label{rem:TV_bounds_generalization}
Note that in the proof of \cref{theo:TV_bounds_Wp}, we would not require any further assumption on the initial datum being bounded away from zero, if \(p\geq 1\). This is due to the fact that we were studying the \(TV\) estimate for \(W_{\eta}^{p}\) which behaves better than \(W_{\eta}\) itself.
This fact will play a crucial role in \cref{sec:existence_zero_initial_datum} when we consider, for \(p\in[1,\infty)\), the existence of solutions also for an initial datum which can be zero.

The same argumentation is not true, however, for the case \(0<p<1\) because the assumption required on the initial datum is rather restrictive and prevents the initial datum from being zero to have \(TV\) estimates uniformly in \(\eta\). This can also be seen by the latter derived uniform bound in \cref{eq:uniform_TV_bound_initial_uniform_eta} on the initial datum, which would not hold if \(p<1\) and \(\essinf_{x\in\R}q_{0}=0\).
It is also worth mentioning that the result in \cref{theo:TV_bounds_Wp} might be generalizable to broader kernel classes following the argument in \cite{Marconi2023,keimerpflug23fixed}. As we are mainly focused on ``first results'' regarding the \(p-\)norm, we will not look further into such a generalization here.

Finally, the total variation bound of \(W_{\eta}^{p}\) implies also a total variation bound on \(W_{\eta}\) as, for \(t\in(0,T)\) and \(p\in(1,\infty)\), it holds that
\begin{equation}
\begin{split}
 |W_{\eta}(t,\cdot)|_{TV(\R)}&=\int_{\R}|\partial_{x}W_{\eta}(t,x)|\dd x=\tfrac{1}{p}\int_{\R}\bigg|\tfrac{pW_{\eta}^{p-1}(t,x)}{W_{\eta}^{p-1}(t,x)}\partial_{x}W_{\eta}(t,x)\bigg|\dd x\\
 &\leq \tfrac{1}{pq_{\min}^{p-1}}\int_{\R}\big|pW_{\eta}(t,x)^{p-1}\partial_{x} W_{\eta}(t,x)\big|\dd x=\tfrac{1}{pq_{\min}^{p-1}}\int_{\R}|\partial_{x}W_{\eta}(t,x)^{p}|\dd x=\tfrac{|W_{\eta}^{p}(t,\cdot)|_{TV(\R)}}{pq_{\min}^{p-1}}.
\end{split}
\label{rem:total_variation_bound_W}
\end{equation}
This is guaranteed by \(q_{\eta},W_{\eta}\) being bounded away from zero by \(q_{\min}\) as assumed in \cref{ass:input_datum_and_velocity} and guaranteed in \cref{theo:existence_uniqueness_maximum_principle}.
\end{rem}

So far, we have a uniform \(TV\) estimate in space for the nonlocal operator's \(p\)th power, but any time regularity, and compactness with respect to both space and time variables, are missing. These are provided by \cref{cor:compactnessWp}.
\begin{cor}[Compactness of {\(W^{p}_{\eta}\) in \(C([0,T];L^{1}_{\loc}(\R))\)}] \label{cor:compactnessWp}
Under the assumptions of \cref{theo:TV_bounds_Wp} on \(p\in\R_{>0}\) and on the initial datum \(q_{0}\), we obtain for all \(\Omega\subset\R\) open and bounded that 
\[
\Big\{W^{p}_{\eta}\big|_{(0,T)\times\Omega}\in C\big([0,T];L^{1}(\Omega)\big):\ \eta\in\R_{>0}\Big\}\overset{\text{c}}{\hookrightarrow} C\big([0,T];L^{1}(\Omega)\big),
\]
i.e., the set \(\{W^{p}_{\eta}:\ \eta\in\R_{>0}\}\) is compactly embedded into the space \(C\big([0,T];L^{1}(\Omega)\big)\).
\end{cor}
\begin{proof}
We consider the Banach space \(L^1(\Omega)\) with \(\Omega \subset \R\) open bounded and, for $t\in [0,T]$, the set of functions
\[F(t)\coloneqq \{W_{\eta}^p(t, \cdot)\in L^1(\Omega),\,\eta\in\R_{>0}\}.
\] 
This set is, according to the classical argument that \(BV\) is compactly embedded in \(L^{1}\) as stated in \cite[Theorem 13.35]{leoni} for all \(t\in[0,T]\), compact in \(L^1(\Omega)\) thanks to the uniform (in \(\eta\in\R_{>0}\)) total variation bounds as demonstrated in \cref{theo:TV_bounds_Wp}.
According to \cite[Lemma 1]{Simon1986}, demonstrating the compactness in space-time in \(C([0,T];L^{1}(\Omega))\) for \(\Omega\subset\R\) open and bounded boils down to showing equicontinuity in time of the set of functions considered, and is what we will prove next: 
Let \((t_{1},t_{2})\in [0,T]^{2}\) be given and assume, without loss of generality, that \(t_{1}\geq t_{2}\). We obtain
    \begin{align}
        &\big\|W_{\eta}(t_{1},\cdot)^{p}-W_{\eta}(t_{2},\cdot)^{p}\big\|_{L^{1}(\R)}\notag\\
        &\leq \int_{\R}\int_{t_{2}}^{t_{1}}|\partial_{t}W_{\eta}(s,x)^{p}|\dd s\dd x\notag
        \intertext{using \cref{rem:PDE_nonlocal_p}, i.e., the dynamics derived for \(W_{\eta}^{p}\),}
        &\leq \int_{\R}\int_{t_{2}}^{t_{1}}\big|V\big(W_{\eta}(s,x)\big)\partial_{x}W_{\eta}(s,x)^{p}\big|\dd s\dd x\notag\\
        &\quad +\tfrac{p^{2}}{\eta}\int_{\R}\int_{t_{2}}^{t_{1}}\!\!\int_{x}^{\infty}\exp\big(\tfrac{px-py}{\eta}\big)\big|V'\big(W_{\eta}(s,y)\big)\partial_{y}W_{\eta}(s,y) W_{\eta}(s,y)^{p}\big|\dd y\dd s\dd x\notag\\
        &\quad +\tfrac{p(p-1)}{\eta}\int_{\R}\int_{t_{2}}^{t_{1}}\!\!\int_{x}^{\infty}\exp\big(\tfrac{px-py}{\eta}\big)\big|V'\big(W_{\eta}(s,y)\big)W_{\eta}(s,y)^{p-1}\partial_{y}W_{\eta}(s,y)\eta \partial_{y}W_{\eta}(s,y)\big|\dd s\dd y\dd x.\notag
        \intertext{Exchanging the order of integration in the second and third term, rearranging terms and pulling out \(V,V'\) uniformly (which is possible thanks to the maximum principle in \cref{theo:existence_uniqueness_maximum_principle}) yields}
        &\leq \|V\|_{L^{\infty}((0,\|q_{0}\|_{L^{\infty}(\R)}))}|t_{2}-t_{1}|\big|W_{\eta}^{p}\big|_{L^{\infty}((0,T);TV(\R))}\notag\\
        &\quad +\tfrac{p}{\eta}\|V'\|_{L^{\infty}((0,\|q_{0}\|_{L^{\infty}(\R)}))}\int_{t_{2}}^{t_{1}}\int_{\R}\big|\partial_{y}W_{\eta}(s,y)^{p}\big| W_{\eta}(s,y)\int_{-\infty}^{y}\exp\big(\tfrac{px-py}{\eta}\big)\dd x\dd y\dd s\notag\\
        &\quad +\tfrac{|p-1|}{\eta}\|V'\|_{L^{\infty}((0,\|q_{0}\|_{L^{\infty}(\R)}))}\int_{t_{2}}^{t_{1}}\int_{\R}\big|\partial_{y}W_{\eta}(s,y)^{p}\big||\eta \partial_{y}W_{\eta}(s,y)|\int_{-\infty}^{y}\exp\big(\tfrac{px-py}{\eta}\big)\dd y\dd x\dd s\notag\\
        &= \|V\|_{L^{\infty}((0,\|q_{0}\|_{L^{\infty}(\R)}))}|t_{2}-t_{1}|\big|W_{\eta}^{p}\big|_{L^{\infty}((0,T);TV(\R))}\notag\\
        &\quad +\|V'\|_{L^{\infty}((0,\|q_{0}\|_{L^{\infty}(\R)}))}\int_{t_{2}}^{t_{1}}\int_{\R}\big|\partial_{y}W_{\eta}(s,y)^{p}\big| W_{\eta}(s,y)\dd y\dd s\notag\\
        &\quad +\tfrac{|p-1|}{p}\|V'\|_{L^{\infty}((0,\|q_{0}\|_{L^{\infty}(\R)}))}\int_{t_{2}}^{t_{1}}\int_{\R}\big|\partial_{y}W_{\eta}(s,y)^{p}\big||\eta \partial_{y}W_{\eta}(s,y)|\dd x\dd s\notag\\
        &\leq \|V\|_{L^{\infty}((0,\|q_{0}\|_{L^{\infty}(\R)}))}|t_{2}-t_{1}|\big|W_{\eta}^{p}\big|_{L^{\infty}((0,T);TV(\R))}\label{eq:time_compactness_1}\\
        &\quad +\|V'\|_{L^{\infty}((0,\|q_{0}\|_{L^{\infty}(\R)}))}|t_{2}-t_{1}|\|W_{\eta}\|_{L^{\infty}((0,T);L^{\infty}(\R))}\big|W_{\eta}^{p}\big|_{L^{\infty}((0,T);TV(\R))}\\
        &\quad +\tfrac{|p-1|}{p}\|V'\|_{L^{\infty}((0,\|q_{0}\|_{L^{\infty}(\R)}))}|t_{2}-t_{1}|\|\eta \partial_{2}W_{\eta}\|_{L^{\infty}((0,T);L^{\infty}(\R))}\big|W_{\eta}^{p}\big|_{L^{\infty}((0,T);TV(\R))}.\label{eq:time_compactness_3}
    \end{align}
    According to \cref{theo:TV_bounds_Wp}, the total variation  on \(W_{\eta}^{p}\) is uniformly bounded in \(\eta\). 
    Thanks to the maximum principle in \cref{theo:existence_uniqueness_maximum_principle}, we have also that the following estimate holds uniformly in \(\eta\)
    \[
    \|W_{\eta}\|_{L^{\infty}((0,T);L^{\infty}(\R))}\leq \|q_{0}\|_{L^{\infty}(\R)}.
    \]
    
Finally, the term \(\eta\partial_{x}W_{\eta}\) is also uniformly bounded as long as the initial datum \(q_0\) is bounded away from zero. This can be understood when recalling \cref{eq:nonlocal_identity} and \cref{ass:input_datum_and_velocity}, as we then have
    \[
\eta \partial_{2}W_{\eta}\equiv\tfrac{W_{\eta}^{p}-q_{\eta}^{p}}{W_{\eta}^{p-1}}\implies \|\eta \partial_{2}W_{\eta}\|_{L^{\infty}((0,T);L^{\infty}(\R))}\leq \max\Big\{\|q_{0}\|_{L^{\infty}(\R)},\tfrac{\|q_{0}\|_{L^{\infty}(\R)}^{p}}{q_{\min}^{p-1}}\Big\}.
    \]
    All three summands in \crefrange{eq:time_compactness_1}{eq:time_compactness_3} are uniformly bounded with respect to \(\eta\) and converge to zero for \(|t_{1}-t_{2}|\rightarrow0\), so that we obtain even Lipschitz continuity with respect to \(|t_{1}-t_{2}|\). Altogether, this gives the claimed compactness in \(C\big([0,T];L^{1}_{\text{loc}}(\R)\big)\).
\end{proof}
So far, we have only looked into \(W_{\eta}^{p}\), even though the crucial term in the nonlocal conservation law in \cref{defi:model_class} is actually \(W_{\eta}\). This is why we will establish, in the following theorem, compactness also for the set of nonlocal terms and solutions in \(C([0,T];L^{1}(\Omega))\) for a given \(\Omega\subset\R\) open and bounded.

\begin{theo}[Convergence in \(C(L^{1}_{\text{loc}})\) of the nonlocal operator and the nonlocal solution]
\label{thm:strong_convergence}
Let the assumptions of \cref{theo:TV_bounds_Wp} hold, i.e., that \cref{ass:input_datum_and_velocity},\ \cref{ass:exponential} and either
\begin{itemize}
    \item \(p\in\R_{\geq 1}\)
    \end{itemize}
    or
    \begin{itemize}
    \item \(p<1\) and the initial datum satisfying \(\tfrac{\essinf_{x\in\R}q_{0}(x)}{\|q_{0}\|_{L^{\infty}(\R)}}\geq \sqrt[p]{1-p}\)
\end{itemize}
being satisfied. Then, for every sequence \(\{\eta_k\}_{k\in \N}\subset \R_{>0}\) with \(\displaystyle{\lim_{k\to\infty}}\eta_k=0,\) there exists a subsequence (again denoted by \(\{\eta_k\}_{k\in \N}\) for convenience) and a function \(q_\ast\in C\big([0,T]; L^1_{\text{loc}}(\R)\big)\) such that 
\[
\lim_{k\rightarrow\infty}\|q_{\eta_k}-q_{*}\|_{C([0,T];L^{1}_{\text{loc}}(\R))}=0\ \text{ and } \ \lim_{k\rightarrow\infty}\|W_{\eta_{k}}-q_{*}\|_{C([0,T];L^{1}_{\text{loc}}(\R))}=0 \text{ in \cref{defi:model_class}},
\]
i.e., the nonlocal term and nonlocal solution converge in \(C\big([0,T];L^{1}_{\text{loc}}(\R)\big)\) to \(q_\ast\in C\big([0,T];L^{1}_{\text{loc}}(\R)\big)\).
\end{theo}
\begin{proof}
Thanks to the compactness in \(C\big([0,T];L^1_{\text{loc}}(\R)\big)\) of the set of nonlocal terms \(W^p_\eta\) in \cref{cor:compactnessWp}, we know that there exists a limit function \(q_\ast^{p}\in C\big([0,T];L^1_{\text{loc}}(\R)\big)\) such that 
\begin{equation*}
    \displaystyle{\lim_{k\to \infty}} \|W^p_{\eta_k}-q_\ast^{p} \|_{C([0,T]; L^1_{\text{loc}}(\R))}=0.
\end{equation*}
The \(p\)th power can be chosen here as limit as it is nonnegative due to \(q_{\eta},W_{\eta}\geqq 0\).

Thanks to the identity in \cref{eq:nonlocal_identity} \(q^p_{\eta}\equiv W^p_{\eta}-\eta W^{p-1}\partial_{2}W_{\eta}\ \forall\eta\in\R_{>0}\), we can write, for \(t\in[0,T]\),
\begin{align*}
    \|W^p_{\eta_k}(t,\cdot)- q^p_{\eta_k}(t,\cdot)\|_{L^1(\R)}\leq \tfrac{\eta_k}{p} |W_{\eta_{k}}^{p}|_{L^{\infty}((0,T);TV(\R))}
    \intertext{and, according to \cref{theo:TV_bounds_Wp},}
    \leq \tfrac{\eta_k}{p} \|W_{\eta_k}(0,\cdot)\|_{L^{\infty}((0,T);TV(\R))}\leq \eta_{k}\begin{cases} \|q_{0}\|^{p-1}_{L^{\infty}(\R)}|q_{0}|_{TV(\R)}& \text{ for } p\geq 1\\
    \big\|\tfrac{1}{q_{0}}\big\|_{L^{\infty}(\R)}^{1-p} |q_{0}|_{TV(\R)} &\text{ for } p\leq 1.
    \end{cases}
\end{align*}
The right hand side goes to zero for \(\eta_{k}\rightarrow 0\) 
and as \(W_{\eta_{k}}^{p}\) converges to a limit \(q^{p}_{*}\), so does \(q^{p}_{\eta_{k}}\).
From this and the Lipschitz continuity of \(x\mapsto x^{\frac{1}{p}}\) for \(x\in\R_{\geq \eps}\) and an \(\eps\in\R_{>0}\), the convergence of \(q_{\eta_{k}}, W_{\eta_{k}}\) to \(q_{*}\) in \(C\big([0,T];L^{1}_{\text{loc}}(\R)\big)\) follows.
\end{proof}

\subsection{The corresponding local conservation law and results on existence and uniqueness}
As we aim for the singular limit convergence, we first state against which solution we expect convergence to. Looking into the nonlocal dynamics in \cref{eq:pnorm_problem}, it becomes apparent that, when the nonlocal kernel converges to a Dirac distribution, we formally end up with the following (local) conservation law.
\begin{defi}[The corresponding local conservation laws]\label{defi:local_conservation_law}
We call the following local conservation law subject to the velocity \(V\in W^{1,\infty}_{\text{loc}}(\R)\) and the initial condition \(q_{0}\in L^{\infty}(\R)\cap TV(\R)\) satisfying \cref{ass:input_datum_and_velocity}
\begin{equation}
\begin{aligned}
q_{t}(t,x)+\partial_{x}\big(V\big( q(t,x)\big) q(t,x)\big)&=0,&& (t,x)\in\OT,\\
q(0,x)      &=q_{0}(x),      && x\in\R.
\end{aligned}
    \label{eq:pnorm_local_problem}
\end{equation}
the \textbf{corresponding} (to the nonlocal) local conservation law.
\end{defi}
The theory on local conservation laws is quite extensive, and in the following we present some rather standard results concerning existence and uniqueness and proper definition of entropy solutions which we require in our later established analysis.
We start with weak solutions although they are not necessarily unique.
\begin{defi}[Weak solutions for local conservation laws]\label{defi:weak_solution_local}
  We call, for \(T\in\R_{>0}\), the function \(q\in C\big([0,T];L^{1}_{\text{loc}}(\R)\big)\) a weak solution to the Cauchy problem in \cref{defi:local_conservation_law} given that \(V\in W^{1,\infty}_{\text{loc}}(\R)\) and the initial condition \(q_{0}\in L^{\infty}(\R)\cap TV(\R)\) iff \(\forall\phi\in C^{1}_{\text{c}}((-42,T)\times\R)\) it holds that
  \begin{gather}
      \iint_{\OT}\Big(\phi_{t}(t,x)+V\big(q(t,x)\big)\phi_{x}(t,x)\Big)q(t,x)\dd x\dd t+\int_{\R} q_{0}(x)\phi(0,x)\dd x=0.\label{eq:weak_solution_local}
  \end{gather}
     
\end{defi}
As weak solutions for local conservation laws are typically not unique, an additional condition is prescribed which is usually referred to as entropy condition:
\begin{defi}[Entropy condition for the local conservation law.] \label{def:entropysol_local}
Let us consider for the local conservation law as in \cref{defi:local_conservation_law}
\begin{equation}
    \alpha \in C^2(\R) \text{ convex, } \beta'\equiv\alpha' f' \text{ with } f\equiv(\cdot) V(\cdot) \text{ on } \R. \label{eq:entropy_flux_pair}
\end{equation}
Then, we call for $\varphi \in C^1_\text{c}((-42,T)\times \R; \R_{\geq 0})$ and \(q\in C\big([0,T]; L^1_{\text{loc}}(\R)\big)\), the expression
\begin{align*}
    \mathcal{EF}[\varphi, \alpha, q]\coloneqq\iint_{\Omega_T} \alpha\big(q(t,x)\big) \varphi_t(t,x)+ \beta\big(q(t,x)\big) \varphi_x(t,x) \dd x \dd t + \int_\R \alpha\big(q_0(x)\big) \varphi(0,x) \dd x
\end{align*}
the \textbf{entropy condition}.
Furthermore, $q\in C\left([0,T]; L^1_{loc}(\R; \R)\right)$ is called an entropy solution of the local conservation law in \cref{defi:local_conservation_law} if it satisfies 
\begin{equation}
    \mathcal{EF}[\varphi, \alpha, q]\geq 0 \quad \forall \varphi \in C^1_{\text{c}}\left((-42,T)\times \R; \R_{\geq 0}\right),\  \forall (\alpha,\beta)\in C^{2}(\R) \text{ as in \cref{eq:entropy_flux_pair}}.\label{eq:entropy_condition}
\end{equation} 
\end{defi}
With the entropy condition in hand,  we recall a well known result for local conservation laws: the existence and uniqueness of entropy solutions.
\begin{theo}[Existence, Uniqueness \& Maximum principle for the local conservation law]\label{theo:local_existence_uniqueness_max_principle}
Let \(V\in W^{1,\infty}_{\text{loc}}(\R)\) and the initial condition \(q_{0}\in L^{\infty}(\R)\cap TV(\R)\) hold. Then, there exists a unique, weak entropy solution \(q\in C\left([0,T]; L^1_{\text{loc}}(\R)\right)\cap L^\infty \left((0,T); L^\infty(\R)\right)\) in the sense of \cref{def:entropysol_local} of the local conservation law as stated in \cref{defi:local_conservation_law}, namely \cref{eq:pnorm_local_problem}. In addition, the following maximum principle holds:
\begin{equation*}
    \essinf_{x\in\R} q_{0}(x)\leq q(t,x)\leq \|q_{0}\|_{L^{\infty}(\R)}, \quad (t,x)\in \Omega_T \text{ a.e. }
\end{equation*}
Additionally, if the flux \(f\equiv (\cdot) V(\cdot)\) is strictly concave, uniqueness is guaranteed if \cref{eq:entropy_condition} holds for \textbf{one} strictly convex entropy \(\alpha\) (and corresponding \(\beta\)) only.
\end{theo}

\begin{proof}
    Regarding the existence of a unique weak entropy solution for the local problem~\cref{eq:pnorm_local_problem}, we refer to the classical theory in \cite{bressan, Dafermos2010,kruzkov}. Moreover, the reduction to one strict entropy in the case of strictly concave (or convex) flux has been proven in \cite{Panov1994,DeLellisOttoWestdickenberg}.
\end{proof}
\subsection{The singular limit of the nonlocal conservation law for \texorpdfstring{\(\eta\rightarrow 0\)}{η to zero} in the case of exponential kernels}

In the following, we first show, without the further assumption on the flux, that the nonlocal solution as well as the nonlocal term converge to a weak solution of the local conservation law. As in our previous analysis in \cref{sec:singular_limit}, we restrict ourselves to the exponential kernel in the nonlocal operator with the remark that it might be possible to generalize to a broad class of kernels as in \cite{Marconi2023,keimerpflug23fixed} and as already discussed in \cref{rem:TV_bounds_generalization}.

\begin{theo}[Convergence to weak solutions of the local conservation law]\label{theo:convergence_nonlocal_local_weak}
Let the assumptions of \cref{theo:TV_bounds_Wp} hold and the nonlocal term be of exponential type, i.e., satisfying \cref{ass:exponential}. Then, the solution \(q_{\eta}\) and the nonlocal term \(W_{\eta}\coloneqq W_{\eta}[q,\exp(-\cdot)]\) of the nonlocal conservation law as in \cref{defi:model_class} converge modulo subsequences for any \(p\in\R_{>0}\) to a weak solution (in the sense of \cref{defi:weak_solution_local}) of the corresponding local conservation law as in \cref{defi:local_conservation_law}. In other words, there exists \(q_{*}\in C\big([0,T];L^{1}_{\text{loc}}(\R)\big)\cap L^{\infty}((0,T);L^{\infty}(\R))\) so that \(q_{*}\) satisfies \cref{eq:weak_solution_local} and \((\eta_{k})_{k\in\N}: \lim_{k\rightarrow \infty}\eta_{k}=0\)
\[
\lim_{k\rightarrow\infty} \|q_{*}-q_{\eta_{k}}\|_{C([0,T];L^{1}_{\text{loc}}(\R))}=0=\lim_{k\rightarrow\infty} \|q_{*}-W_{\eta_{k}}[q_{\eta_{k}},\exp(-\cdot)]\|_{C([0,T];L^{1}_{\text{loc}}(\R))}.
\]

\end{theo}
\begin{proof}
    The existence of such \(q_{*}\in C\big([0,T];L^{1}_{\text{loc}}(\R)\big)\cap L^{\infty}\big((0,T);L^{\infty}(\R)\big)\) as the limit point is ensured by \cref{thm:strong_convergence} as well as the convergence of \(q_{\eta},\ W_{\eta}\coloneqq W_{p,\eta}[q_{\eta},\exp(-\cdot)]\) in \(C([0,T];L^{1}(\Omega))\) for \(\Omega\subset\R\) open and bounded. As \(q_{\eta}\) is a weak solution (in the sense of \cref{defi:weak_solution}) of the nonlocal conservation law, it holds for \(\phi\in C^{1}_{\text{c}}((-42,T)\times\R)\) that
    \begin{align*}
              0&=\iint_{\OT}\Big(\phi_{t}(t,x)+V\big(W_{\eta}(t,x)\big)\phi_{x}(t,x)\Big)q_{\eta}(t,x)\dd x\dd t+\int_{\R} q_{0}(x)\phi(0,x)\dd x.
    \intertext{
    Thanks to \(V\) being locally Lipschitz and the strong convergence in \(C([0,T];L^{1}_{\text{loc}}(\R))\) of  \(q_{\eta}, W_{\eta}\) to the same limit point as well as their uniform boundedness, we have}
            0&=\lim_{\eta\rightarrow 0}\iint_{\OT}\Big(\phi_{t}(t,x)+V\big(W_{\eta}(t,x)\big)\phi_{x}(t,x)\Big)q_{\eta}(t,x)\dd x\dd t+\int_{\R} q_{0}(x)\phi(0,x)\dd x\\
              &=\iint_{\OT}\Big(\phi_{t}(t,x)+V\big(q_{*}(t,x)\big)\phi_{x}(t,x)\Big)q_{*}(t,x)\dd x\dd t+\int_{\R} q_{0}(x)\phi(0,x)\dd x.
    \end{align*}
    This indeed matches the definition of a weak solution as stated in \cref{defi:weak_solution_local}. Thanks to \cref{thm:strong_convergence}, \(W_{\eta_{k}}\) converges also to the limit \(q_{*}\) in \(C([0,T];L^{1}_{\text{loc}}(\R)),\) concluding the proof.
\end{proof}
As pointed out before, weak solutions of local conservation laws are not necessarily unique, so the previous result alone does not guarantee convergence to the ``correct'' entropy-admissible local solution. This convergence to the local entropy solution is established in the following under more restrictive assumptions on flux and velocity function. The key idea is due to \cite{bressan2021entropy} where such convergence for the exponential kernel is carried out for \(p=1\). However, in the given setup, the arbitariness of \(p\) makes the analysis more involved. Moreover, we will also give an alternative proof in \cref{thm:entropy_admissibility_nonlocal}, directly working on the nonlocal term and not the solution, similarly to \cite[Theorem 1.2]{Marconi2023} for the case \(p=1\).
\begin{theo}[Entropy admissibility of the nonlocal solution] \label{thm:entropy_admissibility}
    Let the assumptions of \cref{thm:strong_convergence} hold and assume in addition that \([q_{\min},\|q_{0}\|_{L^{\infty}(\R)}]\ni x\mapsto xV(x)\) is strictly concave. 
    Then, for any \(p\in\R_{>0}\), the nonlocal term \(W_{\eta}\) and the corresponding nonlocal solution \(q_\eta\in C\big([0,T];L^1_{\text{loc}}(\R)\big)\) of the nonlocal problem in \cref{eq:pnorm_problem} converge in \(C\big([0,T];L^1_{\text{loc}}(\R)\big)\) to the entropy solution of the corresponding local Cauchy problem in \cref{eq:pnorm_local_problem}.
\end{theo}
\begin{proof}
We consider the entropy condition in \cref{def:entropysol_local} with smooth initial datum and, as a result, smooth solution, which is due to  \cref{lem:approx} being possible and there being no restriction. Then, we will be dealing with classical solutions, and, suppressing the dependency on the smoothing parameter, we obtain for \(\phi\in C^{1}_{\text{c}}((-42,T)\times\R;\R_{\geq 0})\)
\begin{align*}
\mEF[\phi,\alpha,q_{\eta}]&=\iint_{\OT}\alpha\big(q_{\eta}(t,x)\big)\phi_{t}(t,x)+\beta\big(q_{\eta}(t,x)\big)\phi_{x}(t,x)\dd x\dd t+\int_{\R}\alpha\big(q_{0}(x)\big)\phi(0,x)\dd x.\notag
\intertext{An integration by parts, recalling that most boundary evaluations cancel due to the assumption on \(\phi\) being compactly supported in the spatial variable and zero at \(t=T\), yields}
&=-\iint_{\OT}\phi(t,x)\Big(\alpha'\big(q_{\eta}(t,x)\big)\partial_{t}q_{\eta}(t,x)+\beta'\big(q_{\eta}(t,x)\big)\partial_{x}q_{\eta}(t,x)\Big)\dd x\dd t.\notag\\
\intertext{Plugging in the strong form for \(\partial_{t}q_{\eta}\) in \cref{defi:model_class} and using the identity for \(\beta'\) in \cref{eq:entropy_flux_pair} results in}
&=\iint_{\OT}\phi(t,x)\alpha'\big(q_{\eta}(t,x)\big)\partial_{x}\Big(V\big(W_{\eta}(t,x)\big)q_{\eta}(t,x)-V\big(q_{\eta}(t,x)\big)q_{\eta}(t,x)\Big)\dd x\dd t\notag\\
\intertext{for \(W_{\eta}\coloneqq W_{p,\eta}[q_{\eta},\exp(-\cdot)].\) Choosing next, as entropy, \(\alpha\equiv\tfrac{(\cdot)^{mp}}{mp},\ x\in\R,\) for \(mp\in\N_{>1}\), yields}
&=\iint_{\OT}\phi(t,x)q_{\eta}(t,x)^{m p-1}\partial_{x}\Big(q_{\eta}(t,x)\big(V(W_{\eta}(t,x))-V(q_{\eta}(t,x))\big)\Big)\dd x\dd t.\notag
\intertext{Carrying out the derivative in front of the bracket}
&=\iint_{\OT}\phi(t,x)q_{\eta}^{m p-1}(t,x)\partial_{x}q_{\eta}(t,x)\Big(V\big(W_{\eta}(t,x)\big)-V\big(q_{\eta}(t,x)\big)\Big)\dd x\dd t\notag\\
&\quad+ \iint_{\OT}\phi(t,x) q_{\eta}(t,x)^{m p}\partial_{x}\Big(V\big(W_{\eta}(t,x)\big)-V\big(q_{\eta}(t,x)\big)\Big)\dd x\dd t\notag\\
\intertext{and integrating by parts in the first term gives}
&=-\tfrac{1}{mp}\iint_{\OT}\partial_{x}\phi(t,x) q_{\eta}(t,x)^{mp}\Big(V\big(W_{\eta}(t,x)\big)-V\big(q_{\eta}(t,x)\big)\Big)\dd x\dd t\\
&\quad+\tfrac{mp-1}{mp} \iint_{\OT}\phi(t,x)q_{\eta}(t,x)^{mp}\partial_{x}\Big(V\big(W_{\eta}(t,x)\big)-V\big(q_{\eta}(t,x)\big)\Big)\dd x\dd t.
\intertext{Carrying out the derivative in the second term and splitting up yields}
&=-\tfrac{1}{mp}\iint_{\OT}\partial_{x}\phi(t,x)q_{\eta}(t,x)^{mp}\Big(V\big(W_{\eta}(t,x)\big)-V\big(q_{\eta}(t,x)\big)\Big)\dd x\dd t\\
&\quad+\tfrac{mp-1}{mp} \iint_{\OT}\phi(t,x)q_{\eta}(t,x)^{mp} V'\big(W_{\eta}(t,x)\big)\partial_x W_{\eta}(t,x) \dd x\dd t\\
&\quad-\tfrac{mp-1}{mp}\iint_{\OT} \phi(t,x)q_{\eta}(t,x)^{mp} V'\big(q_{\eta}(t,x)\big)\partial_{x}q_{\eta}(t,x)\dd x\dd t
\intertext{and, applying the well-used identity \( q^p_{\eta}\equiv W_{\eta}^p-\eta W_{\eta}^{p-1} \partial_2 W_{\eta}\) (see \cref{eq:nonlocal_identity}) in the second term, we have}
&=-\tfrac{1}{mp}\iint_{\OT}\partial_{x}\phi(t,x)q_{\eta}(t,x)^{mp}\Big(V\big(W_{\eta}(t,x)\big)-V\big(q_{\eta}(t,x)\big)\Big)\dd x\dd t\\
&\quad+\tfrac{mp-1}{mp} \iint_{\OT}\phi(t,x)q_{\eta}(t,x)^{mp-p} W_{\eta}(t,x)^p V'\big(W_{\eta}(t,x)\big)\partial_x W_{\eta}(t,x) \dd x\dd t\\
&\quad-\tfrac{\eta (mp-1)}{m p} \!\!\!\iint_{\OT}\!\!\!\!\!\phi(t,x)q_{\eta}(t,x)^{mp-p} \Big(W_{\eta}(t,x)^{p-1} \partial_x W_{\eta}(t,x)\Big) V'(W_{\eta}(t,x))\partial_x W_{\eta}(t,x) \dd x\dd t\\
&\quad-\tfrac{mp-1}{mp}\iint_{\OT} \phi(t,x)q_{\eta}(t,x)^{mp} V'(q_{\eta}(t,x))\partial_{x}q_{\eta}(t,x)\dd x\dd t
\intertext{leaving off the third term as it is positive if \(mp-1\geq 0\) recalling that \(V'\leqq 0\)}
&\geq -\tfrac{1}{mp}\iint_{\OT}\partial_{x}\phi(t,x)q_{\eta}(t,x)^{mp}\Big(V\big(W_{\eta}(t,x)\big)-V\big(q_{\eta}(t,x)\big)\Big)\dd x\dd t\\
&\quad+\tfrac{mp-1}{mp} \iint_{\OT}\phi(t,x)q_{\eta}(t,x)^{mp-p} W_{\eta}(t,x)^p V'\big(W_{\eta}(t,x)\big)\partial_x W_{\eta}(t,x) \dd x\dd t\\
&\quad-\tfrac{mp-1}{mp}\iint_{\OT} \phi(t,x)q_{\eta}(t,x)^{mp} V'\big(q_{\eta}(t,x)\big)\partial_{x}q_{\eta}(t,x)\dd x\dd t
\intertext{and repeating these steps by plugging in the aforementioned and used identity \( q^p_{\eta}\equiv W_{\eta}^p-\eta W_{\eta}^{p-1} \partial_2 W_{\eta}\) and leaving out in each step the positive term}
&\geq -\tfrac{1}{mp}\iint_{\OT}\partial_{x}\phi(t,x)q_{\eta}(t,x)^{mp}\Big(V\big(W_{\eta}(t,x)\big)-V\big(q_{\eta}(t,x)\big)\Big)\dd x\dd t\\
&\quad+\tfrac{mp-1}{mp} \iint_{\OT}\phi(t,x) W_{\eta}(t,x)^{mp}  V'\big(W_{\eta}(t,x)\big)\partial_x W_{\eta}(t,x) \dd x\dd t\\
&\quad-\tfrac{\eta(mp-1)}{mp} \iint_{\OT}\phi(t,x) W_{\eta}(t,x)^{mp-1} \big(\partial_x W_{\eta}(t,x)\big)^{2} V'\big(W_{\eta}(t,x)\big) \dd x\dd t\\
&\quad-\tfrac{mp-1}{mp}\iint_{\OT} \phi(t,x)q_{\eta}(t,x)^{mp} V'\big(q_{\eta}(t,x)\big)\partial_{x}q_{\eta}(t,x)\dd x\dd t.\\
\intertext{Finally, leaving out again the remaining last positive term and noticing that the second term is entirely in \(W_{\eta},\partial_{2}W_{\eta}\) and the third one entirely in \(q_{\eta},\ \partial_{2}q_{\eta}\) we have with \(F(\cdot)\coloneqq \int_{0}^{\cdot}s^{mp}V'(s)\dd s\)}
&\geq -\tfrac{1}{mp}\iint_{\OT}\partial_{x}\phi(t,x)q_{\eta}(t,x)^{mp}\Big(V\big(W_{\eta}(t,x)\big)-V\big(q_{\eta}(t,x)\big)\Big)\dd x\dd t\\
&\quad+\tfrac{mp-1}{mp} \iint_{\OT} \phi(t,x) \partial_{x}F(W_{\eta}(t,x)) \dd x\dd t-\tfrac{mp-1}{mp}\iint_{\OT} \phi(t,x) \partial_{x}F\big(q_{\eta}(t,x)\big) \dd x\dd t.
\intertext{Integrating by parts yields }
&=-\tfrac{1}{mp}\iint_{\OT}\partial_{x}\phi(t,x)q_{\eta}^{mp}(t,x)\Big(V\big(W_{\eta}(t,x)\big)-V\big(q_{\eta}(t,x)\big)\Big)\dd x\dd t\\
&\quad-\tfrac{mp-1}{mp} \iint_{\OT} \phi_{x}(t,x)\Big(F\big(W_{\eta}(t,x)\big)-F\big(q_{\eta}(t,x)\big)\Big) \dd x\dd t\\
&\overset{\eta\rightarrow 0}{=}0
\end{align*}
where the last implication is due to \(\phi\) having compact support, the strong convergence  of \(W_{\eta},\ q_{\eta}\) in \(C\big([0,T];L^{1}_{\text{loc}}(\R)\big)\) to the same limit point according to \cref{thm:strong_convergence} and the Lipschitz continuity of \(V,F\). Altogether, this gives the entropy admissibility of the nonlocal solution in the limit \(\eta\rightarrow 0\) for the chosen entropy \(\alpha\equiv (\cdot)^{mp}\) which is strictly convex on \([q_{\min},\|q_{0}\|_{L^{\infty}(\R)}]\) for \(mp>1\) as assumed, so that we can apply \cref{theo:local_existence_uniqueness_max_principle} and have \(q^{*}\in C([0,T];L^{1}_{\text{loc}}(\R)),\) a limit point of a subsequence of \(q_{\eta},\ W_{\eta}\) respectively, be a weak entropy solution in the sense of \cref{def:entropysol_local}. As weak entropy solutions are unique, we can afterwards argue that the convergence to the local entropy solution holds for all subsequence.
\end{proof}
\begin{rem}[Comments on the entropy admissibility]
The entropy admissibility for \(p\in(0,1)\) in \cref{thm:entropy_admissibility} is due to the introduction of \(m\in\N_{\geq 0}\) in the choice of strict convex entropy and the choice of \(mp\geq 1\). When only looking for \(p\geq 1\), we could have worked with the entropy \(\alpha(\cdot)\equiv (\cdot)^{p}\).
\end{rem}

In the following theorem, we give an alternative approach to showing entropy admissibility by focusing not on the solution \(q_{\eta}\) but rather on the nonlocal operator \(W_{\eta}[q_{\eta},\exp(-\cdot)]\) directly, however, only for \(p\in(1,\infty)\).
\begin{theo}[Entropy admissibility of the nonlocal operator] \label{thm:entropy_admissibility_nonlocal}
    Let the assumptions of \cref{thm:strong_convergence} hold and assume in addition that \([q_{min},\|q_{0}\|_{L^{\infty}(\R)}]\ni x\mapsto xV(x)\) is strictly concave and \(p\in(1,\infty)\). Then, the nonlocal term \(W_{\eta}\) and the corresponding nonlocal solution \(q_\eta\in C\big([0,T];L^1_{\text{loc}}(\R;\R)\big)\) of the nonlocal problem \cref{eq:pnorm_problem} are entropy admissible for the entropy \(\alpha\equiv (\cdot)^{p}\) in the limit \(\eta\rightarrow0\), i.e.,
    \[\lim_{\eta\rightarrow 0}\mathcal{EF}[\varphi, (\cdot)^{p}, W_{\eta}]\geq 0 \quad \forall \varphi \in C^1_{\text{c}}\left((-42,T)\times \R; \R_{\geq 0}\right),\  \text{with \(\mathcal{EF}\) as in \cref{eq:entropy_flux_pair}}.
    \]
\end{theo}
\begin{proof}
Similar to the proof of \cref{thm:entropy_admissibility_nonlocal}, we show the entropy admissibility of the nonlocal operator, i.e., the following entropy admissibility (see \cref{eq:entropy_condition})
\begin{align*}
        \mathcal{EF}[\varphi, \alpha, W_{\eta}]\coloneqq\iint_{\Omega_T} \alpha\big(W_{\eta}(t,x)\big) \varphi_t(t,x)+ \beta\big(W_{\eta}(t,x)\big) \varphi_x(t,x) \dd x \dd t + \int_\R \alpha\big(W_{\eta}(0,x)\big) \varphi(0,x) \dd x\geq 0.
\end{align*}
Integrating by parts yields, for the entropy \(\alpha'(\cdot)\coloneqq (\cdot)^{p-1}\) after a multiplication with \(p\), 
\begin{align}
    p\mathcal{EF}[\varphi,\alpha,W_{\eta}]&=-\iint_{\OT} \Big(\partial_{1}W_{\eta}^{p}+ pW_{\eta}^{p-1}\big(V(W_{\eta})+V'(W_{\eta})W_{\eta}\big) \partial_{2}W_{\eta}\Big)\phi\dd x\dd t\notag
    \intertext{and plugging in the expression for \(\partial_{1}(W^{p})\) in \cref{rem:PDE_nonlocal_p}}
    &=\iint_{\OT}\!\!\! \bigg(V(W_{\eta})\partial_{2}W^{p}_{\eta}+\tfrac{p^{2}}{\eta}\int_{x}^{\infty}\!\!\!\!\e^{\frac{px-py}{\eta}}V'(W_{\eta}(t,y))\partial_{y}W_{\eta}(t,y) W_{\eta}(t,y)^{p}\dd y-V(W_{\eta})\partial_{2}W_{\eta}^{p}\notag\\
    &\quad  -p(p-1)\!\int_{x}^{\infty}\!\!\!\!\e^{\frac{px-py}{\eta}}V'(W_{\eta}(t,y))W_{\eta}(t,y)^{p-1}\big(\partial_{y}W_{\eta}(t,y)\big)^{2}\dd y -pW_{\eta}^{p}V'(W_{\eta})\partial_{2}W_{\eta} \bigg)\phi\dd x\dd t\notag
    \intertext{which is equivalent to}
    &=\iint_{\OT} \bigg(\tfrac{p^{2}}{\eta}\int_{x}^{\infty}\exp(\tfrac{px-py}{\eta})V'(W_{\eta}(t,y))\partial_{y}W_{\eta}(t,y) W_{\eta}(t,y)^{p}\dd y\notag\\
    &\quad - p(p-1)\!\!\int_{x}^{\infty}\!\!\!\e^{\frac{px-py}{\eta}}V'(W_{\eta}(t,y))W_{\eta}(t,y)^{p-1}\big(\partial_{y}W_{\eta}(t,y)\big)^{2}\!\dd y -pW_{\eta}^{p}V'(W_{\eta})\partial_{2}W_{\eta} \bigg)\phi\dd x\dd t\notag
    \intertext{and, with the definition \(F\equiv p\int_{0}^{\cdot} s^{p}V'(s)\dd s\), we can rewrite the expression to}
    &= \iint_{\OT} \bigg(- \tfrac{\dd}{\dd x}F(W_{\eta}(t,x)) +\tfrac{p}{\eta}\int_{x}^{\infty}\exp(\tfrac{px-py}{\eta})\tfrac{\dd}{\dd y}F(W_{\eta}(t,y))\dd y\notag\\
    &\qquad -p(p-1)\int_{x}^{\infty}\!\!\!\!\!\!\exp(\tfrac{px-py}{\eta})V'(W_{\eta}(t,y))W_{\eta}(t,y)^{p-1}\big(\partial_{y}W_{\eta}(t,y)\big)^{2}\dd y  \bigg)\phi\dd x\dd t.\notag
    \intertext{Integration by parts in the first term and changing the order of integration in the second term gives}
    &= \iint_{\OT}\partial_{2}\phi F(W_{\eta})\dd x\dd t +\tfrac{p}{\eta}\iint_{\OT}\tfrac{\dd}{\dd y}F(W_{\eta}(t,y)) \int_{-\infty}^{y}\exp(\tfrac{px-py}{\eta})\phi\dd x\dd y\dd t\notag\\
    &\qquad  -p(p-1)\iint_{\OT}\int_{x}^{\infty}\!\!\!\!\!\!\exp(\tfrac{px-py}{\eta})V'(W_{\eta}(t,y))W_{\eta}(t,y)^{p-1}\big(\partial_{y}W_{\eta}(t,y)\big)^{2}\dd y \phi\dd x\dd t\notag
    \intertext{and, integrating by parts in the inner integration of the second term,}
    &= \iint_{\OT}\partial_{2}\phi F(W_{\eta})\dd x\dd t +\iint_{\OT}\tfrac{\dd}{\dd y}F(W_{\eta}) \phi - \tfrac{\dd}{\dd y}F(W_{\eta}) \int_{-\infty}^{y}\exp(\tfrac{px-py}{\eta})\partial_{x}\phi\dd x\dd y\dd t\notag\\
    &\qquad - p(p-1)\iint_{\OT}\int_{x}^{\infty}\!\!\exp(\tfrac{px-py}{\eta})V'(W_{\eta})W_{\eta}^{p-1}\big(\partial_{y}W_{\eta}(t,y)\big)^{2}\dd y \phi\dd x\dd t\notag
    \intertext{which is, after another integration by parts in the second term, equivalent to}
    &= -\iint_{\OT} \tfrac{\dd}{\dd y}F(W_{\eta}) \int_{-\infty}^{x}\exp(\tfrac{px-py}{\eta})\partial_{x}\phi\dd x\dd y\dd t\label{eq:entropy_admissibility_W_1}\\
    &\qquad - p(p-1)\iint_{\OT}\int_{x}^{\infty}\!\!\exp(\tfrac{px-py}{\eta})V'(W_{\eta})W_{\eta}^{p-1}\big(\partial_{y}W_{\eta}(t,y)\big)^{2}\dd y \phi\dd x\dd t.\label{eq:entropy_admissibility_W_2}
\end{align}
However, as \(F\in W^{1,\infty}_{\text{loc}}(\R),\) \(0\leqq W_{\eta}\leqq \|q_{0}\|_{L^{\infty}(\R)},\) we can estimate the first term as follows
\[
\bigg|\iint_{\OT} \tfrac{\dd}{\dd y}F(W) \int_{-\infty}^{x}\exp(\tfrac{px-py}{\eta})\partial_{x}\phi \dd x\dd y\dd t\bigg|\leq \big\|F'\big\|_{L^{\infty}((0,q_{\max}))}|W_{\eta}|_{L^{\infty}((0,T);TV(\R))}\|\partial_2\phi\|_{L^{\infty}(\OT)}\tfrac{\eta}{p}.
\]
Thanks to the uniform \(TV\) bounds on \(W_{\eta}\) as explicated in \cref{rem:total_variation_bound_W}, the first term in\cref{eq:entropy_admissibility_W_1} converges to zero for \(\eta\rightarrow 0\), while the second term in \cref{eq:entropy_admissibility_W_2} is, by construction, non-negative as long as \(p\geq 1\). Thus, we have indeed obtained that \(\lim_{\eta\rightarrow 0}\mathcal{EF}[\varphi, (\cdot)^{p}, W_{\eta}]\geq 0\) as claimed.
\end{proof}
\begin{rem}[The case \(p=1\)]
Note that we have excluded \(p=1\) in \cref{thm:entropy_admissibility_nonlocal} although the calculations go through there as well. However, \(x\mapsto x\) is not a \textbf{strictly} convex entropy, preventing us from applying \cref{theo:local_existence_uniqueness_max_principle}. 

Interestingly, the entropy admissibility for the case \(p>1\) of the nonlocal operator is quite straightforward and significantly easier than for \(p=1\) as was carried out in \cite{Marconi2023,Colombo2024}.
\end{rem}

We finalize this section by uniting the previous statements into one theorem:
\begin{theo}[Convergence of nonlocal solutions to local entropy solutions]
Let \cref{ass:input_datum_and_velocity} hold and either
\begin{itemize}
    \item \(p\in\R_{\geq 1}\)
    \end{itemize}
    or
    \begin{itemize}
    \item \(p<1\) and let the initial datum satisfy \(\tfrac{\essinf_{x\in\R}q_{0}(x)}{\|q_{0}\|_{L^{\infty}(\R)}}\geq \sqrt[p]{1-p}\).
\end{itemize}
In addition, assume that
\begin{itemize}
\item the nonlocal operator is of exponential type, i.e., satisfies \cref{ass:exponential}
\item \([q_{\min},\|q_{0}\|_{L^{\infty}(\R)}]\ni x\mapsto xV(x)\) is strictly concave.
\end{itemize}
Then, the nonlocal solution \(q_{\eta}\in C([0,T];L^{1}_{\text{loc}}(\R))\) as well as the nonlocal term \(W_{p,\eta}[q_{\eta},\exp(-\cdot)]\) as in \cref{defi:model_class}, converge in \(C([0,T];L^{1}_{\text{loc}}(\R))\) to the local entropy solution \(q_{*}\in C([0,T];L^{1}_{\text{loc}}(\R))\) of \cref{defi:local_conservation_law}, i.e., for every \(\Omega\subset\R\) open bounded it holds that
\[
\lim_{\eta\rightarrow 0}\|q_{\eta}-q_{*}\|_{C([0,T];L^{1}(\Omega))}=0=\lim_{\eta\rightarrow 0}\|W_{p,\eta}[q_{\eta},\exp(-\cdot)]-q_{*}\|_{C([0,T];L^{1}(\Omega))}.
\]
\end{theo}
\begin{proof}
    This is a consequence of \cref{thm:entropy_admissibility} and \cref{theo:convergence_nonlocal_local_weak}, based on \cref{cor:compactnessWp}.
\end{proof}
\begin{rem}[Strict concavity/convexity of the flux \(x\mapsto xV(x)\)]
    Note that the assumed strict concavity of \(x\mapsto xV(x)\) on \(x\in[q_{\min},\|q_{0}\|_{L^{\infty}(\R)}]\) can be replaced by the strict convexity, and the results in \cite{Panov1994} are still applicable. However, as we assume in this contribution that \(V'\leqq 0\) having traffic flow modelling in mind, this condition is unclassical in traffic, and indeed, when looking at the convexity assumption we would require that
    \[
    x\mapsto 2V'(x)+xV''(x)>0,\ x\in[q_{\min},\|q_{0}\|_{L^{\infty}(\R)}]
    \]
    which is generally more restrictive particularly when \(q_{\min}\) is small as the only term compensating for the negative sign of \(V'\) is \(V''\). This requires also at least that \(V''>0\).
\end{rem}
\subsection{\texorpdfstring{Ole\u inik}{Oleinik} estimates and further singular limit convergence}\label{subsec:Oleinik}
In this section, we shortly look into Ole\u inik estimates for the nonlocal operator. A similar approach for \(p=1\) had been taken in \cite{coclite2023oleinik} and in the present case, we obtain similar estimates for arbitary \(p\in(1,\infty)\) which become more restrictive the larger \(p\) becomes. We also obtain estimates for \(p<1\):
\begin{lem}[Ole\u inik-type inequality for \(W^p_{\eta}\)] \label{thm:oleinik}
Let \ref{ass:input_datum_and_velocity} and let the exponential kernel as in \cref{ass:exponential} be given. 
\begin{itemize}
\item If \(p>1,\) assume furthermore that it holds
\begin{enumerate}
        \item \(V''(\cdot)(\cdot)-(p-2)V'(\cdot)\geq 0\) on \([q_{\min},\|q_{0}\|_{L^{\infty}(\R)}]\)\label{item:1}
        \item \(\essinf_{x\in(q_{\min},\|q_{0}\|_{L^{\infty}(\R)})}|V'(x)x^{1-p}|-\tfrac{1}{p}\Big\|\tfrac{V''(\cdot)(\cdot)-(p-2)V'(\cdot)}{(\cdot)^{p-1}}\Big\|_{L^{\infty}((q_{\min},\|q_{0}\|_{L^{\infty}(\R)}))}\eqqcolon\kappa>0\).\label{item:2}
\end{enumerate}
\item If \(0<p<1,\) assume furthermore that the following holds
\begin{enumerate}
        \item \(V''(\cdot)(\cdot)+V'(\cdot)\geq 0\) on \([q_{\min},\|q_{0}\|_{L^{\infty}(\R)}]\)\label{item:2_1}
        \item \(\essinf_{x\in(q_{\min},\|q_{0}\|_{L^{\infty}(\R)})}
 \left|\tfrac{V'(x)}{x^{p-1}}\right|+\tfrac{1}{p}\left\|\tfrac{V''(\cdot)(\cdot)+V'(\cdot)}{(\cdot)^{p-1}}\right\|_{L^{\infty}((q_{\min},\|q_{0}\|_{L^{\infty}(\R)}))} 
\eqqcolon\kappa>0\).\label{item:2_2}
\end{enumerate}
\end{itemize}
Then, the nonlocal term $W_{\eta}\coloneqq W_{p,\eta}[q_{\eta},\exp(-\cdot)]$ satisfies the following inequality: 
\[
    \tfrac{W_{\eta}^p(t,x)-W_{\eta}^p(t,y)}{x-y}\geq -\tfrac{1}{\kappa t}, \quad \forall (x,y)\in\R^{2},\ x\neq y,\ t\in[0,T]
\]
implying for \(p>1\)
\[
  \tfrac{W_{\eta}(t,x)-W_{\eta}(t,y)}{x-y}\geq -\tfrac{1}{p q_{\max}^{p-1}\kappa t} \quad \forall (x,y)\in\R^{2},\ x\neq y,\ t\in[0,T]
\]
and for \(p<1\)
\[
  \tfrac{W_{\eta}(t,x)-W_{\eta}(t,y)}{x-y}\geq -\tfrac{1}{p q_{\min}^{p-1}\kappa t} \quad \forall (x,y)\in\R^{2},\ x\neq y,\ t\in[0,T].
\]
\end{lem}
\begin{proof}
We take the nonlocal equation in \(W^{p}\), i.e.,\ \cref{eq:nonlocal_W_p}
\begin{align*}
\partial_{t}W^{p}_{\eta}(t,x)&=-V(W_{\eta}(t,x))\partial_{x}W^{p}_{\eta}(t,x)\\
        &\quad -\tfrac{p^{2}}{\eta}\int_{x}^{\infty}\exp(\tfrac{px-py}{\eta})V'(W_{\eta}(t,y))\partial_{y}W_{\eta}(t,y)W_{\eta}(t,y)^{p}\dd y\\
    &\quad +p(p-1)\int_{x}^{\infty}\!\!\!\!\!\!\exp(\tfrac{px-py}{\eta})V'(W_{\eta}(t,y))W_{\eta}^{p-1}(t,y)\big(\partial_{y}W_{\eta}(t,y)\big)^{2}\dd y
\end{align*}
and its spatial derivative, which yields
\begin{align*}
\partial_{t}\partial_{x}W^{p}_{\eta}(t,x)&=-\tfrac{\dd}{\dd x}\Big(V(W_{\eta}(t,x))\partial_{x}W_{\eta}^{p}(t,x)\Big)+\tfrac{p^{2}}{\eta}V'(W_{\eta}(t,x))\partial_{x}W_{\eta}(t,x)W_{\eta}(t,x)^{p}\\
&\quad-\tfrac{p^{3}}{\eta^{2}}\int_{x}^{\infty}\!\!\!\!\!\exp(\tfrac{px-py}{\eta})V'(W_{\eta}(t,y))\partial_{y}W_{\eta}(t,y)W_{\eta}(t,y)^{p}\dd y\\
    &\quad-p(p-1)V'(W_{\eta}(t,x))W_{\eta}(t,x)^{p-1}\big(\partial_{x}W_{\eta}(t,x)\big)^{2}\\
    &\quad +\tfrac{p^{2}(p-1)}{\eta}\int_{x}^{\infty}\!\!\!\!\!\!\exp(\tfrac{px-py}{\eta})V'(W_{\eta}(t,y))W_{\eta}^{p-1}(t,y)\big(\partial_{y}W_{\eta}(t,y)\big)^{2}\dd y\\
    &=-V(W_{\eta}(t,x))\partial_{x}^{2}W_{\eta}^{p}-V'(W_{\eta}(t,x))\partial_{x}W_{\eta}\partial_{x}W_{\eta}^{p}+\tfrac{p}{\eta}V'(W_{\eta}(t,x))W_{\eta}(t,x)\partial_{x}W_{\eta}^{p}\\
&\quad-\tfrac{p^{2}}{\eta^{2}}\int_{x}^{\infty}\!\!\!\!\!\exp(\tfrac{px-py}{\eta})V'(W_{\eta}(t,y))W_{\eta}(t,y)\partial_{y}W_{\eta}(t,y)^{p}\dd y\\
    &\quad-(p-1)V'(W_{\eta}(t,x))\partial_{x}W_{\eta}^{p}\partial_{x}W_{\eta}\\
    &\quad +\tfrac{p(p-1)}{\eta}\int_{x}^{\infty}\!\!\!\!\!\!\exp(\tfrac{px-py}{\eta})V'(W_{\eta}(t,y))\partial_{y}W_{\eta}^{p}(t,y)\partial_{y}W_{\eta}(t,y)\dd y.
    &\intertext{Assuming that we are minimal at \((t,x)\),\ i.e. \(\partial_{x}^{2}W^{p}(t,x)=0\) and \(\partial_{x}W^{p}(t,x)<0\)}
    &=-pV'(W_{\eta}(t,x))\partial_{x}W_{\eta}\partial_{x}W_{\eta}^{p}+\tfrac{p}{\eta}V'(W_{\eta}(t,x))W_{\eta}(t,x)\partial_{x}W_{\eta}^{p}\\
    &\quad-\tfrac{p}{\eta^{2}}\int_{x}^{\infty}\!\!\!\!\!\exp(\tfrac{px-py}{\eta})V'(W_{\eta}(t,y))\partial_{y}W_{\eta}(t,y)^{p}\Big(pW_{\eta}-(p-1)\eta \partial_{y}W\Big)\dd y
    \intertext{and as \(p>1\) and \(pW_{\eta}-(p-1)\eta \partial_{y}W_{\eta}\geq 0\)}
    &\geq -pV'(W_{\eta}(t,x))\partial_{x}W_{\eta}\partial_{x}W_{\eta}^{p}+\tfrac{p}{\eta}V'(W_{\eta}(t,x))W_{\eta}\partial_{x}W_{\eta}^{p}\\
    &\quad-\tfrac{p}{\eta^{2}}\partial_{x}W_{\eta}^{p}\int_{x}^{\infty}\!\!\!\!\!\exp(\tfrac{px-py}{\eta})V'(W_{\eta}(t,y))\Big(pW_{\eta}-(p-1)\eta \partial_{y}W\Big)\dd y\\
    &= -pV'(W_{\eta}(t,x))\partial_{x}W_{\eta}\partial_{x}W_{\eta}^{p}+\tfrac{p}{\eta}V'(W_{\eta}(t,x))W_{\eta}\partial_{x}W_{\eta}^{p}\\
    &\quad-\tfrac{p^{2}}{\eta^{2}}\partial_{x}W_{\eta}^{p}\int_{x}^{\infty}\!\!\!\!\!\exp(\tfrac{px-py}{\eta})V'(W_{\eta}(t,y))W_{\eta}\dd y\\
    &\quad+\tfrac{p(p-1)}{\eta}\partial_{x}W_{\eta}^{p}\int_{x}^{\infty}\!\!\!\!\!\exp(\tfrac{px-py}{\eta})V'(W_{\eta}(t,y))\partial_{y}W\dd y.
    \intertext{Integrating by parts in the first integral term}
    &= -pV'(W_{\eta}(t,x))\partial_{x}W_{\eta}\partial_{x}W_{\eta}^{p}\\
    &\quad-\tfrac{p}{\eta}\partial_{x}W_{\eta}^{p}\int_{x}^{\infty}\!\!\!\!\!\exp(\tfrac{px-py}{\eta})\partial_{y}\big(V'(W_{\eta}(t,y))W_{\eta}\big)\dd y\\
    &\quad+\tfrac{p(p-1)}{\eta}\partial_{x}W_{\eta}^{p}\int_{x}^{\infty}\!\!\!\!\!\exp(\tfrac{px-py}{\eta})V'(W_{\eta}(t,y))\partial_{y}W\dd y\\
    &= -pV'(W_{\eta}(t,x))\partial_{x}W_{\eta}\partial_{x}W_{\eta}^{p}\\
    &\quad-\tfrac{p}{\eta}\partial_{x}W_{\eta}^{p}\int_{x}^{\infty}\!\!\!\!\!\exp(\tfrac{px-py}{\eta})W_{\eta,y}\big(V'(W_{\eta}(t,y))+V''(W_{\eta})W_{\eta}-(p-1)V'(W_{\eta})\big)\dd y\\
    &= -pV'(W_{\eta}(t,x))\partial_{x}W_{\eta}\partial_{x}W_{\eta}^{p}-\tfrac{p}{\eta}\partial_{x}W_{\eta}^{p}\int_{x}^{\infty}\!\!\!\!\!\e^{\frac{px-py}{\eta}}\partial_{y}W_{\eta}\big(V''(W_{\eta})W_{\eta}-(p-2)V'(W_{\eta})\big)\dd y\\
    &= \tfrac{-V'(W_{\eta}(t,x))}{W_{\eta}^{p-1}}\Big(\partial_{x}W_{\eta}^{p}\Big)^{2}-\tfrac{1}{\eta}\partial_{x}W_{\eta}^{p}\int_{x}^{\infty}\!\!\!\!\!\e^{\frac{px-py}{\eta}}\tfrac{\partial_{y}(W_{\eta}^{p})}{W^{p-1}_{\eta}}\big(V''(W_{\eta})W_{\eta}-(p-2)V'(W_{\eta})\big)\dd y
    \intertext{and for \(V''(\cdot)(\cdot)-(p-2)V'(\cdot)\geq 0\) and \(I(q_{0})\coloneqq (q_{\min},\|q_{0}\|_{L^{\infty}(\R)})\subset\R\)}
    &\geq \tfrac{-V'(W_{\eta}(t,x))}{W_{\eta}^{p-1}}\Big(\partial_{x}W_{\eta}^{p}\Big)^{2}-\tfrac{1}{\eta}\Big(\partial_{x}W_{\eta}^{p}\Big)^{2}\!\!\!\int_{x}^{\infty}\!\!\!\!\!\e^{\frac{px-py}{\eta}}\tfrac{1}{W^{p-1}_{\eta}}\big(V''(W_{\eta})W_{\eta}-(p-2)V'(W_{\eta})\big)\dd y\\
    &\geq \tfrac{-V'(W_{\eta}(t,x))}{W_{\eta}^{p-1}}\Big(\partial_{x}W_{\eta}^{p}\Big)^{2}-\tfrac{1}{p}\Big(\partial_{x}W_{\eta}^{p}\Big)^{2}\Big\|\tfrac{V''(\cdot)(\cdot)-(p-2)V'(\cdot)}{(\cdot)^{p-1}}\Big\|_{L^{\infty}(I(q_{0}))}\\
&\geq\Big(\partial_{x}W_{\eta}^{p}\Big)^{2} \essinf_{x\in I(q_{0})}|V'(x)x^{1-p}|-\tfrac{1}{p}\Big(\partial_{x}W_{\eta}^{p}\Big)^{2}\Big\|\tfrac{V''(\cdot)(\cdot)-(p-2)V'(\cdot)}{(\cdot)^{p-1}}\Big\|_{L^{\infty}(I(q_{0}))}\\
&=\Big(\partial_{x}W_{\eta}^{p}\Big)^{2}\bigg(\essinf_{x\in I(q_{0})}|V'(x)x^{1-p}|-\tfrac{1}{p}\big\|\tfrac{V''(\cdot)(\cdot)-(p-2)V'(\cdot)}{(\cdot)^{p-1}}\big\|_{L^{\infty}(I(q_{0}))} \bigg).
\end{align*}
        Thus, whenever we have
        \begin{enumerate}
        \item \(V''(\cdot)(\cdot)-(p-2)V'(\cdot)\geq 0\) on \([q_{\min},\|q_{0}\|_{L^{\infty}(\R)}]=I(q_{0})\)
        \item \(\essinf_{x\in(q_{\min},\|q_{0}\|_{L^{\infty}(\R)})}|V'(x)x^{1-p}|-\tfrac{1}{p}\Big\|\tfrac{V''(\cdot)(\cdot)-(p-2)V'(\cdot)}{(\cdot)^{p-1}}\Big\|_{L^{\infty}((q_{\min},\|q_{0}\|_{L^{\infty}(\R)}))}>0\),
        \end{enumerate}
        we have an Ole\u inik estimate as stated for \(p>1\).
        
        Let us analyze now the case $0<p<1.$ Taking advantage of \cref{eq:partial_t_x_W_p},
        we assume that we are minimal at \((t,x)\),\ i.e. \(\partial_{x}^{2}W^{p}(t,x)=0\) and \(\partial_{x}W^{p}(t,x)<0\). Then, we can estimate
\begin{align*}
    \partial_{t}\partial_{x}W^{p}_{\eta}(t,x)&=-V'(W)W_{x}\partial_{x}W^{p}(t,x)+\tfrac{p}{\eta}V'(W)W\partial_{x}W^{p}-(p-1)V'(W)\partial_{x}W^{p}\partial_{x}W\\
&\ -\tfrac{p^{2}}{\eta^{2}}\int_{x}^{\infty}\exp(\tfrac{px-py}{\eta})V'(W)\partial_{y}W^{p}W\dd y  \\
&\ +\tfrac{p(p-1)}{\eta}\int_{x}^{\infty}\exp(\tfrac{px-py}{\eta})V'(W)\partial_{y}W^{p}W_{y}\dd y  \\
&\geq -pV'(W)W_{x}\partial_{x}W^{p}(t,x)+\tfrac{p}{\eta}V'(W)W\partial_{x}W^{p} -\tfrac{p^{2}}{\eta^{2}}\partial_{x}W^{p}\int_{x}^{\infty}\exp(\tfrac{px-py}{\eta})V'(W)W\dd y  \\
&=-pV'(W)W_{x}\partial_{x}W^{p}(t,x)+\tfrac{p}{\eta}V'(W)W\partial_{x}W^{p} -\tfrac{p}{\eta}\partial_{x}W^{p}V'(W)W\\
&\quad -\tfrac{p}{\eta}\partial_{x}W^{p}\int_{x}^{\infty}\exp(\tfrac{px-py}{\eta})\tfrac{\dd}{\dd y}\big(V'(W)W)\dd y\\
&=-pV'(W)W_{x}\partial_{x}W^{p}(t,x)-\tfrac{p}{\eta}\partial_{x}W^{p}\int_{x}^{\infty}\exp(\tfrac{px-py}{\eta})\big(V''(W)W+V'(W))W_{y}\dd y.
\intertext{If \(V''(x)x+V'(x)\geq 0,\)}
&\geq -pV'(W)W_{x}\partial_{x}W^{p}(t,x)-\tfrac{1}{\eta}(\partial_{x}W^{p})^{2}\int_{x}^{\infty}\exp(\tfrac{px-py}{\eta})\tfrac{V''(W)W+V'(W)}{W^{p-1}}\dd y\\
&\geq -\tfrac{V'(W)}{W^{p-1}}\big(\partial_{x}W^{p}(t,x)\big)^{2}-\tfrac{1}{p}(\partial_{x}W^{p})^{2}\Big\|\tfrac{V''(W)W+V'(W)}{W^{p-1}}\Big\|_{L^{\infty}((q_{\min},\|q_{0}\|_{L^{\infty}(\R)}))}.
\end{align*} 
Thus, whenever we have
        \begin{enumerate}
        \item \(V''(\cdot)(\cdot)+V'(\cdot)\geq 0,\)
        \item \(\essinf_{x\in(q_{\min},\|q_{0}\|_{L^{\infty}(\R)})}
 \left|\tfrac{V'(x)}{x^{p-1}}\right|+\tfrac{1}{p}\Big\|\tfrac{V''(\cdot)(\cdot)+V'(\cdot)}{(\cdot)^{p-1}}\Big\|_{L^{\infty}((q_{\min},\|q_{0}\|_{L^{\infty}(\R)}))} 
\eqqcolon\kappa>0\),
\end{enumerate}
        we have an Ole\u inik estimate as stated for \(0<p<1\).\\
        \end{proof}
\begin{rem}[Examples of eligibility]
Consider a linear velocity, i.e., \(V(x)=1-x,\ x\in[0,1]\). Then, \cref{item:1} is satisfied as long as \(p\geq 2\). \cref{item:2} reduces to
\[
\essinf_{x\in(q_{\min},\|q_{0}\|_{L^{\infty}(\R)})}|x^{1-p}|-\tfrac{(p-2)}{p}\big\|\tfrac{1}{(\cdot)^{p-1}}\big\|_{L^{\infty}((q_{\min},\|q_{0}\|_{L^{\infty}(\R)}))}>0
\]
which is clearly satisfied if
\[
\tfrac{1}{\|q_{0}\|_{L^{\infty}(\R)}^{p-1}}>\tfrac{p-2}{p}\tfrac{1}{q_{\min}^{p-1}}\Longleftrightarrow \tfrac{q_{\min}}{q_{\max}}> \sqrt[p-1]{\tfrac{p-2}{p}}
\]
This leaves no restriction for \(p=2\) except for \(q_{\min}>0\). For \(p\rightarrow\infty\) the right hand side converges from below to \(1\), allowing only constant data, however, for smaller \(p\), there are nontrivial initial data satisfying the inequality.
\end{rem}
\begin{theo}[Singular limit in the Ole\u inik case]
Let the assumptions of \cref{thm:oleinik} hold and assume in addition that the ``flux'' \( [q_{\min},\|q_{0}\|_{L^{\infty}(\R)}]\ni x\mapsto xV(x)\) is strictly concave. Then, the nonlocal operator \(W_{\eta}\) as well as the corresponding solution \(q_{\eta}\) converge in \(C\big([0,T];L^{1}_{\text{loc}}(\R)\big)\) toward the unique weak entropy solution of the correpsonding local conservation law in \cref{defi:local_conservation_law}.
\end{theo}
\begin{proof}
The OSL condition together with the uniform bounds gives uniform \(TV\) bounds on any bounded set implying, particularly by means of compactness, strong convergence to a weak solution of the local conservation law. \cite{oleinik,oleinik_english} then guarantees by means of the OSL condition the entropy admissibility. For details, we refer to \cite{coclite2023oleinik} where these arguments for the case \(p=1\) have been carried out in greater depth.
\end{proof}

\section{Existence of solutions for initial data not being bounded away from zero}\label{sec:existence_zero_initial_datum}
In this section, we will investigate the existence of solutions in the case where the initial datum is not bounded away from zero as assumed before in \cref{sec:existence_uniqueness}. As we want to keep the nonlocal operator from becoming zero on compact sets (as it then remains locally Lipschitz-continuous in the spatial variable), we focus only on one-sided nonlocal kernels with \textit{infinite} length of support and initial datum, bounded away from zero close to infinity (and zero close to \(-\infty\)). Under these conditions, we are able to show the existence of weak solutions. Regarding the uniqueness of solutions, since we can only work on compact domains (as the lower bound of the nonlocal operator does not hold close to \(-\infty\)), we can only partially use the methodology of \cref{sec:existence_uniqueness} which prevents us from proving a general result of uniqueness.

\begin{theo}[Existence of solutions for initial data becoming zero]\label{theo:existence_zero}
    Let \cref{ass:input_datum_and_velocity} hold but replace the lower bound on the initial datum, i.e., \(q_{0}\geqq q_{\min}\) by either
    \begin{enumerate}
        \item 
        \(
\exists q_{\min}\in\R_{>0},\ \exists R_{\text{l}},R_{\text{r}}\in\R:\ \big(q_{0}\geq q_{\min}\ \text{ for } x\in\R_{>R_{\text{r}}}\ \text{a.e.}\big)\ \wedge\ \big(q_{0}\equiv 0\ \text{a.e.}\ x\in \R_{<R_{\text{l}}}\big)
    \)\label{item:1_zero}
    \item or \(\exists q_{\min}\in\R_{>0},\ K\subset\R \text{ compact}:\ \big(\supp(q_{0})=\R\setminus K\big)\ \wedge\ \big(q_{0}|_{\supp(q_{0})}\geq q_{\min}\big)\).\label{item:2_zero}
    \end{enumerate}
    Assume, in addition, that the nonlocal kernel \(\gamma\) satisfies \(\gamma(x)>0\ \forall x\in\R_{>0}\), i.e., has infinite support and that \(p\in[1,\infty)\). Then, there exists a weak solution \(q_{\eta}\in C\big([0,T];L^{1}_{\text{loc}}(\R)\big)\) in the sense of \cref{defi:weak_solution} to the nonlocal \(p-\)norm conservation law as presented in \cref{defi:model_class}.
    Moreover, the solution satisfies the following min/max principle: 
    \[0\leqq q_{\eta}\leqq \|q_{0}\|_{L^{\infty}(\R)}\ \text{ on } (t,x) \text{ a.e.}
    \]
    In the case of \cref{item:2}, the solution is unique. 
\end{theo}
\begin{proof}
We start with the first part, i.e.\ \cref{item:1_zero}. Let \(q_{0}\) be given, we consider the following approximation for the initial datum for a given \(\eps\in\R_{>0}\)
\begin{align*}
    q_{0,\eps}\coloneqq \begin{cases}
                q_{0}+\eps &x\in\R_{\leq R_{\text{l}}}\\
                q_{0}, &\text{otherwise}.
    \end{cases}
\end{align*}
Then, for \(\eta\) fixed, there exists a unique weak solution on any finite time horizon \(T\in\R_{>0}\) as the initial datum is bounded away from zero except for a compact set. This enables it to apply the fixed-point methods in \cref{theo:existence_uniqueness_maximum_principle} to obtain a unique weak solution \(q_{\eps}\in C([0,T];L^{1}_{\text{loc}}(\R))\).
We shortly sketch the boundedness away from zero of the nonlocal operator:
\begin{align*}
    W_{\eta}[q_{\eps},\gamma](t,x)^{p}&=\tfrac{1}{\eta}\int_{x}^{\infty}\gamma(\tfrac{y-x}{\eta})^{p}q_{\eps}(t,y)^{p}\dd y\geq \tfrac{1}{\eta}\int_{\xi_{\eps}(0,R_{\text{r}};T)}^{\infty}\gamma(\tfrac{y-x}{\eta})^{p}q_{\eta,\eps}(t,y)^{p}\dd y\\
    &\geq q_{\min}^{p}\tfrac{1}{\eta}\int_{\xi_{\eps}(0,R_{\text{r}};T)}^{\infty}\!\!\!\gamma(\tfrac{y-x}{\eta})^{p}\dd y=q_{\min}^{p}\|\gamma\|_{L^{p}\big(\big(\frac{\xi_{\eps}(0,R_{\text{r}};T)-x}{\eta},\infty\big)\big)}^{p}
\end{align*}
with \(\xi\) being the characteristics as in \cref{eq:characteristics}. This is not uniform in \(x\in\R\), as the left hand side converges to zero when \(x\rightarrow-\infty\). But, for \(x> 2R_{\text{l}}\), we have the lower bound
\[
    \essinf_{t\in[0,T],\ x\in (2R_{\text{l}},\infty)}W_{\eta}[q_{\eps},\gamma](t,x)^{p}\geq q_{\min}^{p}\|\gamma\|_{L^{p}\big(\big(\frac{\xi_{\eps}(0,R_{\text{r}};T)-2R_{\text{l}}}{\eta},\infty\big)\big)}^{p}
\]
and, thanks to the assumed monotonicity of \(V\), \cref{ass:input_datum_and_velocity}
\[
    \essinf_{t\in[0,T],\ x\in (2R_{\text{l}},\infty)}W_{\eta}[q_{\eps},\gamma](t,x)\geq q_{\min}\|\gamma\|_{L^{p}\big(\big(\frac{R_{\text{r}}+V(0)T-2R_{\text{l}}}{\eta},\infty\big)\big)}.
\]
Similarly, one can obtain a lower bound on \(W_{\eta}[q_{\eps}]\) for \(x\leq 2R_{\text{l}}\) which can be established as we have assumed the initial datum in this region to be \(\eps\in\R_{>0}\). With this in mind, we have the boundedness of \(\partial_{2}W_{\eta}\) as this boundedness could only be destroyed when the nonlocal operator could approach zero as
\begin{align*}
    \partial_{x}W_{\eta}[\tilde{q}_{\eta,\eps},\gamma](t,x)&=\tfrac{1}{p}W_{p}[\tilde{q}_{\eta,\eps},\gamma](t,x)^{1-p}\bigg(-\tfrac{p}{\eta^{2}}\int_{x}^{\infty}\gamma(\tfrac{y-x}{\eta})^{p-1}\gamma'(\tfrac{y-x}{\eta})\tilde{q}_{\eta,\eps}(t,y)^{p}\dd y-\tfrac{1}{\eta}\gamma(0)^{p}\tilde{q}_{\eta,\eps}(t,x)^{p} \bigg).
\end{align*}
Thus, we have a weak solution
\[
q_{\eps}(t,x)=q_{0,\eps}\big(\xi_{w_{\eps}}(t,x;0)\big)\partial_{2}\xi_{w_{\eps}}(t,x;0),\ (t,x)\in\OT.
\]
with \(w\) being the unique solution to the fixed-point equation in \cref{eq:fixed_point_1}, \cref{eq:fixed_point_2}.

Now, we claim that we obtain a solution to the problem stated in \cref{item:1_zero} by defining
\[
q(t,x)=\begin{cases}q_{\eps}(t,x) &t\in[0,T],\ x\in [\xi_{w_{\eps}}(0,x;R_{\text{l}}),\infty)\\
0 &\text{else}
\end{cases}.
\]
The rationality behind this approach lies in the fact that the nonlocal operator only looks ``to the right'' so that any changes on the left hand side will not affect the solution ``on the right'', and thus, as long as we set the solution to zero at this set, the originally built solution is not affected. This is proven rigorously next. Recall the weak formulation of solutions in \cref{defi:weak_solution}; it holds for \(\eps\in\R_{>0}\) and 
\(\forall\phi\in C^{1}_{\text{c}}((-42,T)\times\R)\) that
  \begin{align*}
      0&=\iint_{\OT}\Big(\phi_{t}(t,x)+V(W_{\eta}[q_{\eps},\gamma](t,x))\phi_{x}(t,x)\Big)q_{\eps}(t,x)\dd x\dd t+\int_{\R} q_{0,\eps}(x)\phi(0,x)\dd x\\
      &=\int_{0}^{T}\int_{\xi_{w}(0,R_{\text{l}};t)}^{\infty}\Big(\phi_{t}(t,x)+V(W_{\eta}[q_{\eps},\gamma](t,x))\phi_{x}(t,x)\Big)q_{\eps}(t,x)\dd x\dd t\\
      &\quad +\int_{0}^{T}\int_{-\infty}^{\xi_{w}(0,R_{\text{l}};t)}\Big(\phi_{t}(t,x)+V(W_{\eta}[q_{\eps},\gamma](t,x))\phi_{x}(t,x)\Big)q_{\eps}(t,x)\dd x\dd t+\int_{\R} q_{0,\eps}(x)\phi(0,x)\dd x
      \intertext{and plugging in the solution formula}
      &=\int_{0}^{T}\int_{\xi_{w}(0,R_{\text{l}};t)}^{\infty}\Big(\phi_{t}(t,x)+V(W_{\eta}[q_{\eps},\gamma](t,x))\phi_{x}(t,x)\Big)q_{0,\eps}(\xi_{\eps}(t,x;0))\partial_{2}\xi_{w}(t,x;0)\dd x\dd t\\
      &\quad +\int_{0}^{T}\int_{-\infty}^{\xi_{w}(0,R_{\text{l}};t)}\Big(\phi_{t}(t,x)+V(W_{\eta}[q_{\eps},\gamma](t,x))\phi_{x}(t,x)\Big)q_{\eps}(t,x)\dd x\dd t+\int_{\R} q_{0,\eps}(x)\phi(0,x)\dd x
      \intertext{and substituting}
      &=\int_{0}^{T}\int_{R_{\text{l}}}^{\infty}\Big(\partial_{1}\phi(t,\xi_{\eps}(0,y;t))+V(W_{\eta}[q_{\eps},\gamma](t,\partial_{2}\xi_{\eps}(0,y;t)))\partial_{2}\phi(t,\xi_{w}(0,y;t))\Big)q_{0,\eps}(y)\dd y\dd t\\
      &\quad +\int_{0}^{T}\int_{-\infty}^{\xi_{w}(0,R_{\text{l}};t)}\Big(\phi_{t}(t,x)+V(W_{\eta}[q_{\eps},\gamma](t,x))\phi_{x}(t,x)\Big)q_{\eps}(t,x)\dd x\dd t+\int_{\R} q_{0,\eps}(x)\phi(0,x)\dd x\\
      &=\int_{R_{\text{l}}}^{\infty}\int_{0}^{T}\tfrac{\dd}{\dd t} \phi(t,\xi_{\eps}(0,y;t))q_{0}(y)\dd y\\
      &\quad +\int_{0}^{T}\int_{-\infty}^{\xi_{w}(0,R_{\text{l}};t)}\Big(\phi_{t}(t,x)+V(W_{\eta}[q_{\eps},\gamma](t,x))\phi_{x}(t,x)\Big)q_{\eps}(t,x)\dd x\dd t+\int_{\R} q_{0,\eps}(x)\phi(0,x)\dd x
      \intertext{and, thanks to the assumptions on the test function and the characteristics,}
      &=-\int_{R_{\text{l}}}^{\infty}\phi(0,y)q_{0}(y)\dd y\\
      &\quad +\int_{0}^{T}\int_{-\infty}^{\xi_{w}(0,R_{\text{l}};t)}\Big(\phi_{t}(t,x)+V(W_{\eta}[q_{\eps},\gamma](t,x))\phi_{x}(t,x)\Big)q_{\eps}(t,x)\dd x\dd t+\int_{\R} q_{0,\eps}(x)\phi(0,x)\dd x\\
      &=\int_{0}^{T}\int_{-\infty}^{\xi_{w}(0,R_{\text{l}};t)}\Big(\phi_{t}(t,x)+V(W_{\eta}[q_{\eps},\gamma](t,x))\phi_{x}(t,x)\Big)q_{\eps}(t,x)\dd x\dd t+\int_{-\infty}^{R_{\text{l}}} q_{0,\eps}(x)\phi(0,x)\dd x.
\end{align*}
Now, we can make the last two terms zero if we just set the initial datum in this area to zero as well as the solution. However, this was exactly the previously mentioned construction. Thus, we have built a solution to the initial value problem for the initial datum being bounded away from zero for \(x\) large and being zero for \(x\) small enough, i.e., the case of \cref{item:1_zero}.

Concerning \cref{item:2_zero}, we use the fact that, for a given \(\eta\in\R_{>0},\) the nonlocal operator will always be bounded away from zero, uniformly in the spatial variable, so that we can built on the sketched proof in \cref{theo:existence_uniqueness_maximum_principle}.
\end{proof}
\begin{rem}[Uniqueness in the zero density case]
The fact that we cannot prove uniqueness in \cref{theo:existence_zero} \cref{item:1} does not mean that uniqueness does not hold. It means, rather, that the classical fixed-point approach does not work in this case as the velocity field generated by the nonlocal operator is not globally Lipschitz in the spatial variable. However, it might be possible that other methods, like semigroup approaches or weaker topologies to consider the fixed-point problem, do work. Currently, it is not clear how we could expect even existence if the initial data can be zero without satisfying either \cref{item:1} or \cref{item:2}.
\end{rem}
\begin{rem}[Singular limit convergence in the case of the assumptions in \cref{theo:existence_zero}]
Note that the \(TV\) bound obtained in \cref{theo:TV_bounds_Wp} does not require the solution to be bounded away from zero. 
Thus, we have compactness in the spatial variable of \(W_{\eta}^{p}\) in the case of an exponential kernel, which also gives strong convergence of \(W_{\eta}\) in the spatial variable.

However, obtaining time compactness as in \cref{cor:compactnessWp} seems to be difficult, as we cannot rely on the boundedness of
\[
\eta \partial_{2}W_{\eta} \ \text{ but only on } \ \eta \partial_{2}W_{\eta}^{p}
\]
preventing us from estimating the time derivative of \(W_{\eta},\ W_{\eta}^{p}\) in \cref{eq:nonlocal_PDE} or \cref{rem:PDE_nonlocal_p} respectively. This again prevents us from obtaining a general singular limit convergence for non-negative initial data \(q_{0}\) so that we have the restriction of 
\(q_{0}\geqq q_{\min}\in\R_{>0}\).
\end{rem}

\section{Monotonicity preserving dynamics}\label{sec:monotonicity}
In this section, we show that the nonlocal dynamics conserve monotonicity under additional assumptions on the second derivative of \(V\). This is well-known for nonlocal conservation laws (see \cite[Theorem 4.13\ \& Theorem 4.18]{pflug4} and \cite[Theorem 5.1]{friedrich2022conservation} for the nonlocality in the velocity), and holds in the \(p\)-norm case as well:
\begin{theo}[Monotonicity of the nonlocal solution] \label{thm:monotonicity_pinfty}
Let \cref{ass:input_datum_and_velocity}  hold. Assume that \(p\in\R_{>0}\) and $q_0$ and $V''$ satisfy additionally
\begin{equation}
\begin{aligned}
    &\Big(p\in(0,1]\ \wedge\ q_0(x)\leq q_0(y) \text{ for a.e. } x,y\in \R,\, x\leq y \wedge V''(s)\leq 0 \text{ for }s\in (0,\|q_0\|_{L^\infty (\R)}) \text{ a.e.} \Big) \\
    \vee\ 
    &\Big(p\in[1,\infty)\ \wedge\ q_0(x)\geq q_0(y) \text{ for a.e. } x,y\in\R, \, x\leq y \wedge V''(s)\geq 0 \text{ for }s\in (0,\|q_0\|_{L^\infty (\R)}) \text{ a.e.}\Big),
\end{aligned}
\label{eq:monotonicity_preserving}
\end{equation}
i.e., \(V''\leqq 0\) and \(p\in(0,1]\) for monotonically increasing initial data, and \(V''\geqq 0\) and \(p\in[1,\infty)\) for monotonically decreasing initial data.

Then, for every nonlocal kernel as in \cref{ass:input_datum_and_velocity}, the solution \(q\) is ``monotonicity preserving'' on the entire time horizon considered, i.e., $\forall t\in[0,T]$ the solution $q(t,\cdot)$ is monotone and, consequently, the total variation is uniformly bounded, i.e., 
\begin{equation*}
   \forall (\eta, t)\in \R_{>0}\times [0,T]: \, |q_\eta(t,\cdot)|_{TV(\R)} \leq \|q_0\|_{L^\infty(\R)}.
\end{equation*}
For the exponential kernel \(\gamma\equiv \exp(-\cdot)\), monotonicity is preserved for \textbf{all} \(p\in(0,\infty)\) under the condition in \cref{eq:monotonicity_preserving}.
\end{theo}

\begin{proof}
We approximate the initial datum according to \cref{theo:existence_uniqueness_maximum_principle} by an arbitrary smooth initial datum in \(L^1(\R)\) such that the corresponding solution is sufficiently smooth. Now, we look at the classical solution that satisfies \cref{defi:model_class}. Let us suppress, in \(W_{\eta}[\gamma,q]\), the dependency with regard to \(\gamma,q\) and leave out the time and space dependencies.  Differentiating \cref{eq:pnorm_problem} in space yields
\begin{equation}
\begin{aligned}
\partial_{t}\partial_{x}q&=\partial_{x}\big(-V'(W_{\eta})\partial_{x}W_{\eta}q -V(W_{\eta})q_{x}\big)\\
&=-V''(W_{\eta})(\partial_{x}W_{\eta})^{2}q -V'(W_{p})\partial_{x}^{2}W_{\eta}q -2V'(W_{\eta})\partial_{x}W_{\eta}q_{x}-V(W_{\eta})q_{xx}.
\end{aligned}
\label{eq:PDE_differentiated}
\end{equation}
Next, we analyze  \(\partial_{x}^{2}W_{\eta}\) and compute it explicitly taking advantage of \cref{eq:partial_2_W_reformulations} to write
\begin{equation}
\begin{aligned}
\partial_{x}^{2}W_{\eta}&=\partial_{x}\Bigg(W_{\eta}^{1-p}\bigg(-\tfrac{1}{\eta^{2}}\int_{x}^{\infty}\gamma(\tfrac{y-x}{\eta})^{p-1}\gamma'(\tfrac{y-x}{\eta})q(t,y)^{p}\dd y-\tfrac{1}{p\eta}\gamma(0)^{p}q(t,x)^{p} \bigg) \Bigg)\\
&=\partial_{x}\bigg(W_{\eta}^{1-p}\tfrac{p}{\eta}\int_{x}^{\infty}\gamma(\tfrac{y-x}{\eta})^{p}q(t,y)^{p-1}\partial_{y}q(t,y)\dd y\bigg)\\
&=(1-p)W_{\eta}^{-p}\partial_{x}W_{\eta}\tfrac{p}{\eta}\int_{x}^{\infty}\gamma(\tfrac{y-x}{\eta})^{p}q(t,y)^{p-1}\partial_{y}q(t,y)\dd y-W_{\eta}^{1-p}\tfrac{p}{\eta}\gamma(0)^{p}q(t,x)^{p-1}\partial_{x}q(t,x)\\
&\quad -W_{\eta}^{1-p}\tfrac{p^{2}}{\eta^{2}}\int_{x}^{\infty}\gamma(\tfrac{y-x}{\eta})^{p-1}\gamma'(\tfrac{y-x}{\eta})q(t,y)^{p-1}\partial_{y}q(t,y)\dd y.
\end{aligned}
\label{eq:partial_2_squared_W}
\end{equation}
Now, let us look into the case in which the initial datum is monotonically increasing, i.e., \(q_{0}'(x)\geqq 0\).
Let us suppose that there is a time \(t\in[0,T]\) and \(\tilde{x}\in\R\) such that \(\partial_{2}q(t,\tilde{x})=0\) (this is the spatial point and time where the monotonicty could be destroyed first). Then, we know that \(\partial_{2}^{2}q(t,\tilde{x})= 0\) and \(\partial_{2}q(t,x)\geq 0\ \forall x\in\R\). Evaluating \cref{eq:PDE_differentiated} at \((t,\tilde x)\) we arrive at
\begin{align*}
    \partial_{t}\partial_{2}q(t,\tilde{x})&=-V''(W_{\eta}(t,\tilde{x})(\partial_{2}W_{\eta}(t,\tilde{x}))^{2}q(t,\tilde{x}) -V'(W_{\eta}(t,\tilde{x}))\partial_{2}^{2}W_{\eta}(t,\tilde{x}) q(t,\tilde{x})\\
    &\qquad -2V'(W_{\eta}(t,\tilde{x}))\partial_{2}W_{\eta}(t,\tilde{x})\partial_{2}q(t,\tilde{x})-V(W_{\eta}(t,\tilde{x}))\partial_{2}^{2}q(t,\tilde{x}),
    \intertext{using the conditions on \(V''\) in the monotonically increasing case, i.e., \(V''\leqq 0\) and that \(\partial_{2}^{2}q(t,\tilde{x})=0=\partial_{2}q(t,\tilde{x})\) as well as \(q\geqq0\)}
    &\geq -V'(W_{\eta}(t,\tilde{x}))\partial_{x}^{2}W_{\eta}(t,\tilde{x}) q(t,\tilde{x}),
\end{align*}
and, as \(q\geqq 0\) once more, it remains to find the sign of \(\partial_{2}^{2}W_{\eta}(t,\tilde{x})\). To this end, we invoke \cref{eq:partial_2_squared_W} and have
\begin{align*}
    \partial_{2}^{2}W_{\eta}(t,\tilde{x})&=-(p-1)W_{\eta}(t,\tilde{x})^{-p}\partial_{2}W_{\eta}(t,\tilde{x})\tfrac{1}{\eta}\int_{\tilde{x}}^{\infty}\gamma(\tfrac{y-\tilde{x}}{\eta})^{p}q(t,y)^{p-1}\partial_{y}q(t,y)\dd y  \\
&\qquad -W_{\eta}^{1-p}\tfrac{p^{2}}{\eta^{2}}\int_{\tilde{x}}^{\infty}\gamma(\tfrac{y-\tilde{x}}{\eta})^{p-1}\gamma'(\tfrac{y-\tilde{x}}{\eta})q(t,y)^{p-1}\partial_{y}q(t,y)\dd y -W_{\eta}^{1-p}\tfrac{p}{\eta}\gamma(0)^{p}q(t,\tilde{x})^{p-1} \partial_{2}q(t,\tilde{x}).
\end{align*}
For \(p\in(0,1]\), the first summand is nonnegative as \(\partial_{2}W_{\eta}\) is nonnegative thanks to \(\partial_{2}q(t,\cdot)\geqq0\) as well as the integral operator integrating a nonnegative function.
The second summand is positive by almost the same reason, recalling that \(\gamma'\leqq 0\), and the third summand is zero as \(\partial_{2}q(t,\tilde{x})=0\).
Thus, monotonically increasing solutions remain monotonically increasing for all times, given that \(p\in(0,1]\) and \(V''\leqq 0\). The argument is similar for a monotonically decreasing datum, \(p\in[1,\infty)\) and \(V''\geqq 0\).
The stated total variation estimate follows directly from the monotonicity.

Next, let us consider the exponential kernel. Again, assume that we are in the monotonically increasing case. We require \cref{eq:partial_2_squared_W} to be nonnegative at \(\tilde{x}\). We are left with the term \(\partial_{2}^{2}W_{\eta}\) and we need to show that it is nonnegative. To do this, let us compute and estimate as follows for \(p\in(1,\infty)\) and \((t,x)\in\OT\):
\begin{equation}
\begin{aligned}
    \partial_{2}^{2}W_{\eta}(t,x)&=-(p-1)W^{-p}_{\eta}\partial_{x}W_{\eta}\tfrac{1}{\eta}\int_{x}^{\infty}\exp(\tfrac{x-y}{\eta})^{p}q(t,y)^{p-1}\partial_{y}q(t,y)\dd y  \\
&\quad +W^{1-p}_{\eta}\tfrac{p}{\eta^{2}}\int_{x}^{\infty}\exp(\tfrac{x-y}{\eta})^{p}q(t,y)^{p-1}\partial_{y}q(t,y)\dd y -W^{1-p}_{\eta}\tfrac{p}{\eta}\exp(0)^{p}q(t,x)^{p-1} q_{x}(t,x) \\
&=W^{-p}_{\eta}\left(W_{\eta}\tfrac{p}{\eta}-(p-1)\partial_{x}W_{\eta}\right) \tfrac{1}{\eta}\!\int_{x}^{\infty}\!\!\!\!\!\exp(\tfrac{x-y}{\eta})^{p}q(t,y)^{p-1}\partial_{y}q(t,y)\dd y -\tfrac{pq^{p-1}}{W^{p-1}_{\eta}\eta} q_{x}(t,x) 
\end{aligned}
\label{eq:partial_2^2W}
\end{equation}
using the nonlocal identity~\eqref{eq:nonlocal_identity}, we have that \(W^{-p}_{\eta}\left(W_{\eta}\tfrac{p}{\eta}-(p-1)\partial_{x}W_{\eta}\right)\geqq0\).

In particular, at \(x=\tilde{x}\), we can conclude that \(\partial_{2}^{2}W(t,\tilde{x})\geq 0\), together with \cref{eq:PDE_differentiated}, yields the increasing monotonicity. One can prove the monotonically decreasing case with a similar argument.
We performed all these computations thanks to the smoothness of the considered solutions, but the monotonicity is preserved also in the limit for non-smooth initial data according to \cref{theo:existence_uniqueness_maximum_principle}.
\end{proof}
\begin{rem}[Monotonicity uniformly in \(p\) and generalizations]
From the previous monotonicity results, it becomes apparent that monotonicity is preserved uniformly in \(p\) when taking the exponential kernel in \cref{ass:exponential}. Then, one can pass to the limit for \(p\rightarrow\infty\) and obtain solutions for the corresponding limit. A similar argument can be made for \(p\rightarrow 0\) without further restrictions on the kernel. This is further detailed in the following \cref{sec:p=0} and numerically approximated in \cref{sec:numerics}.

\end{rem}
\section{Limits of the solution for \texorpdfstring{\(p\rightarrow 0\)}{p to 0} and some comments on \texorpdfstring{\(p\rightarrow \infty\)}{p to ∞}}\label{sec:p=0}
We look into the limit of the nonlocal conservation law's solution as \(p\rightarrow 0\).
Firstly, we define what we expect to be the limiting weak solution. Interestingly, the equation becomes a different type of nonlocal conservation law involving the integral of a logarithm. This is in line with what is well known from the \(p\)-quasi-norm behavior for \(p\rightarrow 0\) (we recall that the triangle inequality is not verified) and is detailed in the following.
\begin{lem}[{Convergence of \(\|\cdot\|_{L^{p}(\Omega)}\) for \(p\rightarrow 0\) on bounded \(\Omega\)}]\label{lem:L_p_p=0}
    Let \(\Omega\subset\R\) be open and bounded and, for a \(q\in \R_{>0}\), \(f\in L^{p}(\Omega)\ \forall p\in(0,q]\) be given. Then, it holds that
    \begin{equation}
    \lim_{p\rightarrow 0}\|f\|_{L^{p}(\Omega)}=\exp\bigg(\tfrac{1}{|\Omega|}\int_{\Omega}\ln(|f(x)|)\dd x\bigg)\eqqcolon \|f\|_{L^{0}(\Omega)}.\label{eq:L_p_p=0}
    \end{equation}
\end{lem}
\begin{proof}
    The result can be found in \cite[p. 187, Exercise 8c]{folland}. It can also be found, with a broader perspective on \(L^{p}\) metric spaces for \(p\in(0,1)\) in \cite[p. 139, 6.8, 1.87, Equation 6.8.1]{hardy_littlewood_polya}.
\end{proof}

The right hand side of \cref{eq:L_p_p=0} is sometimes also called a geometric mean \cite[6.7.``The geometric mean of a function'']{hardy_littlewood_polya} which becomes clear when taking the counting measure instead of the Lebesgue measure.
It should furthermore be mentioned that writing \(\|f\|_{L^{p}}\) for \(p\in(0,1)\) is an abuse of notation insofar as \(\|\cdot\|\) typically represents a norm. Nevertheless, for convenience we will stick with this notation.

As can be seen in \cref{lem:L_p_p=0}, it is crucial to have a bounded \(\Omega\). This means that we cannot work with nonlocal kernels with infinite support and need to restrict our analysis to finitely supported kernels.
\begin{defi}[Weak solution for \(p=0\) ]\label{defi:weak_solution_p_equal_zero}
Let \(\gamma\) as in \cref{ass:input_datum_and_velocity} in addition be compactly supported (without loss of generality on \((0,1)\)) and \(\gamma\geq \gamma_{\min} \text{ on } (0,1).\)
Then, we call, for \(\eta\in\R_{>0}\), the following conservation law 
\begin{equation}
\begin{aligned}
\partial_{t}q+\partial_{x}\big(V(W_{\eta,0}[q,\gamma_{0}](t,x))q\big)&=0,&& (t,x)\in\OT\\
    q(0,x)&=q_{0}(x), && x\in\R
\end{aligned}
\label{eq:strong_form_solution_p=0}
\end{equation}
a nonlocal conservation law for \(p=0\) where the nonlocal operator is defined by
  \begin{equation}
    W_{\eta,0}[q,\gamma_{0}](t,x)\coloneqq \exp\bigg(\tfrac{1}{\eta}\int_{x}^{x+\eta}\ln\big(\big|\gamma_{0}(\tfrac{x-y}{\eta})q(t,y)\big|\big)\dd y\bigg),\qquad (t,x)\in\OT\label{eq:W_eta_0}
  \end{equation}
  and (with the definition in \cref{eq:L_p_p=0})
\[
\gamma_{0}\coloneqq \tfrac{\gamma}{\|\gamma\|_{L^{0}((0,1))}}.
\]
  Furthermore, we call \(q:\OT\rightarrow\R\) a weak solution to \cref{eq:strong_form_solution_p=0} if \[q\in C\big([0,T];L^{1}_{\text{loc}}(\R)\big)\] for \(q_{0}\in L^{1}_{\text{loc}}(\R;\R_{>q_{\min}})\) and it holds for all \(\phi\in C^{1}_{\text{c}}((-42,T)\times\R)\)
  \begin{gather*}
      \iint_{\OT}\Big(\phi_{t}(t,x)+V(W_{\eta,0}[q,\gamma_{0}](t,x))\phi_{x}(t,x)\Big)q(t,x)\dd x\dd t+\int_{\R} q_{0}(x)\phi(0,x)\dd x=0.
  \end{gather*}
\end{defi}
The nonlocal conservation law in \cref{eq:strong_form_solution_p=0} does not fall in the model class analyzed in the previous sections as the nonlinearity in the nonlocal term there is only of the \(p\)th power, with the \(\frac{1}{p}\) around the integral which could be compensated ``into'' the velocity field. In contrast, the present setup has a logarithmic density function.

In \cite{Friedrich2024}, the authors study a more general nonlinear nonlocal conservation law that can almost be applied to \cref{eq:strong_form_solution_p=0}.
\begin{lem}[Existence and Uniqueness of weak solutions as in \cref{defi:weak_solution_p_equal_zero}]\label{lem:existence_uniqueness_p=0}
    Let \(q_{0}\in L^{\infty}\big(\R;\R_{\geq q_{\min}}\big)\cap TV(\R)\) for a given \(q_{\min}\in\R_{>0}\) be given. Then, there exists a unique weak solution \(q\in C\big([0,T];L^{1}_{\text{loc}}(\R)\big)\cap L^{\infty}((0,T);L^{\infty}(\R))\) to \cref{eq:strong_form_solution_p=0} on any finite time horizon \(T\in\R_{>0}\) in the sense of \cref{defi:weak_solution_p_equal_zero}, satisfying the max/min principle
    \[
    q_{\min}\leqq q \leqq \|q_{0}\|_{L^{\infty}(\R)} \text{ on } \OT \text{ a.e.}
    \]
\end{lem}
\begin{proof}
 This result can be found in \cite[Theorem 2]{Friedrich2024} under slightly more restrictive assumptions, i.e., \(q_{0}\in BV(\R)\) and the nonlocal kernel being outside of the nonlinearity. In our setup, we cannot postulate \(q_{0}\in BV(\R)\) as we bound the initial datum away from zero so that we do not have a finite \(L^{1}-\)norm. However, this is also not required in the proof. In  \cite{Friedrich2024}, the nonlinearity in the nonlocal term is only on the solution and not on the kernel. Although this nonlinearity is in the case considered on both the solution and the kernel function, such can also be incorporated into the framework by redefining the kernel. We do not detail it further.
\end{proof}
This brings us to the main theorem of this section, i.e., the convergence result for \(p \rightarrow 0\) of solutions to nonlocal conservation laws in the \(p\)-norm.
\begin{theo}[Convergence of the nonlocal solution for \(p\rightarrow 0\)]\label{theo:convergence_p=0}
Consider the solution of the nonlocal conservation law in \cref{defi:model_class} with  \cref{ass:input_datum_and_velocity}. Let us assume in addition that the kernel \(\gamma\) is compactly supported on \((0,1)\) and  \(\gamma(1)>\gamma_{\min}\in\R_{>0}\). Define the kernel, normalized in \(L^{p}\), as
\[
\gamma_{p}\coloneqq  \tfrac{\gamma}{\|\gamma\|_{L^{p}((0,1))}} \text{ on }\ (0,1).
\]
Denote by \(q_{\eta,p}\in C\big([0,T];L^{1}_{\text{loc}}(\R)\big)\) the nonlocal solution with the nonlocal operator \(W_{p,\eta}[q_{\eta,p},\gamma_{p}]\).

Let us call \(q_{\eta,\ast}\in C\big([0,T];L^{1}_{\text{loc}}(\R)\big)\) the solution of the nonlocal conservation law for ``\(p=0\),'' i.e., the solution of \cref{defi:weak_solution_p_equal_zero} and  \cref{eq:strong_form_solution_p=0}.

Then, it holds for any \(\Omega\subset\R\) open and bounded that
\[
\lim_{p\rightarrow 0} \|q_{\eta,p}-q_{\eta,\ast}\|_{L^{1}((0,T)\times\Omega)}=0,
\]
i.e., the solution for the nonlocal conservation law and \(p\in\R_{>0}\) converges to the solution of the corresponding nonlocal conservation law with \(p=0\) in the \(L^{1}((0,T^{*})\times\Omega\big)\) norm when \(p\) approaches zero.
\end{theo}
\begin{proof}
We start the proof by showing that the solutions \(q_{\eta,p}\) are compact in \(C\big([0,T];L^{1}_{\text{loc}}(\R)\big)\) with respect to the parameter \(p\in\R_{>0}\). 
We pursue this by means of spatial \(TV\) estimates and additional time compactness.

To this end, let us consider smooth nonlocal solutions, thanks to \cref{theo:existence_uniqueness_maximum_principle}, and, following the same argument as in the proof of \cref{theo:existence_zero}, we obtain (suppressing the dependency of \(q_{\eta,p}\) and \(W_{\eta,p}\) on \(\eta\in\R_{>0}\))
\begin{align}
\tfrac{\dd}{\dd t}\!\! \int_{\R}\!\! |\partial_{x}q_{p}(t,x)^{p}|\dd x
    &=-p^{2}\int_{\R}\sgn(\partial_{x}q_{p}(t,x)^{p})q_{p}^{p-1} \partial_{x}q_{p}V'(W_{p}[q_{p},\gamma_{p}](t,x))\partial_{x}W_{p}[q_{p},\gamma_{p}](t,x)\dd x\label{eq:TV_estimate_p_0_0}\\
    &\quad -p\int_{\R}\sgn(\partial_{x}q_{p}(t,x)^{p})q_{p}^{p}V''(W_{p}[q_{p},\gamma_{p}](t,x))\big(\partial_{x}W_{p}[q_{p},\gamma_{p}](t,x)\big)^{2}\dd x\notag\\
    &\quad -p\int_{\R}\sgn(\partial_{x}q_{p}(t,x)^{p})q_{p}^{p}V'(W_{p}[q_{p},\gamma_{p}](t,x))\partial_{x}^{2}W_{p}[q_{p},\gamma_{p}](t,x)\dd x\notag\\
    &\quad -p(p-1)\int_{\R}\sgn(\partial_{x}q_{p}(t,x)^{p})q_{p}^{p-2}\big(\partial_{x}q_{p}(t,x)\big)^{2}V(W_{p}[q_{p},\gamma_{p}](t,x))\dd x\notag\\
    &\quad -p\int_{\R}\sgn(\partial_{x}q_{p}(t,x)^{p})q_{p}^{p-1}\partial_{x}^{2}q_{p}V(W_{p}[q_{p},\gamma_{p}](t,x))\dd x\notag\\
    &\quad -p\int_{\R}\sgn(\partial_{x}q_{p}(t,x)^{p})q_{p}^{p-1}\partial_{x}q_{p}V'(W_{p}[q_{p},\gamma_{p}](t,x))\partial_{x}W_{p}[q_{p},\gamma_{p}](t,x)\dd x.\notag
    \intertext{Integrating by parts on the second last term and rearranging,}
    &=-p\int_{\R}|\partial_{x}q_{p}(t,x)^{p}|V'(W_{p}[q_{p},\gamma_{p}](t,x))\partial_{x}W_{p}[q_{p},\gamma_{p}](t,x)\dd x\notag\\
    &\quad -p\int_{\R}\sgn(\partial_{x}q_{p}(t,x)^{p})q_{p}^{p}V''(W_{p}[q_{p},\gamma_{p}](t,x))\big(\partial_{x}W_{p}[q_{p},\gamma_{p}](t,x)\big)^{2}\dd x\notag\\
    &\quad -p\int_{\R}\sgn(\partial_{x}q_{p}(t,x)^{p})q_{p}^{p}V'(W_{p}[q_{p},\gamma_{p}](t,x))\partial_{x}^{2}W_{p}[q_{p},\gamma_{p}](t,x)\dd x\notag
    \intertext{estimating in the proper norms---assuming that those will be finite and with \(q_{\max}\coloneqq \|q_{0}\|_{L^{\infty}(\R)}\)}
    &\leq p\|V'\|_{L^{\infty}((q_{\min},q_{\max}))}|q_{p}^{p}(t,\cdot)|_{TV(\R)}\|\partial_{2}W_{p}[q_{p},\gamma_{p}](t,\cdot)\|_{L^{\infty}(\R)}\label{eq:TV_estimate_p_0_1}\\
    &\quad + pq_{\max}^{p}\|V''\|_{L^{\infty}((q_{\min},q_{\max}))}\|\partial_{2}W_{p}[q_{p},\gamma_{p}](t,\cdot)\|_{L^{\infty}(\R)}\|W_{p}[q_{p},\gamma_{p}](t,\cdot)\|_{TV(\R)}\label{eq:TV_estimate_p_0_2}\\
    &\quad +pq_{\max}^{p}\|V'\|_{L^{\infty}((q_{\min},q_{\max}))}|\partial_{2}W_{p}[q_{p},\gamma_{p}](t,\cdot)|_{TV(\R)}.\label{eq:TV_estimate_p_0_3}
\end{align}
At this point, we require to obtain estimates of \(\partial_{2}W_{p}\) in \(L^{\infty}(\R),\) in \(L^{1}(\R)\) and  in \(TV(\R)\) in terms of \(|q^{p}(t,\cdot)|_{TV(\R)}\) for \(t\in[0,T]\). Recall the identity in \cref{eq:partial_2_W_reformulations} for \((t,x)\in (0,T)\times\R\)
\[
\partial_{2}W_{p}[q_{p},\gamma_{p}](t,x)=\tfrac{W_{p}[q_{p},\gamma_{p}](t,x)^{1-p}}{p\eta}\bigg(\gamma_{p}(1)^{p}q_{p}(t,x+\eta)^{p}-\gamma_{p}(0)^{p}q_{p}(t,x)^{p}+\!\int_{x}^{x+\eta}\!\!\!\!\!\!\!\!\!\!\!\tfrac{\dd}{\dd x}\gamma_{p}(\tfrac{y-x}{\eta})^{p}q_{p}(t,y)^{p}\dd y \bigg).
\]
From this, we can deduce, thanks to \(\|\gamma_{p}\|_{L^{p}((0,1))}=1,\) and \(\gamma\geqq \gamma_{\min} \text{ on } (0,1)\),
\begin{align*}
\|\partial_{2}W_{p}[q_{p},\gamma_{p}](t,\cdot)\|_{L^{\infty}(\R)}&\leq \tfrac{1}{\eta p}\|W_{p}[q_{p},\gamma_{p}](t,\cdot)\|_{L^{\infty}(\R)}^{1-p}\Big(\big|\gamma_{p}(1)^{p}q_{\min}^{p}-\gamma_{p}(0)^{p}q_{\max}^{p}\big|+|\gamma_{p}^{p}|_{TV((0,1))}q_{\max}^{p}\Big)\\
&\leq \tfrac{q_{\max}^{1-p}}{\eta p\|\gamma\|_{L^{p}(0,1)}^{p}}\Big(\big|\gamma(1)^{p}q_{\min}^{p}-\gamma(0)^{p}q_{\max}^{p}\big|+|\gamma^{p}|_{TV((0,1))}q_{\max}^{p}\Big)\\
&\leq \tfrac{q_{\max}^{1-p}}{\eta p\gamma_{\min}^{p}}\Big(\big|\gamma(1)^{p}q_{\min}^{p}-\gamma(0)^{p}q_{\max}^{p}\big|+p\gamma_{\min}^{p-1}|\gamma|_{TV((0,1))}q_{\max}^{p}\Big)\\
&\leq \tfrac{q_{\max}^{1-p}}{\eta\gamma_{\min}^{p}}\bigg(\Big|\int_{\gamma(0)^{p}q_{\max}^{p}}^{\gamma(1)^{p}q_{\min}^{p}} s^{p-1}\dd s\Big|+\gamma_{\min}^{p-1}|\gamma|_{TV((0,1))}q_{\max}^{p}\bigg)\\
&\leq \tfrac{q_{\max}^{1-p}}{\eta\gamma_{\min}^{p}}\bigg(\Big|\int_{\gamma(0)^{p}q_{\max}^{p}}^{\gamma(1)^{p}q_{\min}^{p}} s^{p-1}\dd s\Big|+\gamma_{\min}^{p-1}|\gamma|_{TV((0,1))}q_{\max}^{p}\bigg)\\
&\leq \tfrac{q_{\max}^{1-p}}{\eta\gamma_{\min}^{p}}\Big(\gamma_{\min}^{p(p-1)}q_{\min}^{p(p-1)}\gamma(0)^{p}q_{\max}^{p}+\gamma_{\min}^{p-1}|\gamma|_{TV((0,1))}q_{\max}^{p}\Big)\\
&\leq \tfrac{q_{\max}}{\eta}\gamma_{\min}^{p(p-2)}q_{\min}^{p(p-1)}\gamma(0)^{p} + \tfrac{q_{\max}}{\eta\gamma_{\min}}|\gamma|_{TV((0,1))}\\
&\leq \tfrac{q_{\max}}{\eta}\gamma_{\min}^{-1}q_{\min}^{-\frac{1}{4}}\gamma(0)^{p} + \tfrac{q_{\max}}{\eta\gamma_{\min}}|\gamma|_{TV((0,1))}
\eqqcolon W_{2,\max}
\end{align*}
However, this term remains bounded uniformly for \(p\in(0,1)\).

Furthermore, we obtain by an intergration by parts in the integral term
\begin{align}
\|\partial_{2}W_{p}[q_{p},\gamma_{p}](t,\cdot)\|_{L^{1}(\R)}&\leq \tfrac{1}{\eta p}q_{\max}^{1-p} \int_{\R}\int_{x}^{x+\eta}\gamma_{p}\big(\tfrac{y-x}{\eta}\big)^{p} \big|\tfrac{\dd}{\dd y}q_{p}(t,y)^{p}\big|\dd y\dd x\label{eq:partial_2_W_L_1_0}
\intertext{and changing the order of integration yields}
&=\tfrac{1}{\eta p}q_{\max}^{1-p} \int_{\R}|\tfrac{\dd}{\dd y}q_{p}(t,y)^{p}|\dd y\int_{y-\eta}^{y}\gamma_{p}\big(\tfrac{y-x}{\eta}\big)^{p}\dd x\dd y=\tfrac{1}{p}q_{\max}^{1-p}|q_{p}(t,\cdot)^{p}|_{TV(\R)}.\label{eq:partial_2_W_L_1_1}
\end{align}
Now to the total variation of \(\partial_{2}W_{p}[q_{p},\gamma_{p}]\). It holds for \((t,x)\in\OT\) that
\begin{align*}
    \partial_{x}^{2}W_{p}[q_{p},\gamma_{p}](t,x)&=(1-p)\tfrac{\big(\partial_{x}W_{p}[q_{p},\gamma_{p}](t,x)\big)^{2}}{W_{p}[q_{p},\gamma_{p}](t,x)} + \tfrac{W_{p}[q_{p},\gamma_{p}](t,x)^{1-p}}{p\eta}\bigg(\gamma_{p}(1)^{p}\tfrac{\dd}{\dd x} q_{p}(t,x+\eta)^{p}-\gamma_{p}(0)^{p}\tfrac{\dd}{\dd x} q_{p}(t,x)^{p}\\
    &\qquad\qquad\qquad\qquad\qquad\qquad\qquad\qquad +\int_{x}^{x+\eta}\tfrac{\dd}{\dd x}\gamma_{p}\big(\tfrac{y-x}{\eta}\big)^{p}\tfrac{\dd}{\dd y}q_{p}(t,y)^{p}\dd y\bigg)
\end{align*}
and thus, together with the previous results,
\begin{align*}
    |\partial_{2}W_{p}[q_{p},\gamma_{p}](t,\cdot)|_{TV(\R)}&\leq (1-p)\bigg\|\tfrac{\big(\partial_{2}W_{p}[q_{p},\gamma_{p}](t,\cdot)\big)^{2}}{W_{p}[q_{p},\gamma_{p}](t,\cdot)}\bigg\|_{L^{1}(\R)}\!\!\!\!\!\!\!\!\!\!+\tfrac{q_{\max}^{1-p}}{p\eta}\big(2\gamma_{p}(0)^{p}+|\gamma_{p}^{p}|_{TV((0,1))}\big)|q_{p}(t,\cdot)^{p}|_{TV(\R)}\\
    &\leq \tfrac{1-p}{q_{\min}}\|\partial_{2}W_{p}[q_{p},\gamma_{p}](t,\cdot)\|_{L^{\infty}(\R)}\|\partial_{2}W_{p}[q_{p},\gamma_{p}](t,\cdot)\|_{L^{1}(\R)}\\
    &\quad+\tfrac{2q_{\max}^{1-p}}{p\eta\gamma_{\min}^{p}}\gamma(0)^{p}|q_{p}(t,\cdot)^{p}|_{TV(\R)}+\tfrac{q_{\max}^{1-p}}{\eta\gamma_{\min}}|\gamma|_{TV((0,1))}|q_{p}(t,\cdot)^{p}|_{TV(\R)}
    \intertext{and, with the previous estimates,}
    &\leq \underbrace{\Big(\tfrac{(1-p)}{pq_{\min} }W_{2,\max}+\tfrac{2}{p\eta\gamma_{\min}^{p}}\gamma(0)^{p}+\tfrac{1}{\eta\gamma_{\min}}|\gamma|_{TV((0,1))} \Big)}_{\eqqcolon \Gamma(p,W_{2,\max})}q_{\max}^{1-p}|q_{p}(t,\cdot)^{p}|_{TV(\R)}.
\end{align*}
As \(\Gamma(p,W_{2,\max})\) is uniformly in \(p\in(0,1)\), we can connect back to \crefrange{eq:TV_estimate_p_0_0}{eq:TV_estimate_p_0_3} and obtain
\begin{align*}
\tfrac{\dd}{\dd t}\!\! \int_{\R}\!\! |\partial_{x}q_{p}(t,x)^{p}|\dd x&\leq  p\|V'\|_{L^{\infty}((q_{\min},q_{\max}))}W_{2,\max}|q_{p}(t,\cdot)^{p}|_{TV(\R)}\\
    &\quad + q_{\max}\|V''\|_{L^{\infty}((q_{\min},q_{\max}))}W_{2,\max}|q_{p}(t,\cdot)^{p}|_{TV(\R)}\\
    &\quad +pq_{\max}\|V'\|_{L^{\infty}((q_{\min},q_{\max}))}\Gamma(p,W_{2,\max})|q_{p}(t,\cdot)^{p}|_{TV(\R)}
\end{align*}
from which it follows with a Gronwall estimate that
\[
|q_{p}(t,\cdot)^{p}|_{TV(\R)}\leq |q_{0}^{p}|_{TV(\R)}\exp\Big(t C(p,V,\eta)\Big),\ t\in[0,T]
\]
with
\begin{align*}
    C(p,V,\eta)&\coloneqq \tfrac{pq_{\max}}{\eta\gamma_{\min}}\|V'\|_{L^{\infty}((q_{\min},q_{\max}))}\Big(q_{\min}^{-\frac{1}{4}}\gamma(0)^{p} + |\gamma|_{TV((0,1))}\Big)\\
    &\quad +\tfrac{q_{\max}^{2}}{\eta\gamma_{\min}}\|V''\|_{L^{\infty}((q_{\min},q_{\max}))}\Big(q_{\min}^{-\frac{1}{4}}\gamma(0)^{p} + |\gamma|_{TV((0,1))}\Big)\\
    &\quad +\tfrac{q_{\max}}{\gamma_{\min}\eta}\|V'\|_{L^{\infty}((q_{\min},q_{\max}))}\Big(\tfrac{(1-p)q_{\max}}{q_{\min}^{\frac{5}{4}} }\gamma(0)^{p} + q_{\max}|\gamma|_{TV((0,1))}+\tfrac{2}{\gamma_{\min}^{p-1}}\gamma(0)^{p}+p|\gamma|_{TV((0,1))} \Big).
\end{align*}
As \(C(p,V,\eta)\) is uniformly bounded in \(p\in(0,1)\), we obtain uniform \(TV\) estimates on \(q^{p}(t,\cdot)\) for any \(t\in[0,T]\) and with those, due to \(q,\ q_{0}\) being bounded away from zero uniformly in \(p\), those \(TV\) bounds also carry over to \(q_{p}\).

Next, we look into the time compactness and derive, for \(t_{1},t_{2}\in[0,T],\ t_{2}\geq t_{1}\), similar to the result in \cref{cor:compactnessWp} on \(W_{\eta}^{p}\)
\begin{align*}
    \int_{\R}|q_{p}(t_{1},x)-q_{p}(t_{2},x)|\dd x&\leq \int_{\R}\int_{t_{1}}^{t_{2}}|\partial_{s}q_{p}(s,x)|\dd s\dd x\\
    &\leq \int_{t_{1}}^{t_{2}}\int_{\R}|V(W_{p}[q_{p},\gamma_{p}](s,x))\partial_{x}q_{p}(s,x)|\dd x\dd s\\
    &\quad +\int_{t_{1}}^{t_{2}}\int_{\R}|V'(W_{p}[q_{p},\gamma_{p}](s,x))\partial_{x}W_{p}[q_{p},\gamma_{p}](s,x)|\dd x\dd s\\
    &\leq \|V\|_{L^{\infty}((q_{\min},q_{\max}))}(t_{2}-t_{1})|q_{p}|_{L^{\infty}((0,T);TV(\R))}\\
    &\quad +\|V'\|_{L^{\infty}((q_{\min},q_{\max}))}(t_{2}-t_{1})|\partial_{2}W_{p}[q_{p},\gamma_{p}]|_{L^{\infty}((0,T);L^{1}(\R))}.\\
    \intertext{Taking advantage of \crefrange{eq:partial_2_W_L_1_0}{eq:partial_2_W_L_1_1}}
    &\leq \|V\|_{L^{\infty}((q_{\min},q_{\max}))}(t_{2}-t_{1})|q_{p}|_{L^{\infty}((0,T);TV(\R))}\\
    &\quad +\|V'\|_{L^{\infty}((q_{\min},q_{\max}))}(t_{2}-t_{1})\tfrac{q_{\max}^{1-p}}{q_{\min}^{1-p}}|q_{p}|_{L^{\infty}((0,T);TV(\R))}.
\end{align*}
As again the right hand side is uniformly bounded in \(p\in(0,1)\), we obtain the equi-integrability in time, and have, thanks to \cite[Lemma 1]{Simon1986},  for any \(\Omega\subset\R\)
\[
\big\{q_{p}\in C([0,T];L^{1}(\Omega)): p\in(0,1)\big\}\xhookrightarrow{\text{c}} C([0,T];L^{1}(\Omega)).
\]
Thus, we know that there exists a subsequence \((p_{k})_{k\in\N}:\ \displaystyle{\lim_{k\rightarrow\infty}} p_{k}=0\) and a \(q_{*}\in C([0,T];L^{1}(\Omega))\) so that
\[
\lim_{k\rightarrow\infty} \|q_{p_{k}}-q_{*}\|_{C([0,T];L^{1}(\Omega))}=0.
\]

In the next step, we establish that the nonlocal operator converges to the ``right'' limit.

To this end, we estimate 
\begin{enumerate}
\item first as follows---writing for the mentioned subsequence just the sequence with an abuse of notation---for \((t,x)\in\OT\) using the mean value theorem in integral form
\begin{align*}
    \big|W_{p}[q_{p},\gamma_{p}](t,x)-W_{p}[q_{*},\gamma_{p}](t,x)\big|
    &=\tfrac{1}{p}\bigg|\int_{a}^{b} s^{\frac{1}{p}-1}\dd s\big|_{a=W_{p}[q_{*},\gamma_{p}](t,x)^{p}}^{b=W_{p}[q_{p},\gamma_{p}](t,x)^{p}}\bigg|,
   \intertext{and, as \(p< 1\), the integrand is monotonically increasing. For this reason, we can perform a ``worst case'' estimate, using that \(W_{p}[q_{p},\gamma_{p}]^{p}\leqq \|q_{0}\|_{L^{\infty}(\R)}^{p}\), and obtain}
    &\leq \tfrac{1}{p} \big(\|q_{0}\|_{L^{\infty}(\R)}^{p}\big)^{\frac{1}{p}-1}\!\!\!\int_{0}^{1}\!\!\gamma_{p}(y)^{p}\big|q_{p}(t,x+\eta y)^{p}-q_{\ast}(t,x+\eta y)^{p}\big|\dd y.
    \intertext{Applying another time the mean value theorem on the function \(x\mapsto  x^{p}\) and recalling the lower bound on the solution}
    &\leq  \tfrac{\|q_{0}\|_{L^{\infty}(\R)}^{1-p}}{q_{\min}^{1-p}}\int_{0}^{1}\!\!\gamma_{p}(y)^{p}\big|q_{p}(t,x+\eta y)-q_{\ast}(t,x+\eta y)\big|\dd y.
\end{align*}
Integrating over \(\tilde{\Omega}\subset \Omega\) so that \(\tilde{\Omega}\subset \Omega + \eta[0,1]\) (this always exists as the \(\eta[0,1]\) is a bounded set), we have, for \(t\in[0,T]\), using the previously derived estimate
\begin{align*}
    \big\|W_{p}[q_{p},\gamma_{p}](t,\cdot)-W_{p}[q_{*},\gamma_{p}](t,\cdot)\big\|_{L^{1}(\tilde{\Omega})}&\leq \tfrac{\|q_{0}\|_{L^{\infty}(\R)}^{1-p}}{q_{\min}^{1-p}}\!\!\int_{\tilde{\Omega}}\int_{0}^{1}\!\!\!\!\gamma_{p}(y)^{p}\big|q_{p}(t,x+\eta y)-q_{\ast}(t,x+\eta y)\big|\dd y \dd x\\
    &\leq \|q_{p}(t,\cdot)-q_{*}(t,\cdot)\|_{L^{1}(\Omega)}\tfrac{\|q_{0}\|_{L^{\infty}(\R)}^{1-p}}{q_{\min}^{1-p}}\int_{0}^{1}\gamma_{p}(y)^{p}\dd y\\
    &= \tfrac{\|q_{0}\|_{L^{\infty}(\R)}^{1-p}}{q_{\min}^{1-p}}\|q_{p}(t,\cdot)-q_{*}(t,\cdot)\|_{L^{1}(\Omega)}.
\end{align*}
Making this uniform in \(t\), we obtain
\begin{equation}
\lim_{p\rightarrow 0} \big\|W_{p}[q_{p},\gamma_{p}]-W_{p}[q_{*},\gamma_{p}]\big\|_{C([0,T];L^{1}(\tilde{\Omega}))}=0.\label{eq:convergence_p=0_solution}
\end{equation}
\item Next, we estimate
\begin{align*}
    \big|W_{p}[q_{*},\gamma_{p}](t,x)-W_{p}[q_{*},\gamma_{0}](t,x)\big|&\leq\|q_{0}\|_{L^{\infty}(\R)} \Bigg|\bigg(\tfrac{1}{\eta}\!\int_{x}^{x+\eta}\!\!\!\!\!\!\!\!\!\!\gamma_{p}^{p}\big(\tfrac{y-x}{\eta}\big)\dd y\bigg)^{\frac{1}{p}}\!\!\!\!-\bigg(\tfrac{1}{\eta}\!\int_{x}^{x+\eta}\!\!\!\!\!\!\!\!\!\!\gamma_{0}^{p}\big(\tfrac{y-x}{\eta}\big)\dd y\bigg)^{\frac{1}{p}}\Bigg|\\
    &=\|q_{0}\|_{L^{\infty}(\R)} \Bigg|\bigg(\tfrac{1}{\|\gamma\|_{L^{p}((0,1))}^{p}}\!\int_{0}^{1}\gamma^{p}(y)\dd y\bigg)^{\frac{1}{p}}\!\!\!\!-\tfrac{\Big(\int_{0}^{1}\gamma(y)^{p}\dd y\Big)^{\frac{1}{p}}}{\|\gamma\|_{L^{0}((0,1))}}\Bigg|\\
    &=\|q_{0}\|_{L^{\infty}(\R)}\Big|1-\tfrac{\|\gamma\|_{L^{p}((0,1))}}{\|\gamma\|_{L^{0}((0,1))}}\Big|\overset{p\rightarrow 0}{\longrightarrow} 0
\end{align*}
applying \cref{lem:L_p_p=0}.
This is uniform in \((t,x)\in\OT\) so that we have
\begin{equation}
  \lim_{p\rightarrow 0}\big\|W_{p}[q_{*},\gamma_{p}]-W_{0}[q_{*},\gamma_{0}]\big\|_{C([0,T];L^{1}_{\text{loc}}(\R))\cap L^{\infty}((0,T);L^{\infty}_{\text{loc}}(\R))}=0.\label{eq:convergence_p=0_nonlocal_weight}
\end{equation}
\item Finally, we look into the convergence of the nonlocal operator, and obtain when adding zeros for \((t,x)\in(0,T)\times \tilde{\Omega}\) and \(p\in(0,1)\)
\begin{align*}
    \big\|W_{p}[q_{p},\gamma_{p}]-W_{0}[q_{*},\gamma_{0}]\big\|_{L^{1}((0,T)\times\tilde{\Omega})}&\leq 
    \big\|W_{p}[q_{p},\gamma_{p}]-W_{p}[q_{*},\gamma_{p}]\big\|_{L^{1}((0,T)\times\tilde{\Omega})}\\
    &\quad +\big\|W_{p}[q_{*},\gamma_{p}]-W_{p}[q_{*},\gamma_{0}]\big\|_{L^{1}((0,T)\times\tilde{\Omega})}\\
    &\quad +\big\|W_{p}[q_{*},\gamma_{0}]-W_{0}[q_{*},\gamma_{0}]\big\|_{L^{1}((0,T)\times\tilde{\Omega}))}.
\end{align*}
The first term in the previous estimate converges for \(p\rightarrow 0\) thanks to \cref{eq:convergence_p=0_solution}, and the second term due to \cref{eq:convergence_p=0_nonlocal_weight}. The third term needs some further justification. According to \cref{lem:L_p_p=0},  we have   
\begin{align*}
W_{\eta,p}[q_{\eta,*},\gamma_{0}](t,x)\overset{p\rightarrow 0}{\longrightarrow} W_{\eta,0}[q_{\eta,*},\gamma_{0}](t,x) \ \text{ for } (t,x)\in (0,T)\times \tilde{\Omega} \text{ a.e.}
\end{align*}
with \(W_{0}\coloneqq W_{\eta,0}\) as in \cref{eq:W_eta_0}. As we have an integrable majorant, i.e., 
\[
|W_{\eta,p}[q_{\eta,*},\gamma_{0}](t,x)| \leq \|q_{0}\|_{L^{\infty}(\R)}\ \forall (t,x,p)\in(0,T)\times\R\times \R_{>0},
\]
we can apply the dominated convergence theorem to obtain
\[
W_{\eta,p}[q_{\eta,*},\gamma_{0}]\overset{p\rightarrow 0}{\longrightarrow} W_{\eta,0}[q_{\eta,*},\gamma_{0}]\ \text{ in } L^{1}((0,T)\times\tilde{\Omega}).
\]
\end{enumerate}
Thus, we have obtained
\[
\lim_{p\rightarrow 0}     \big\|W_{p}[q_{p},\gamma_{p}]-W_{0}[q_{*},\gamma_{0}]\big\|_{L^{1}((0,T)\times\tilde{\Omega})}=0.
\]
It remains to show that the limit \(q_{*}\) with corresponding nonlocal term \(W_{\eta,0}[q_{*},\gamma_{0}]\) is a weak solution of \cref{eq:strong_form_solution_p=0} as defined in \cref{defi:weak_solution_p_equal_zero}.
But this is a direct consequence of the strong convergence of the nonlocal term in \(L^{1}((0,T^{*})\times\tilde{\Omega})\) together with the strong convergence of the solution \(q_{\eta,*}\) when passing to the limit in the weak formulation for \(p\in\R_{>0}\) as in \cref{defi:weak_solution}. 

As solutions for \(p=0\) are unique according to \cref{lem:existence_uniqueness_p=0}, we obtain, by classical arguments, that each sequence \(q_{\eta,p}\) converges strongly in \(L^{1}_{\text{loc}}((0,T)\times\R)\) to the weak solution of the \(p=0\) case which concludes the proof.
\end{proof}

\begin{rem}[The ``admissible'' kernels \(\gamma\)]
    For the convergence result \(p\rightarrow 0\) in \cref{theo:convergence_p=0} we require on \(\gamma\) that \(\gamma\geqq \gamma_{\min}>0\).
    This is for instance satisfied for the constant kernel. It is also satisfied for decreasing kernels which are bounded away from zero on \([0,1]\), so that one can conclude that the choice of kernel is not restrictive.
    
\end{rem}
\begin{rem}[The case \(p=\infty\)]\label{rem:p=infty}
Notice that, formally, we obtain for the nonlocal conservation law in the \(p=\infty\) case the following
\begin{equation}
\begin{aligned}
    \partial_{t}q+\partial_{x}\big(V(W_{\eta,\infty}[q,\gamma](t,x))q\big)&=0,&& (t,x)\in\OT\\
    q(0,x)&=q_{0}(x), && x\in\R,
\end{aligned}
\label{eq:p=infty}
\end{equation}
with the nonlocal operator that coincides with the weighted supremum downstream, i.e.,
\[
  W_{\eta,\infty}[q,\gamma](t,x)\coloneqq \esssup_{y\in\R_{\geq x}}\gamma(\tfrac{x-y}{\eta})q(t,y)=\esssup_{y\in\R_{\geq 0}}\gamma(y)q(t,x+\eta y)=\|\gamma(\cdot) \, q(t,x+\eta \, \cdot)\|_{L^{\infty}(\R_{\geq0})},\ (t,x)\in\OT.
\]
To the authors' knowledge, ``nonlocal'' equations of this type (in the sense that at a given position one potentially takes the density at a different position) have not been studied in literature and it is unclear in which sense solutions might even exist.

One straightforward approach consists of proving  the convergence of the solution of the nonlocal \(p\)-norm model considered in this work converges for \(p\rightarrow\infty\) in a suitable way, and thus showing the existence of solutions. However, at present we were incapable of proving this, mainly due to the fact that weak star convergence of the solution \(q_{p}\) or even strong convergence in \(L^{1}\) do not seem to be sufficient to pass to the limit in the nonlocal term. A rather ``primitive'' convergence attempt would require uniform convergence of \(q_{p}\) to a limit which does not hold in general. 

Another approach for the \(p\rightarrow\infty\) case would be taking monotone initial data and using the monotonicity, uniform in \(p\), in \cref{eq:monotonicity_preserving}, as the equation then reduces to a local equation \cref{eq:p=infty} and the essential supremum is attained either at \(x+\eta\) or \(x\), depending on the type of monotonicity and assuming constant kernel. As this is a quite weak result, we do not detail it further, but only study the convergence numerically.

Finally, one could make an attempt through vanishing viscosity. However, such an approach also only seems to give a uniform \(BV\) bound, which again would not suffice to pass to the limit.

It is also worth mentioning that an entropy condition for uniqueness is surely required. This can be seen once more for a monotonically decreasing initial datum and the constant kernel as the equation then becomes a classical local conservation law which does not exhibit unique weak solutions for monotonically decreasing initial datum and \(V'\leqq 0\) (i.e., the rarefaction requires additional entropy).
\end{rem}
\section{Numerics}\label{sec:numerics}
In this section, we present numerical examples illustrating the proved results in the manuscript. Our simulations are based on the method of characteristics and are as such not smoothed through the numerical method.

\begin{figure}
    \centering
    \includegraphics[scale=0.7]{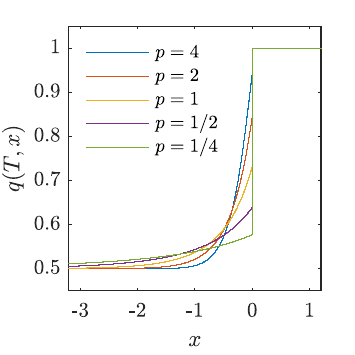}
    \includegraphics[scale=0.7,trim=30 0 0 0,clip=true]{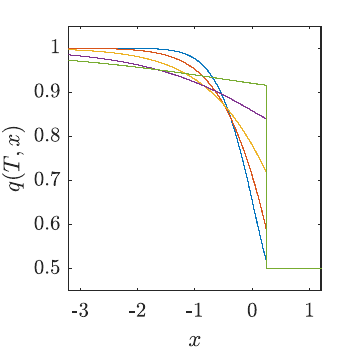}
        \includegraphics[scale=0.7,trim=0 0 0 0,clip=true]{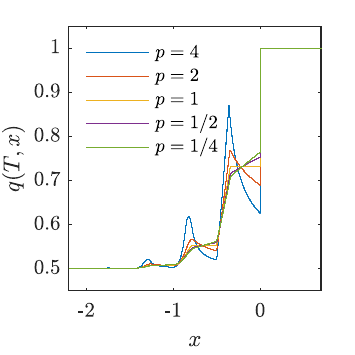}
            \includegraphics[scale=0.7,trim=30 0 0 0,clip=true]{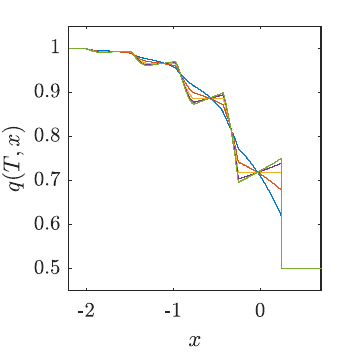}
    \caption{The figures show numerical approximations of solutions to \cref{eq:pnorm_problem} for the exponential kernel (\textbf{first} and \textbf{second}) constant kernel (\textbf{third} and \textbf{fourth}), with $\eta = 0.5$, for various values of $p$ and two choices of initial data: \textbf{first} and \textbf{third}: monotonically increasing, $q_0(x)=0.5+0.5\chi_{\R_{>0}}(x)$; \textbf{second} and \textbf{fourth}: monotonically decreasing, $q_0(x)=0.5+0.5\chi_{(\R_{<0})}(x)$. The first two figures, i.e., the once with the exponential kernel, illustrate the monotonicity preservation for all $p$ (see \cref{eq:monotonicity_preserving}); the third and last figure, i.e., the once with the constant kernel illustrate monotonicity preservation in the increasing case for all $p\in[1,\infty)$, and in the decreasing case for all $p\in (0,1]$.
    }
    \label{fig:const_kernel_vary_p}
\end{figure}

In \cref{fig:const_kernel_vary_p} one can observe the montonicity properties of the nonlocal dynamics in the \(p\)-norm as in \cref{defi:model_class} for different values of \(p\) and a linear velocity \(V(x)=1-x,\ x\in[0,1]\) as stated in \cref{sec:monotonicity} and particularly in \cref{thm:monotonicity_pinfty}. The monotonicity preservation in the case of the exponential kernel can be observed as also the monotonicity break when the range of \(p\) (\((0,1]\) vs.\ \([1,\infty)\)) is not in line with the monotonically decreasing/increasing datum. For further details on the data, we refer to the caption of \cref{fig:const_kernel_vary_p}.

\begin{figure}
    \centering
    \includegraphics[scale=0.7]{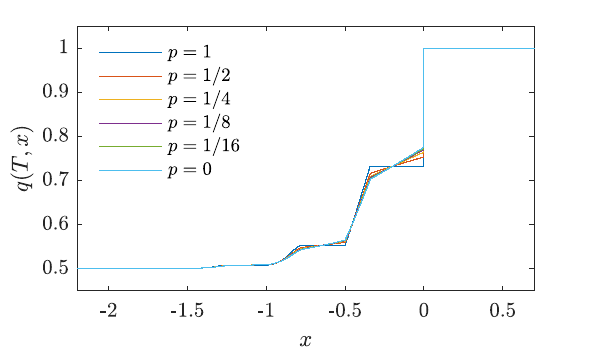}
    \includegraphics[scale=0.7,trim=35 0 0 0,clip=true]{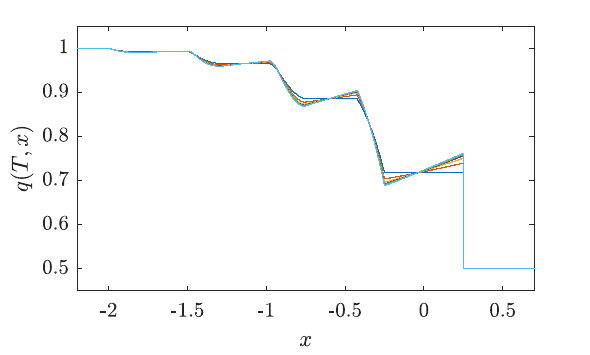}
    \caption{The figures show numerical approximations of solutions to \cref{eq:pnorm_problem} for the constant kernel, i.e., \cref{eq:nonlocal_operator_2} with $\eta = 0.5$, for $p\rightarrow0$ and two choices of initial data: \textbf{left:} monotonically increasing, $q_0(x)=0.5+0.5\chi_{(\R_{>0})}(x)$; \textbf{right:}  monotonically decreasing, $q_0(x)=0.5+0.5\chi_{(\R_{<0})}(x)$. The results clearly show the proven convergence fpr $p\rightarrow 0$ to the equation with nonlinear nonlocality as in \cref{lem:L_p_p=0}.
    }
    \label{fig:const_kernel_p0}
\end{figure}

\Cref{fig:const_kernel_p0} tackles the convergence of the solution to the nonlocal \(p\)-norm conservation law in \cref{defi:model_class} when \(p\rightarrow 0\), i.e., when the limiting solution satisfies the dynamics in \cref{defi:weak_solution_p_equal_zero}. One can not only observe the claimed convergence, but also how such dynamics changes for varying \(p\in(0,1]\). Clearly, for monotone increasing initial datum, the monotonicity is preserved, for monotonically decreasing datum only in the limit \(p=1\).

\begin{figure}
    \centering
    \includegraphics[scale=0.7]{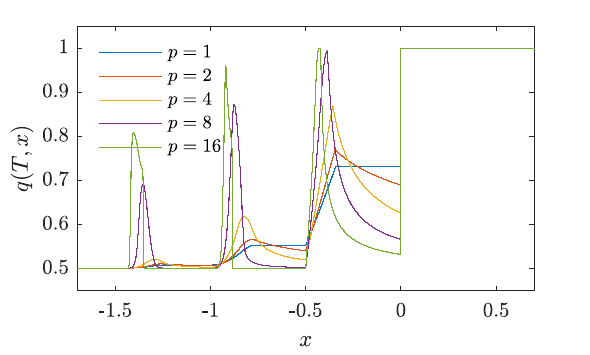}
    \includegraphics[scale=0.7,trim=35 0 0 0,clip=true]{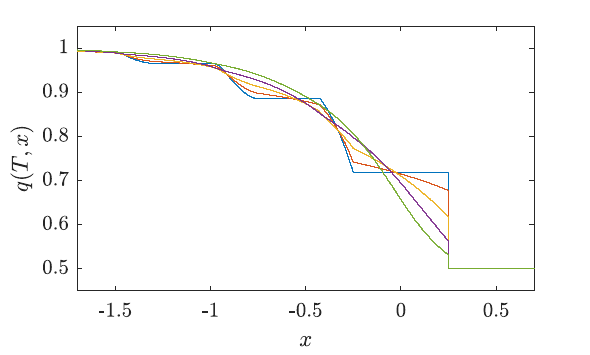}
    \caption{The figures show numerical approximations of solutions to \cref{eq:pnorm_problem} for the constant kernel, i.e., \cref{eq:nonlocal_operator_2} with $\eta = 0.5$, for $p\rightarrow \infty$ at final time $T=0.5$ and two choices of initial data: \textbf{left:} monotonically increasing, $q_0(x)=0.5+0.5\chi_{(\R_{>0})}(x)$; \textbf{right:} monotonically decreasing, $q_0(x)=0.5+0.5\chi_{(\R_{<0})}(x)$. These results show, on the one hand, for mononically inceasing datum
    }
    \label{fig:const_kernel_pinfty}
\end{figure}

In \cref{fig:const_kernel_pinfty}, we analyze numerically the behavior of solutions of \cref{defi:model_class} when \(p\) tends to \(\infty\) as discussed (but not proven) in \cref{rem:p=infty}. Due to the fact that density further downstream is considered ``more'' than in the case of smaller \(p\in\R_{>0},\) one can see some ``phantom shocks'' emerging (\(x\in\{-1.4,-0.9,-0.4\}\)), advertising for the reasonability of the model class as well as interesting new features which cannot be observed for \(p=1\). One can additionally observe that the jump discontinuity stemming from the initial datum is damped slower the higher the \(p\) value (\(x\approx1\)) gets. The monotonically deacreasing case in the right picture is also demonstrated. Interestingly, in contrast to the case described above, the jump discontinuity arising from the initial datum (\(x\approx0.25\)) is dampened out faster the higher $p$ gets. This can be explained as for $p =  \infty$ in the monotne decreasing case, the nonlocal conservation law indeed gets local as $\|q(t,\cdot)\|_{L^\infty((x,x+\eta))}  = q(t,x) \ \forall x \in \R$ and for the local case, the discontinuities vanish for all $t>0$ (recall for instance Oleinik's entropy condition \cite{oleinik_english}.

\begin{figure}
    \centering

    \includegraphics[scale=0.7]{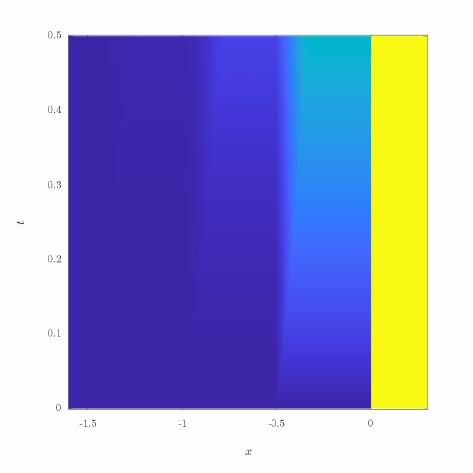}
    \includegraphics[scale=0.7,clip=true,trim=30 0 0 0]{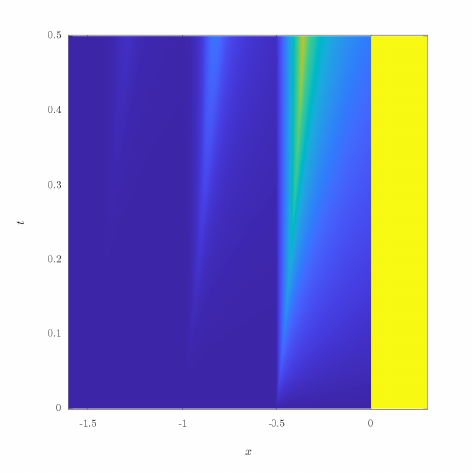}
    \includegraphics[scale=0.7,clip=true,trim=30 0 0 0]{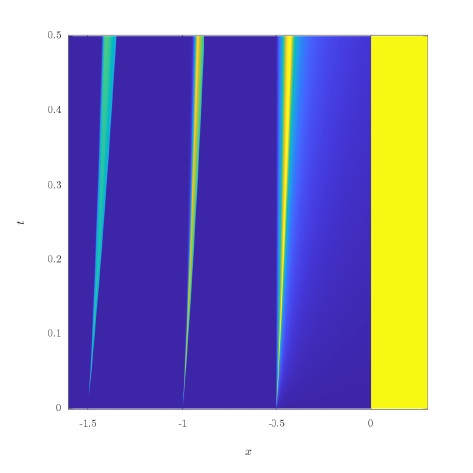}
    
    \includegraphics[scale=0.7]{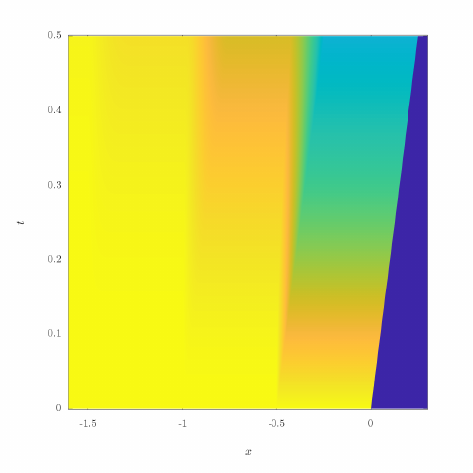}
    \includegraphics[scale=0.7,clip=true,trim=30 0 0 0]{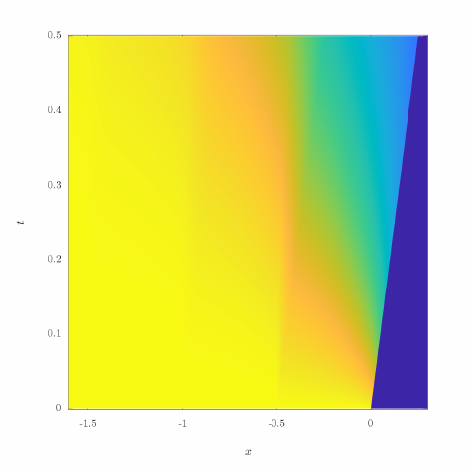}
    \includegraphics[scale=0.7,clip=true,trim=30 0 0 0]{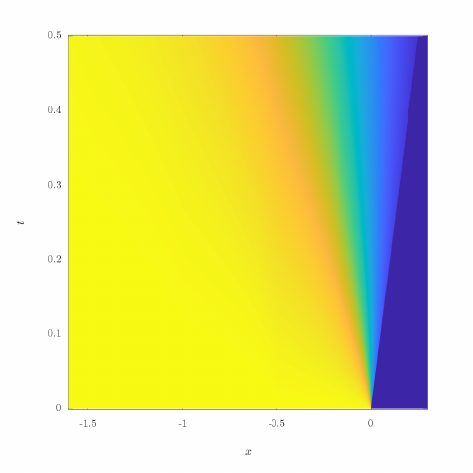}

    \caption{The figures show numerical approximations of solutions to \cref{eq:pnorm_problem} for the constant kernel, i.e., \cref{eq:nonlocal_operator_2} with $\eta = 0.5$, for different $p\in\{1,4,16\}$ over the time-horizon $[0,0.5]$ and two choices of initial data: \textbf{top row:} monotonically increasing, $q_0(x)=0.5+0.5\chi_{(\R_{>0})}(x)$; \textbf{bottom row:} monotonically decreasing, $q_0(x)=0.5+0.5\chi_{(\R_{<0})}(x)$ where $p=1$ in the \textbf{first column}, $p=4$ in the \textbf{second column} and $p=16$ in the \textbf{last column}. }
    \label{fig:3D_const_kernel_vary_p}
\end{figure}

Finally, \cref{fig:3D_const_kernel_vary_p} shows the surface plot over space-time of the problem setup discussed before for \cref{fig:const_kernel_pinfty}. Particularly for monotonically increasing datum, in the considered case a maximal density \(1\) for \(x\geqq 0\), one can observe shockwaves travelling downstream which become intensivied and more localized the larger \(p\) becomes. This observations are similar to the studies in \cite{keimer2019delay} considering time delay in nonlocal conservation laws. Drivers' reaction on traffic jams ahead are either examplified due to a high $p$ in the herein considered model or delayed due to a reaction-time modeled in the nonlocal conservation law with delay. Both model similar kind of super-compensation of the traffic jam ahed and thus to kind of phantom shocks.

\section{Conclusion and open problems}\label{sec:conclusions}
We have established the existence and uniqueness of weak solutions for nonlocal conservation laws with the  \(p\)-norm nonlocality under the restriction on the initial datum (and, as such, the solution) to be bounded away from zero. In addition, for specific not-compactly supported nonlocal kernels, we can get rid of this condition as long as the initial datum is bounded away from zero on a ``substantial part'' of \(\R\). Then, we study the singular limit problem for \(\eta\rightarrow 0\) and obtain in all cases the convergence to the local entropy solution, generalizing the singular limit convergence result currently available in literature significantly. We also look into the convergence \(p\rightarrow 0\) (so that, in the limit, we no longer have a norm) and obtain a very interesting limiting nonlocal equation for which we can prove that the solution for \(p>0\) converges.

The results in this work suggest several new and interesting problems that may be worth future study:

\begin{itemize}
    \item Assume that the initial datum can become zero and we add no further restrictions in contrast to the theory outlined in \cref{sec:existence_zero_initial_datum}. Do we still obtain solutions? Are these unique?
    \item Can we obtain convergence of the solution \(q_{p}\) for \(p\rightarrow\infty\) as commented in \cref{rem:p=infty}?  The resulting conservation law is very different from what has been considered in literature and can be justified also as a reasonable traffic flow model where one would take as density downstream the somewhat most conservative approach breaking as much as possible also if traffic is high further away (taking into account the nonlocal kernel as well, this would be a weighted density). Numerical methods suggest very interesting phanomena worth exploring \cref{sec:numerics} and \cref{fig:3D_const_kernel_vary_p}.
    \item Is the \(p\)-norm of advantage in the singular limit case for multi-D equations as in \cite{spinola}?
    Is there generally an optimal \(p\) given an initial datum?
    Can one quantify the role of \(p\) in the singular limit convergence?
    
    \item Can one obtain a singular limit convergence result for any \(p\in (0,1)\) without further restrictions on kernel (except for being monotone), velocity (except for \(V'\leqq 0\)) and initial datum (except for being nonnegative and \(TV\) bounded) as in \cite{Marconi2023}?
    \item Can one take advantage of the recent \cite{Coclite2025singularlimit} approaching the singular limit convergence by means of compensated compactness? Uniform \(TV\) bounds on the nonlocal operator would then not be required.
\end{itemize}

\bibliographystyle{abbrv}
\bibliography{biblio}

\end{document}